\documentclass[12pt,reqno]{amsart}

\usepackage{graphicx,subfigure}
\usepackage{hyperref}
\hypersetup{colorlinks=true, citecolor=blue, linkcolor=red}

\numberwithin{equation}{section} \numberwithin{figure}{section}
\numberwithin{table}{section} 
{ \theoremstyle{plain}

}

{ \theoremstyle{definition}
\newtheorem{thm}{Theorem}
\newtheorem{cor}{Corollary}
\newtheorem{lem}{Lemma}
\newtheorem{pro}{Proposition}
\newtheorem{rem}{Remark}

\numberwithin{equation}{section} \numberwithin{lem}{section}
\numberwithin{thm}{section} \numberwithin{cor}{section}
\numberwithin{pro}{section} \numberwithin{rem}{section}
}

\begin{document}

\title[Positive solutions of semilinear elliptic equations]{Uniform estimates for positive solutions of a class of semilinear
elliptic equations and related Liouville and one-dimensional
symmetry results}

\author{Christos Sourdis}
\address{Department of Applied Mathematics and Department of Mathematics, University of
Crete, 700 13 Panepistimioupoli Vouton, Crete, Greece.}
\email{csourdis@tem.uoc.gr} \maketitle
\begin{abstract}
We consider the semilinear elliptic equation $\Delta u = W'(u)$ with
Dirichlet boundary conditions in a smooth, possibly unbounded,
domain $\Omega \subset \mathbb{R}^n$. Under suitable assumptions on
the potential $W$, including the double well potential that gives
rise to the Allen-Cahn equation, we deduce a condition on the size
of the domain that implies the existence of a positive solution
satisfying a uniform pointwise estimate. Here, uniform means that
the estimate is independent of $\Omega$. The main advantage of our
approach is that it allows us to remove a restrictive monotonicity
assumption on $W$ that was imposed in the recent paper by G. Fusco,
F. Leonetti and C. Pignotti \cite{fuscoTrans}. In addition, we can
remove a non-degeneracy condition on the global minimum of $W$ that
was assumed in the latter reference. Furthermore, we can generalize
an old result of P. Hess \cite{hess} and D. G. De Figueiredo
\cite{defiguerdo}, concerning semilinear elliptic nonlinear
eigenvalue problems.  Moreover, we study the boundary layer of
global minimizers of the corresponding singular perturbation
problem. For the above applications, our approach is based on a
refinement of a useful result that dates back to P. Cl\'{e}ment and
G. Sweers \cite{clemente}, concerning the behavior of global
minimizers of the associated energy over large balls, subject to
Dirichlet conditions. Combining this refinement with global
bifurcation theory and the celebrated sliding method, we can prove
uniform estimates for solutions away from their nodal set, refining
a lemma from a well known paper of H. Berestycki, L. A. Caffarelli
and L. Nirenberg \cite{cafareliPacard}. In particular, combining our
approach with a-priori estimates that we obtain by blow-up, the
doubling lemma of P. Polacik, P. Quittner, and P. Souplet \cite{pqs}
and known Liouville type theorems, we can give a new proof of a
Liouville type theorem of Y. Du and L. Ma \cite{duMaSqueeeze},
without using boundary blow-up solutions. We can also provide an
alternative proof, and a useful extension, of a Liouville theorem of
H. Berestycki, F. Hamel, and H. Matano \cite{matanoObstacle},
involving the presence of an obstacle. Making use of the latter
extension, we  consider the singular perturbation problem with mixed
boundary conditions. Furthermore, we prove some new one-dimensional
symmetry properties of certain entire solutions to Allen-Cahn type
equations, by exploiting for the first time  an old result of
Caffarelli, Garofalo, and Seg\'{a}la \cite{cafamodica}, and we
suggest a connection with the theory of minimal surfaces. Using this
approach, we can give a new proof of Gibbons' conjecture in two
dimensions which is a weak form of the famous conjecture of De
Giorgi. Furthermore, we provide new proofs of well known symmetry
results in half-spaces with Dirichlet boundary conditions. Moreover,
we can generalize a rigidity result due to A. Farina
\cite{farinaRigidity}. Lastly, we study the one-dimensional symmetry
of solutions in convex cylindrical domains with Neumann boundary
conditions.
\end{abstract}

\tableofcontents

\section{Introduction and statement of the main result}
A problem that has received considerable attention in the literature
is the study of the structure of solutions $(\lambda, u)\in
\mathbb{R}\times C^{2,\alpha}(\bar{\mathcal{D}})$, $0<\alpha<1$,
depending on the nonlinearity $f$, of the semilinear elliptic
nonlinear eigenvalue problem
\begin{equation}\label{eqlions}
\Delta u+\lambda f(u)=0,\ x\in \mathcal{D};\ \ u(x)=0,\ x\in
\partial\mathcal{D},
\end{equation}
where $\mathcal{D}$ is typically a smooth bounded domain. To this
end, the main approaches used  include the method of upper and lower
solutions, bifurcation techniques, as well as topological and
variational methods (see \cite{kormanBook}, \cite{lionsSIAM},
\cite{shiBook}, \cite{SmollerWasserman} and the references therein).

 Recently,   G. Fusco, F. Leonetti and C. Pignotti considered in
\cite{fuscoTrans} the semilinear elliptic problem
\begin{equation}\label{eqEq}
\left\{\begin{array}{ll}
         \Delta u=W'(u), & x\in \Omega, \\
           &   \\
         u=0, & x\in \partial\Omega,
       \end{array}
 \right.
\end{equation}
where $\Omega\subset \mathbb{R}^n$, $n\geq 1$, is a domain with
nonempty Lipschitz boundary (see for instance \cite{evansGariepy}),
under the following assumptions on the $C^2$ function
$W:\mathbb{R}\to \mathbb{R}$, which we will often refer to as a
potential:
\begin{description}
  \item[(a)] There exists a constant $\mu>0$ such that
  \[
0=W(\mu)<W(t),\ t\in [0,\infty),\ t\neq \mu,
  \]
  \[
W(-t)\geq W(t), \ t\in [0,\infty);
  \]
  \item[(b)] $W'(t)\leq 0,\ t\in (0,\mu)$;
  \item[(c)] $W''(\mu)>0$.
\end{description}
A model potential which satisfies the assumptions in
\cite{fuscoTrans} is the double well potential in (\ref{eqAllen})
below, appearing frequently in the mathematical study of phase
transitions, see \cite{delpinoAnnals}. Another, model example is
given in (\ref{eqmodel}). An example of an unbounded domain with
nonempty Lipschitz boundary  is  (\ref{eqOmegaBCN}) below, which was
considered in \cite{cafareliPacard}.
 We stress that, in the case where the domain is unbounded,
the boundary conditions in (\ref{eqEq}) \emph{do not} refer to
$u(x)\to 0$ as $|x|\to \infty$ with $x \in \Omega$. Note that
(\ref{eqlions}) can be related to (\ref{eqEq}) via a simple
rescaling (see the relation between (\ref{eqEV1}) and (\ref{eqEV2})
below).

\emph{\textbf{Some preliminary notation.}} For $x\in \mathbb{R}^n$,
$\rho>0$, we let
\[
B_\rho(x)=\{y\in \mathbb{R}^n\ :\ |y-x|<\rho\},\ \ B_\rho=B_\rho(0),
\]
\[
A+B=\{x+y\ :\ x\in A,\ y\in B \},\ \ A,B\subset\mathbb{R}^n,
\]
and denote by $d(x,E)$ the Euclidean distance of the point $x\in
\mathbb{R}^n$ from the set $E\subset\mathbb{R}^n$, and by $|E|$,
unless specified otherwise, the $n$-dimensional Lebesgue measure of
$E$ (see \cite{evansGariepy}). By $\mathcal{O}(\cdot),\ o(\cdot) $
we will denote the standard Landau's symbols.

The main result of \cite{fuscoTrans} was the following:
\begin{thm}\label{thmfusco}
Assume $\Omega$ and $W$ as above. There are positive constants
$R^*$, $r^*\in (0,R^*)$, $a^*\in (0,\mu)$, $k$, $K$, depending only
on $W$ and $n$, such that if $\Omega$ contains a closed ball of
radius $R^*$, then problem (\ref{eqEq}) has a solution $u \in
C^2(\Omega)\cap C(\bar{\Omega})$ verifying
\begin{equation} \label{eq12} 0<u(x)<\mu,\ \ x\in \Omega,
\end{equation}

\begin{equation}\label{eq14}
\mu-a^*<u(x),\ \ x\in \Omega_{R^*}+B_{r^*},
\end{equation}
and
\begin{equation}\label{eq13}
\mu-u(x)\leq K e^{-kd(x,\partial \Omega)},\ \ x\in \Omega,
\end{equation}
where
\begin{equation}\label{eqomegaR}
\Omega_{R^*}=\{x\in \Omega\ :\ d(x,\partial \Omega)>R^*\}.
\end{equation}
\end{thm}

The approach of \cite{fuscoTrans} to the proof of Theorem
\ref{thmfusco} is variational, involving the construction of various
judicious radial comparison functions, see also \cite{alikakosARMA}.
Although variational, in our opinion, their argument boils down to
the construction of a weak lower solution to (\ref{eqEq}), see
\cite{Berestyckilion}, whose building blocks, after a translation,
are radial solutions of
\begin{equation}\label{eqfuscoTechnics}
\Delta \Phi^r+ c^2(\mu-\Phi^r)=0\  \textrm{in}\ B_r ,\
\Phi^r(r)=\mu-a; \ -\Delta \Psi^{r,R}=0\ \textrm{in}\
B_{r+R}\backslash B_{r},\ \Psi^{r,R}(r)=\mu-a,\ \Psi^{r,R}(r+R)=0,
\end{equation}
where $c^2<W''(t)$, $t\in [\mu-a,\mu]$ (note that assumption
\textbf{(b)} implies that solutions of (\ref{eqEq}) are
super-harmonic). It can be verified that $\Phi'(r)<\Psi'(r)$ for
sufficiently large $r$ and $R$ (having dropped the superscripts for
convenience). So, after a translation, the functions $u$, $v$, and
zero, can be patched together at $|x|=r$ and $|x|=r+R$ to form a
weak lower solution to (\ref{eqEq}), in the sense of
\cite{Berestyckilion}, provided that $\Omega$ contains some large
ball of radius greater than $r+R$. This gives us a solution
satisfying (\ref{eq14}) only in $B_r$ (we use $\mu$ as an upper
solution). However, we may extend the domain of validity, and obtain
the desired bound (\ref{eq14}), by ``sliding around'' that
lower-solution, as in \cite{cafareliPacard}. Using this strategy,
one may considerably simplify the corresponding arguments in
\cite{fuscoTrans}.
 We note that, once (\ref{eq14}) is
established, the proof of the exponential decay estimate
(\ref{eq13}), given in \cite{fuscoTrans}, can be simplified
considerably  by employing Lemma 4.2 in \cite{fife-arma}, making use
of the non-degeneracy condition \textbf{(c)} (the constants in
Theorem \ref{thmfusco} can be chosen so that $W''(t)>0,\ t\in
[\mu-a^*,\mu]$). Moreover, an examination of the proof of Lemma 2.1
in \cite{fuscoTrans} (see Lemma \ref{lemfusco} herein) shows that
assumption \textbf{(a)} above can be relaxed to
\begin{description}
  \item[(a')] There exists a constant $\mu>0$ such that
  \[
0=W(\mu)<W(t),\  t\in [0,\mu),\ \ W(t)\geq 0,\ t\in \mathbb{R},
  \]
  \[
W(-t)\geq W(t), \ t\in [0,\mu] \ \textrm{or}\ W'(t)<0,\ t<0.
  \]
\end{description}
For a typical example of such a potential, see Figure \ref{figW}.
\begin{figure}[htbp]
\begin{center}
\includegraphics[width=4in]{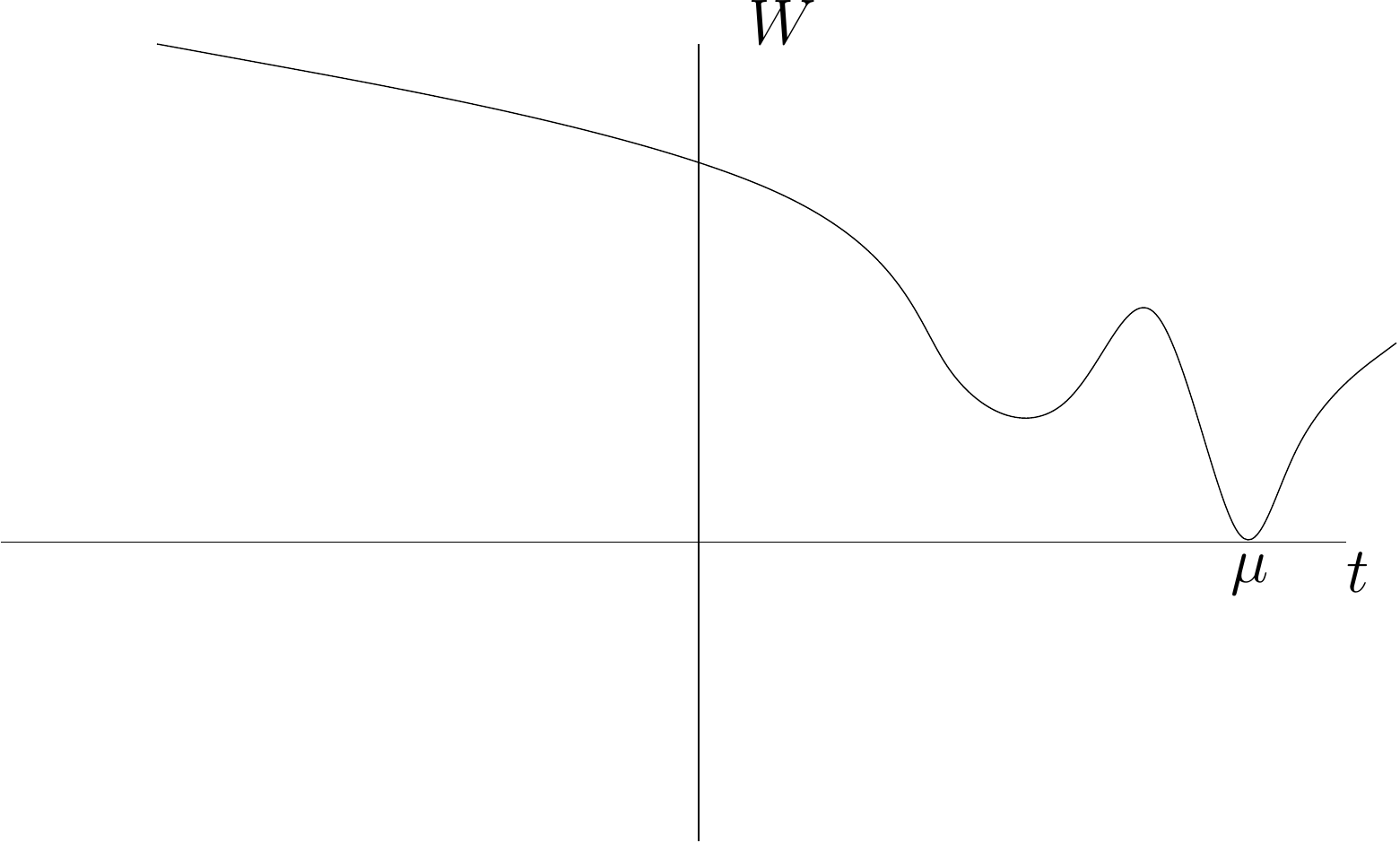}
\caption{An example of a potential $W$ satisfying hypothesis
\textbf{(a')}.}  \label{figW}
\end{center}
\end{figure}

If one further assumes that
\begin{equation}\label{eqpacardW''}
W''(0)<0\ \ \textrm{if}\ \ W'(0)=0,
\end{equation}
and
\begin{equation}\label{eqpacardW'}
W'(t)<0,\ \ t\in (0,\mu),
\end{equation}
then Theorem \ref{thmfusco} can essentially be deduced from Lemmas
3.2--3.3 in the famous article \cite{cafareliPacard} by Berestycki,
Caffarelli and Nirenberg or Lemma 4.1 in the recent article
\cite{kowalczykliupacard} by Pacard, Kowalczyk and Liu, see also
Lemmas 6.1--6.2 in \cite{shiTams},  and \cite{ghosubGui}. In fact,
the latter lemmas hold for arbitrary positive solutions to
(\ref{eqEq}) with values less than $\mu$.

The main purpose of this article is to show that relation
(\ref{eq14}) can be established in a simple manner \emph{without}
assuming the monotonicity condition \textbf{(b)}, and in fact we
will prove a stronger version of it. A well known nonlinearity which
satisfies our assumptions but not \textbf{(b)} is
\begin{equation}\label{eqwei}W'(u)=u(u-a)(u-\mu)\  \textrm{with}\ 0<a<\frac{\mu}{2},
\end{equation}
which arises in the mathematical study of population genetics (see
\cite{aronson}).
 Moreover, we remove completely
the non-degeneracy condition \textbf{(c)} from the proof of
(\ref{eq14}). On the other hand, since an argument of
\cite{fuscoTrans} involving the boundary regularity of weak
solutions to (\ref{eqEq}) when $\partial \Omega$ is arbitrarily
Lipschitz is not clear to us (see the last part of the proof of
Theorem 3.3 therein), we will assume that $\Omega$ has
$C^2$-boundary (to be on the safe side, see however Remarks
\ref{remcone} and \ref{remcafalipschitz} that follow). We will
accomplish the aforementioned improvements, loosely speaking, by
using translations of a positive solution of
\[
\Delta u=W'(u),\ x\in B_R;\ \ u(x)=0,\ x\in \partial B_R,
\]
which minimizes the associated energy, as a lower solution of
(\ref{eqEq}) after we have extended it by zero outside of $B_R$.
Actually, this approach will allow us to refine the results of
\cite{cafareliPacard}, \cite{kowalczykliupacard} that we mentioned
earlier in relation to (\ref{eqpacardW''}), (\ref{eqpacardW'}). On
the other side, assuming further that $W'$ satisfies  a scaling
property and that the corresponding  whole space problem
(\ref{eqentire}) below does not have nontrivial entire solutions (a
Liouville type theorem),
 we will
use ``blow-up'' arguments from \cite{gidasSpruck} together with a
key ``doubling lemma'' from \cite{pqs} to establish that Lemma 3.3
in \cite{cafareliPacard} can be improved.

In passing, we remark that a similar monotonicity assumption to
\textbf{(b)} also appears in a series of papers \cite{alikakosARMA},
\cite{alikakosNewproof}, \cite{alikakosSmyrnelis} in the context of
variational elliptic systems of the form $\Delta u=\nabla_u W(u)$
with $W:\mathbb{R}^n \to \mathbb{R}$. In particular, these
references employ comparison functions of the form
(\ref{eqfuscoTechnics}). In this direction, see also Remarks
\ref{remFusco}, \ref{remalikakossimple} and Appendix \ref{appenAlik}
below.


 Our main result is
\begin{thm}\label{thmmine}
Assume that $\Omega$ is a domain with nonempty boundary of class
$C^2$, and that $W\in C^2$ satisfies \textbf{(a')}. Let $\epsilon\in
(0,\mu)$ and $D>D'$, where $D'$  is determined from the relation
\begin{equation}\label{eqD}
\textbf{U}(D')=\mu-\epsilon,
\end{equation}
where in turn $\textbf{U}$ is the only function in $C^2[0,\infty)$
that satisfies
\begin{equation}\label{eqU}
\textbf{U}''=W'(\textbf{U}),\  s>0;\ \ \textbf{U}(0)=0,\ \lim_{s\to
\infty}\textbf{U}(s)=\mu,
\end{equation}
(see Remark \ref{remU} below). There exists an $R'>D$, depending
only on $\epsilon$, $D$, $W$, and $n$, such that if $\Omega$
contains some closed ball of radius $R'$ then problem (\ref{eqEq})
has a solution $u \in C^2(\Omega)\cap C(\bar{\Omega})$ verifying
(\ref{eq12}), and
\begin{equation}\label{eq14+}
\mu-\epsilon\leq u(x),\ \ x\in \Omega_{R'}+B_{(R'-D)},
\end{equation}
where $\Omega_{R'}$ was previously defined in (\ref{eqomegaR}).
Furthermore, it holds that
\begin{equation}\label{eqcaffaThmmine}
\min\left\{W(t)\ :\ t\in \left[0,u(x)\right] \right\}\leq
\frac{C}{\textrm{dist}(x,\partial \Omega)},\ \ x\in \Omega_{R'},
\end{equation}
for some constant $C>0$ that depends  only on $W,n$.
\\
 If $W''(\mu)>0$ then estimate
(\ref{eq13}) holds true. \\
 If
\begin{equation}\label{eqW''}
W''(t)\geq 0\ \ \textrm{for}\ \ \mu-t>0\ \ \textrm{small},
\end{equation}
then
\begin{equation}\label{eqcafathm}
-W'\left(u(x) \right)\leq
\frac{\tilde{C}}{\left(\textrm{dist}(x,\partial \Omega)
\right)^{2}},\ \ \ x\in \Omega_{R'},\ \ R\geq R',
\end{equation}
for some constant $\tilde{C}>0$ that depends only on $n$, assuming
that $W''\geq 0$ on $[\mu-\epsilon,\mu]$.
\\
If there exist constants $c>0$ and $p>1$ such that
\begin{equation}\label{eqWmonotThm}
-W'(t)\geq c(\mu-t)^p, \ \ t\in [\mu-d,\mu],\ \ \textrm{for\ some\
small}\ d>0,
\end{equation} and $\bar{\Omega}$ is disjoint from the closure of an infinite
open connected cone, or $n=2$ and $\bar{\Omega}\neq \mathbb{R}^2$,
then
\begin{equation}\label{eqfinal}
\mu-u\leq \tilde{K}\textrm{dist}^{-\frac{2}{p-1}}(x,\partial
\Omega),\ \ x\in \Omega,
\end{equation}
for some constant $\tilde{K}>0$ that depends only on  $c$, $p$, $n$
and $W$.
\end{thm}

Estimate (\ref{eqcafathm}) is motivated from \cite{cafareliPacard}.
Condition (\ref{eqWmonotThm}) is in part motivated by some recent
studies \cite{vasilevaDiff,vasileva1-,vasileva2} of a class of
singularly perturbed elliptic boundary value problems of the form
(\ref{eqinhomogeneous}) below in \emph{one space dimension}, where
the degenerate equation $W(u,x)=0$ has a root $u=u_0(x)$ of finite
multiplicity.

The method of our proof is quite flexible, and we came up with a
variety of  applications to related problems that can be found in
the following sections and the included remarks (see the outline at
the end of this section).
  As
will be apparent from the proof, see in particular the comments
leading to Proposition \ref{prolayer} below, a delicacy of our
result is that the constant $D'$ is independent of $n$.

\begin{rem}\label{remU}
The existence and uniqueness of such a solution $\textbf{U}$ of the
ordinary differential equation $u''=W'(u)$ follows readily from
\textbf{(a')} by phase plane analysis, using the fact that the
latter equation has the conserved quantity
$e(s)=\frac{1}{2}(u')^2-W(u)$, see for instance Lemma 3.2 in
\cite{ambrosio3D}, Chapter 2 in \cite{arnold} or page 135 in
\cite{walter} (for a more analytic approach, we refer to
\cite{berestyckiLionsARMA1} or \cite{bouhours}). We note that
\begin{equation}\label{eqUincr}
\textbf{U}'(s)>0,\ \ s\geq 0.
\end{equation}
\end{rem}


\begin{rem}\label{remRobin}
Similar  assertions hold for the Robin boundary value problem:
\[
\Delta u=W'(u),\ x\in \Omega;\ \ \frac{\partial u}{\partial
\nu}+b(x)u=0,\ x\in \partial\Omega,
\]
where $\nu$ denotes the outward unit normal vector to the boundary
of $\Omega$, assuming here that the latter is at least $C^1$, with
$b\in C^{1+\alpha}(\partial \Omega),\ \alpha>0$, being nonnegative
(so that the constant $\mu$ is a positive upper solution, see
\cite{sattinger}). Moreover, as in \cite{fuscoTrans}, we can study
some problems with mixed boundary conditions (see also Section
\ref{secmixed}).
\end{rem}

\begin{rem}\label{remuniq}
A sufficient, and easy to check, condition for the uniqueness of a
positive solution of (\ref{eqEq}), in \emph{any smooth bounded
domain}, is
\begin{equation}\label{eqbrezisOz}
\frac{W'(t)}{t}\ \ \textrm{being\ strictly\ increasing\ in}\ \
(0,\infty).\end{equation} This uniqueness result is originally due
to Krasnoselski, see \cite{brezis-oswald} (see also Proposition 3.5
in \cite{ni} or Theorem 1.16 in \cite{radulescuBook} for a different
proof, and Theorem 3 in \cite{smollerCMP} for a radially symmetric
proof).
 The above condition is clearly satisfied
by the model double well potential in (\ref{eqAllen}) below.
 Related
conditions can be found in \cite{tertikas}. In certain cases, these
type of conditions imply uniqueness of a positive solution in
unbounded domains as well, see for example \cite{saddlecabre3solo}
and \cite{fifesaddle} for uniqueness of the so-called saddle
solutions that we will discuss shortly. Another sufficient
condition, which on the other hand depends partly on the smooth
bounded domain $\Omega$, is
\[
W''(t)\geq -\lambda,\ \ t\geq 0,
\]
for some $\lambda<\lambda_1$, where $\lambda_1>0$ denotes the
principal eigenvalue of $-\Delta$ in $W^{1,2}_0(\Omega)$ (see
\cite{amann-uniq}, \cite{serrin}). This condition is clearly
satisfied, with $\lambda=0$, by the convex model potential in
(\ref{eqmodel}) below.
\end{rem}

Let us mention that for a class  of potentials, including
(\ref{eqAllen}), the dependence of the set of solutions of
(\ref{eqEq}), in one space dimension, on the size of the interval
was studied  for the first time in \cite{chafee} (see also the more
up to date reference \cite{chicone}).

In our opinion, Theorems \ref{thmfusco} and \ref{thmmine} are
important for the following reasons. If we additionally assume that
$W$ is even, namely
\begin{equation}\label{eqWeven}W(-t) = W(t),\ \ \  t\in\mathbb{R},\end{equation} by means of
these theorems, we can derive the existence of various sign-changing
entire solutions for the problem
\begin{equation}\label{eqentire}
\Delta u=W'(u),\ \ \ x\in \mathbb{R}^n.
\end{equation}
This can be done by first establishing  existence of a positive
solution in a suitable large ``fundamental'' domain $\Omega_F\subset
\mathbb{R}^n$, with Dirichlet boundary conditions on $\partial
\Omega_F$, and then performing consecutive odd reflections to cover
the entire space.

\begin{rem}\label{remcone}
The boundary of the fundamental domain $\Omega_F$ may have corner or
conical points. But we can round them off, approximating $\Omega_F$
by a sequence of expanding \emph{smooth} domains $\Omega_j$ (where
Theorem \ref{thmmine} is applicable). Then, we can obtain the
desired solution in $\Omega_F$ by letting $j\to \infty$ along a
subsequence (see \cite{fifesaddle}). In this regard, see also Remark
\ref{remcafalipschitz} below.

The fact that, after reflecting, we obtain a classical solution can
be shown by a standard capacity argument (see Theorem 1.4 in
\cite{cabre}).
\end{rem}

  In the case where
\begin{equation}\label{eqAllen} W(t)=\frac{1}{4}(t^2-1)^2,\ \ t\in
\mathbb{R},
\end{equation}
then  (\ref{eqentire}) becomes the well known  Allen-Cahn equation
(see for instance \cite{pacard}). Assuming that $W$ is even, namely
that (\ref{eqWeven}) holds true, then (\ref{eqEq}) has always the
trivial solution.
  In this regard, the purpose of estimate (\ref{eq14+}) is
twofold: In the case where $\Omega_F$ is bounded,  it ensures that
the solution of (\ref{eqEq}) (on $\Omega_F$), provided by Theorem
\ref{thmmine}, is nontrivial. The situation of unbounded domains
$\Omega_F$ can be treated by exhausting them by an increasing (with
respect to inclusions) sequence $\{\Omega_j\}$ of bounded ones, each
containing the same ball $B_{R'}(x_0)$, and a standard compactness
argument, making use of (\ref{eq12}) together with elliptic
estimates and a Cantor type diagonal argument. The fact that the
region of validity of estimate (\ref{eq14+}) increases, as $j\to
\infty$, rules out the possibility of subsequences of the (chosen)
solutions $u_j$ of $(\ref{eqEq})_j$ on $\Omega_j$ converging,
uniformly in compact subsets of $\Omega_F$, to the trivial solution
of (\ref{eqEq}) on $\Omega_F$. Another approach for excluding this
last scenario, which however does not seem to provide uniform
estimates directly, can be found in the proof of Theorem 1.3 in
\cite{cabre}, based on a similar relation to (\ref{eqcabre}) below
(see also \cite{delpinoScrew} and \cite{pacard}). In this fashion,
and under more general assumptions on $W$ than previous studies
(conditions \textbf{(a')} and (\ref{eqWeven}) suffice for most
applications), one can construct a whole gallery of nontrivial
sign-changing solutions of (\ref{eqentire}) that includes
\begin{itemize}
  \item ``saddle
solutions''  which vanish on the Simons cone $\{(x,y)\in
\mathbb{R}^{2m}\ :\ |x|=|y|\}\subset \mathbb{R}^{2m}=\mathbb{R}^n$
 if $n$ is even (see
\cite{cabre}, \cite{saddleCabreCPDE2}, \cite{saddlecabre3solo},
\cite{fifesaddle}, \cite{gui}, and \cite{pacard}).
In fact, they can be constructed in the \emph{block-radial} class,
namely $u(x,y)=\mathrm{u}(|x|,|y|)=-\mathrm{u}(|y|,|x|)$. In
passing, we note that solutions with these symmetries have been
studied for nonlinear Schr\"{o}dinger type equations, say
(\ref{eqentire}) with $W'(t)=t-t^3$,  in Chapter 3 in \cite{kuzin},
 Section 1.6 in \cite{willem}, and the references therein (for such solutions to the  Gross-Pitaevskii equation with radial trapping potential,
 we refer to Section 6 in \cite{karaliBoseMine}).
 Estimate
(\ref{eq13}) implies that the corresponding saddle solution
converges to $\pm \mu$ exponentially fast, as the signed distance
from the Simons cone tends to plus/minus  infinity respectively.
 Analogous
solutions exist in odd dimensions, for example when $n=3$ it was
shown in \cite{alesioNodea} that there exists a solution which
vanishes on all coordinate planes (see also a related discussion in
\cite{delpinoDiffGeom}).
 In dimension $n = 2$, solutions whose
zero level set has the symmetry of a regular $2k$--polygon and
consists of $k$ straight lines passing through the origin were found
in \cite{saddlealessio} (in the case where $W$ is periodic, similar
solutions but with polynomial growth were found,  following this
strategy, recently in \cite{weiMoser}); such solutions can
appropriately be named ``pizza solutions'', see also
\cite{shiPizza}. Denote $G$ the rotation of order $2k$, and note
that these solutions satisfy $u(Gx) =-u(x)$, $x\in \mathbb{R}^2$.
Another method to get $u$ is to find a minimizing solution $u_R$ of
the equation in the invariant class $\{u\in W^{1,2}_0(B_R) \
\textrm{and}\ u(Gx) =-u(x), \ x\in B_R\}$. The minimizer $u_R$ can
be proved to satisfy (\ref{eqEq}) in $B_R$ by the heat flow method
(see \cite{alikakosARMA}, \cite{alikakosBasicFacts},
\cite{alikakosSmyrnelis}, \cite{batesFuscoSmyrnelis},
\cite{weiEntire}). Note that because $W$ is even, the invariant
class is positively invariant by the heat flow.
  \item ``lattice solutions'' which include solutions that  are periodic in each variable $x_i$ with
period $L_i$, provided that $L_i$, $i=1,\cdots,n$, are sufficiently
large (see \cite{alikakosSmyrnelis},
\cite{batesFuscoSmyrnelisIIIII}, \cite{fifeCheckerboard},
\cite{kielhoferBook}, and \cite{mironescPeriodic}). This type of
solutions, which can be described as having lamellar phase, were
recently conjectured to exist in Chapter 4 of \cite{sigal}. Another
example, which is motivated from \cite{maierJDE}, are solutions in
the plane whose nodal domains consist of sufficiently large
(identical modulo translation and rotation) equilateral triangles
tiling the plane (in relation to this, see also Remark
\ref{remconvex} below). Under some additional hypotheses on $W$,
planar lattice solutions can be constructed
 by local and global bifurcation techniques (see
 \cite{fifeCheckerboard}, \cite{kielhoffer}, \cite{kielhoferBook}, and \cite{maierJDE}).
\item
``tick saddle solutions'' which have saddle (or pizza) structure in
some coordinates while they are periodic in the remaining ones (see
the introduction in \cite{fuscoTrans}). For example, in
$\mathbb{R}^2$, these solutions are odd
 with respect to both $x$ and $y$, having as nodal curves the
lines $x=0$ and $y=kL$, $k\in \mathbb{Z}$, for $L$ sufficiently
large (so that the fundamental domain $\Omega_{F,L}\equiv\{x>0,\
y\in (0,L)\}$ contains a sufficiently large closed ball). In fact,
if $W''(0)<0$ and (\ref{eqdelpinobelowRc}) below hold, by modifying
the approach of the current paper and using some ideas from
Proposition 3.1 in \cite{delpinoScrew} (which dealt with a problem
of similar nature on an infinite half strip, see also Remark
\ref{remscrew2} below), it is plausible that there exists an
explicit constant $L^*>0$  such that (\ref{eqEq}) considered in
$\Omega_{F,L}$ has a positive solution if and only if $L> L^*$ (see
also Remark \ref{remscrew} below); a similar construction should
also work in higher dimensions. We note that tick saddle solutions
can be constructed as limits of appropriate lattice solutions by
letting some of the periods tend to infinity (along a subsequence),
see \cite{alikakosSmyrnelis}. In the case where $W$ is as in
(\ref{eqAllen}), and $n=2$, the spectrum of the linearized operator
about the saddle solution of \cite{fifesaddle} has a unique negative
eigenvalue (see \cite{schatman}).  Moreover, it has been shown
recently that the saddle solution is non-degenerate, namely there
are no decaying elements in the kernel of the linearized operator
(see \cite{saddleKowal}). In view of these two properties it might
also be possible to construct tick saddle solutions in
$\mathbb{R}^3$, with $W$ as in (\ref{eqAllen}), by local bifurcation
techniques (for example, by the ideas in \cite{dancerPiza}). Lastly,
let us point out that the shape of their zero set bears some
qualitative similarities to Sherk's singly periodic minimal surface
(see for example \cite{coldingCour}).
\item ``Screw--motion invariant solutions'' whose nodal set is a helico\"{i}d
of $\mathbb{R}^3$, or analogous minimal surfaces in any odd
dimension (see \cite{delpinoScrew} and Remark \ref{remscrew2}
herein).
\end{itemize}

 A completely different approach to the
construction of sign-changing solutions of (\ref{eqentire}), mainly
applied for potentials  satisfying \textbf{(a)}, \textbf{(c)}, and
(\ref{eqWeven}) (the typical representative being (\ref{eqAllen})),
is based on the implementation of an infinite dimensional
Lyapunov--Schmidt reduction argument, see \cite{delpinoAnnals},
\cite{delpinoScrew}, \cite{delpinoDiffGeom}, \cite{pacard}, and the
references therein. This approach produces solutions with less (or
even without any) symmetry but is technically more involved.

Our Theorem \ref{thmmine} can also be used to construct multiple
positive solutions of (\ref{eqEq}), using estimate (\ref{eq14+}) to
make sure that they are distinct, see Section \ref{secExtensions}
below.

\subsection{Outline of the paper} The outline of the paper is as follows: In
Section \ref{secproof}, we will present the proof of our main
result, with the exception of (\ref{eqfinal}), by using two
different approaches, both based on a special case of a radial lemma
that we prove in Subsection \ref{subsectionBalls}. In the remainder
of the paper we will exploit further  this radial lemma and use it
as a basis to prove interesting results. In Section \ref{secpacard},
we prove uniform lower bounds for arbitrary positive solutions. In
Section \ref{secdoubling}, we prove universal decay estimates for
solutions, in the case where $W$ is a model power nonlinearity
potential, thereby generalizing the exponential decay estimate
(\ref{eq13}) by an algebraic one and relating the obtained result to
a corresponding one in \cite{cafareliPacard}. Moreover, this
algebraic decay estimate allows us to show (\ref{eqfinal}) and thus
complete the proof of Theorem \ref{thmmine}. In Section
\ref{secBoundsonEntire}, under appropriate conditions on $W$, we
will show that all entire solutions of (\ref{eqentire}) are
uniformly bounded; combining this with the main result of Section
\ref{secpacard}, we can give a short self-contained proof of the
main result in the paper of Du and Ma \cite{duMaSqueeeze}.
   In Section
\ref{secmatano}, we prove nonexistence results for nonconstant
solutions with Neumann boundary conditions that are motivated by
some Liouville type result of Berestycki, Hamel and Matano
\cite{berestyckiLiouville} (for which we provide simplified proofs,
while at the same time removing a technical assumption). In Section
\ref{secExtensions}, we will show how our Theorem \ref{thmmine} can
be used to produce multiple positive solutions of (\ref{eqEq}) and
thus generalize an old result of P. Hess from 1981, where nonlinear
eigenvalue problems were considered. In Section \ref{seclayer}, we
study the size of the boundary layer of global minimizers of the
corresponding singular perturbation problem, in the context of
nonlinear eigenvalue problems. In Section \ref{secmixed}, we will
study the corresponding problem with mixed boundary conditions. In
Section \ref{secfarina}, we will prove some new one-dimensional
symmetry results for certain entire solutions to (\ref{eqentire}),
by exploiting for the first time  an old result of Caffarelli,
Garofalo, and Seg\'{a}la \cite{cafamodica}, and we suggest a
connection with the theory of minimal surfaces. Exploiting this
approach, and the Hamiltonian structure of the equation, we can give
a new proof of Gibbons' conjecture in two dimensions, originally
proven (in all dimensions) independently by
\cite{barlow,berestDuke,farinaRendi} (see also
\cite{cafaCordobaIntermed}). This conjecture is a weak form of the
famous conjecture of De Giorgi. Furthermore, we are able to provide
new proofs of well known symmetry results in half-spaces with
Dirichlet boundary conditions. Moreover, we generalize a rigidity
result of \cite{farinaRigidity}. Finally, in Section
\ref{seccylindric}, we study the one-dimensional symmetry of
solutions in convex cylindrical domains with Neumann boundary
conditions.
In Appendix \ref{secappenda}, for completeness purposes, we will
state some useful comparison  lemmas that we will use in this
article. In Appendix \ref{appenLiouville}, for the reader's
convenience, we will state a useful Liouville type theorem of
\cite{farinaHB} which extends a result of \cite{brezisLiouville}. In
Appendix \ref{secAppenDoubling}, for the reader's convenience, we
will state the useful doubling lemma of \cite{pqs} that we mentioned
earlier. In Appendix \ref{appenAlik}, we make some remarks that are
motivated from the recent paper \cite{alikakosNewproof}, dealing
with uniform estimates for equivariant entire solutions to an
elliptic system under assumptions that are analogous to those in
\cite{fuscoTrans}.

\begin{rem}\label{remFusco}
We recently found the paper \cite{alikakosSmyrnelis}, where it is
stated that G. Fusco, in work in progress (now published, see
\cite{fuscoPreprint}, and also \cite{fuscoCPAA}), has been able to
remove the corresponding monotonicity assumption to \textbf{(b)}
from the vector-valued Allen-Cahn type equation that was treated in
\cite{alikakosARMA}. After the first version of the current paper
was completed, we were informed by G. Fusco that himself, F.
Leonetti and C. Pignotti are working in a paper where, using the
same technique developed for the vector case, they are in the
process of  extending the main result in \cite{fuscoTrans} to more
general potentials without assuming \textbf{(b)}. Their approach is
certainly more elaborate but  it is entirely self-contained, while
we use in a simple and coordinate way various deep well known
results.
\end{rem}

\section{Proof  of the main result}\label{secproof}
\subsection{Minimizers of the energy functional on large
balls}\label{subsectionBalls} In this subsection, we will mainly
prove two lemmas concerning  the asymptotic behavior of the
minimizing (of the associated energy) solutions of (\ref{eqEq}) over
large balls as their radius tends to infinity. The first one is
essential for the proof of Theorem \ref{thmmine}, and refines a
result of P. Cl\'{e}ment and G. Sweers \cite{clemente}. The latter
result is quite useful, and has been previously applied in singular
perturbation problems (see \cite{danceryanCVPDE}, \cite{kurata}, and
\cite{yanedinburg}).
The second lemma, an extension of the first, is of independent
interest and in particular  allows for $W'(0)$ to be positive. Even
though the first lemma is a special case of the second, we felt that
it would be more instructive and more convenient for the reader to
present them separately, since the more general second lemma is not
needed for the proof of Theorem \ref{thmmine} and can be skipped at
first reading.

The following is our first lemma, which is motivated from Lemma 2 in
\cite{kurata} and Lemma 2.2 in \cite{yanedinburg} (see also Lemma
2.4 in \cite{efediev}), whose origins can be traced back to
\cite{clemente-peletier,clemente}. In these works, the weaker
relation (\ref{eqsweers}) below was established, which implies that
assertion (\ref{eqestimRect}) holds \emph{but} with constant $D$
possibly \emph{diverging} as $n\to \infty$ (see also Remark
\ref{remn} below).
 Our
improvement turns out to have interesting consequences in the study
of the boundary layer of solutions of singular perturbation problems
of the form (\ref{eqEV1}) below, with $\lambda=\varepsilon^{-1}\to
\infty$, see Remark \ref{remlayer} and Section \ref{seclayer} below.
Moreover, estimate (\ref{eqestimRect}) will be used in a crucial way
in Proposition \ref{proUniq} for studying the asymptotic stability
of minimizing solutions that are provided by the following lemma or
the more general  Lemma \ref{lem1Sign} below.
\begin{lem}\label{lem1}
Assume that $W\in C^2$ satisfies condition \textbf{(a')}. Let
$\epsilon\in (0,\mu)$ and $D>D'$, where $D'$ is as in (\ref{eqD}).
There  exists a positive constant $R'>D$, depending only on
$\epsilon$, $D$, $W$ and $n$, such that there exists a global
minimizer $u_R$ of the energy functional
\begin{equation}\label{eqenergy}
J(v;B_R)=\int_{B_R}^{}\left\{\frac{1}{2}|\nabla v|^2+W(v)
\right\}dx,\ \ \ v\in W^{1,2}_0(B_R),\end{equation} which satisfies
\begin{equation}\label{eqmaxball}
0<u_R(x)<\mu,\ x\in B_R,
\end{equation}
and
\begin{equation}\label{eqestimRect}
\mu-\epsilon \leq u_R(x),\ x\in \bar{B}_{(R-D)},
\end{equation}
provided that $R\geq R'$. Moreover, there exists a constant $C$
depending only on $W,n$ such that
\begin{equation}\label{eqcaffaLemmmine}
\min\left\{W(t)\ :\ t\in \left[0,u_R(r)\right] \right\}\leq
\frac{C}{R-r},\ \ r\in [0,R),\ \forall\  R\geq R'.
\end{equation}
 (If necessary, we assume that $W$ is
extended linearly outside of a large compact interval so that the
above functional is well defined (see also Lemma 2.4 in
\cite{efediev}); clearly this modification does not affect the
assertions of the lemma).
\end{lem}
\begin{proof}
Under our assumptions on $W$, it is standard to show the existence
of a global  minimizer $u_R\in W^{1,2}_0(B_R)$ satisfying
\begin{equation}\label{eqmax} 0\leq u_R(x)\leq \mu\ \  \textrm{a.e.\  in}\   B_R,\end{equation}
see Chapter 2 in \cite{badiale}, \cite{fuscoTrans},
\cite{rabinowitz}, and Lemma \ref{lemfusco} herein (applied to the
minimizing sequence converging, weakly in $W_0^{1,2}(B_R)$, to
$u_R$). (The upper bound in (\ref{eqmax}) can also be derived from
Lemma \ref{lemdancer} below, see also the second proof of Theorem
\ref{thmmine}). By standard elliptic regularity theory
\cite{Gilbarg-Trudinger}, this minimizer is a smooth solution, in
$C^2(\bar{B}_R)$, of
\begin{equation}\label{eqgidasEq}
\Delta u=W'(u)\ \ \textrm{in}\ \ B_R;\ u=0\ \ \textrm{on}\ \
\partial B_R.
\end{equation}
By  the strong maximum principle (see for example Lemma 3.4 in
\cite{Gilbarg-Trudinger}), via (\ref{eqmax}) and (\ref{eqgidasEq}),
we deduce that $u_R(x)<\mu,\ x\in B_R$, and that either $u_R$ is
identically equal to zero or $u_R(x)>0,\ x\in B_R$ (recall that
assumption \textbf{(a')} implies that $W'(0)\leq 0$ and
$W'(\mu)=0$).

By adapting an argument from Section 4 in \cite{pacard} (see also
Lemma 5.3 in \cite{guiAnnals} and Theorem 1.13 in
\cite{rabinowitz}), we will show that $u_R$ is nontrivial, provided
that $R$ is sufficiently large (depending only on $W$ and $n$).
(This is certainly the case when $W'(0)<0$).
 It is easy to cook up a test function, and use it as a competitor, to show that there exists a
positive constant $C_1$, depending only on $W$ and $n$, such that
\begin{equation}\label{eqJcomp} J(u_R; B_R)\leq C_1 R^{n-1},\ \ \textrm{say\ for}\
\ R\geq 2.
\end{equation}
(Plainly construct a function which interpolates smoothly from $\mu$
to $0$ in a layer of size $1$ around the boundary of $B_R$ and which
is identically equal to $\mu$ elsewhere, see also (\ref{eqlinear})
below or Lemma 1 in \cite{cafaCordoba}). In fact, as in Proposition
1 in \cite{alama} (see also \cite{lassouedmironescu}), it can be
shown that
\begin{equation}\label{eqK}
J(u_R; B_K)\leq \tilde{C}_1 K^{n-1}\ \ \forall\ K<R,\ \ R\geq 2,
\end{equation}
where the constant $\tilde{C}_1>0$ depends only on $W$ and $n$ (see
also Remark \ref{remMonoto}, and  the arguments leading to relation
(\ref{eqcabre}) below).
 On
the other hand, the energy of the trivial solution is
\[J(0;B_R)= \int_{B_R}^{}W(0) dx= C_2 R^n,
\] where $C_2>0$ depends only on $W,\ n$.
From (\ref{eqJcomp}), and the above relation,
 we
infer that $u_R$ is certainly not identically equal to zero for
\[R\geq C_1 C_2^{-1}+2.\] We thus conclude that (\ref{eqmaxball})
holds. (In the above calculation, we relied on the fact that
\textbf{(a')} implies that $W(0)>0$; in this regard, see Remark
\ref{rempoho} below).

 Since $u_R \in C^2(\bar{B}_R)$ is strictly positive
in the ball $B_R$,  by  (\ref{eqgidasEq}) and the method of moving
planes \cite{brezis,dancePlanes,gidas}, we infer that $u_R$ is
radially symmetric and decreasing, namely
\begin{equation}\label{eqmonotonicity}
u_R'(r)<0,\ \ r\in(0,R),
\end{equation}
(with the obvious notation). In this regard, keep in mind that if
$v\in W^{1,2}_0(B_R)$ is nonnegative, then its Schwarz
symmetrization $v^*\in W^{1,2}_0(B_R)$, which is radially symmetric
and decreasing, satisfies $J(v^*;B_R)\leq J(v;B_R)$ (see for example
\cite{brock} and the references therein). We note that, since $u_R$
is a global minimizer and thus stable (in the usual sense, as
described in Remark \ref{remstable} below), the radial symmetry of
$u_R$, for $n\geq 2$, can also be deduced as in Lemma 1.1 in
\cite{alikakosbates} (see also the related references in the proof
of Lemma \ref{lem1Sign} below). In fact, the monotonicity property
(\ref{eqmonotonicity}) can be alternatively  derived by arguing  as
in Lemma 2 in \cite{cabreRadial} (see also Proposition 1.3.4 in
\cite{dupaigneBook}), making use of the stability of the radial
solution $u_R$ (see also the proof of Lemma \ref{lem1Sign} below,
and Lemma 1 in \cite{alikakosSimpson}). Now, relation
(\ref{eqJcomp}) and the nonnegativity of $W$ clearly imply that
\begin{equation}\label{eqcoarea}
\int_{{B_R}\backslash B_{\frac{R}{2}}}^{}\left\{\frac{1}{2}|\nabla
u_R|^2+W(u_R)\right\}dx\leq C_1 R^{n-1},\ \ R\geq C_1
C_2^{-1}+2.\end{equation} Hence, by the mean value theorem and the
radial symmetry of $u_R$, there exists a $\xi \in
\left(\frac{R}{2},R\right)$ such that
\[
\left\{\frac{1}{2}[u_R'(\xi)]^2+W\left(u_R(\xi)\right)\right\}\left|{B_R}\backslash
B_{\frac{R}{2}} \right|\leq C_1 R^{n-1},\ \ R\geq C_1 C_2^{-1}+2,
\]
i.e.,
\begin{equation}\label{eqgood}
\frac{1}{2}[u_R'(\xi)]^2+W\left(u_R(\xi)\right)\leq C_3 R^{-1},\ \
R\geq C_1 C_2^{-1}+2,
\end{equation}
where the positive constat $C_3$ depends only on $W$ and $n$ (for
simplicity in notation, we have suppressed the obvious dependence of
$\xi$ on $R$). Hence, from assumption \textbf{(a')}, and relations
(\ref{eqmonotonicity}), (\ref{eqgood}), we obtain that
\begin{equation}\label{eqsweers}
u_R\to \mu, \ \textrm{uniformly \ in}\  \bar{B}_{\frac{R}{2}}, \
\textrm{as}\ R\to \infty.
\end{equation}
In the sequel, we will prove that the stronger property
(\ref{eqestimRect}) holds true.

For future reference, we note here that
\begin{equation}\label{eqchen}
[u_R'(R)]^2\to 2W(0)\ \ \textrm{as}\ \ R\to \infty.
\end{equation}
Indeed, let
\begin{equation}\label{eqER1}
E_R(r)=\frac{1}{2}[u_R'(r)]^2-W\left(u_R(r)\right),\ \ r\in (0,R).
\end{equation}
Thanks to (\ref{eqgidasEq}), we find that
\begin{equation}\label{eqER'}
E_R'(r)=u''_Ru'_R-W'(u_R)u'_R=-\frac{n-1}{r}(u_R')^2,\ \ r\in (0,R).
\end{equation}
So,
\begin{equation}\label{eqER2}
E_R(R)=E_R(\xi)-\int_{\xi}^{R}\frac{n-1}{r}(u_R')^2dr,
\end{equation}
where $\xi \in \left(\frac{R}{2},R \right)$ is as in (\ref{eqgood}).
Now, observe that (\ref{eqcoarea}) and the nonnegativity of $W$
imply that
\[
\int_{\xi}^{R}r^{n-1}(u_R')^2dr \leq C_4R^{n-1},\ \ R\geq C_1
C_2^{-1}+2,
\]
with $C_4$ depending only on $W$ and $n$. In turn, the above
estimate clearly implies that
\[
\int_{\xi}^{R}(u_R')^2dr \leq 2^{n-1} C_4,\ \ R\geq C_1 C_2^{-1}+2,
\]
and it follows that
\begin{equation}\label{eqER3}
 \int_{\xi}^{R}\frac{n-1}{r}(u_R')^2dr\leq 2^nC_4(n-1)R^{-1},\ \ R\geq C_1
C_2^{-1}+2.
 \end{equation}
 The claimed relation (\ref{eqchen}) follows readily from
 (\ref{eqgood}), (\ref{eqER1}), (\ref{eqER2}), and (\ref{eqER3}).
In fact, we have shown that $R|E_R(R)|$ remains uniformly bounded as
$R\to \infty$. In relation to (\ref{eqchen}), see also Remark
\ref{remshibata} below.

We also consider  the following family of functions
\begin{equation}\label{eqUR}
U_R(s)=u_R(R-s),\ \ s\in [0,R].
\end{equation}
We claim that
\begin{equation}\label{eqclaim}
U_R\to \textbf{U}, \ \textrm{uniformly\ on\ compact\ intervals\ of}\
[0,\infty), \ \textrm{as}\ R\to \infty,
\end{equation}
where $\textbf{U}$ is as in (\ref{eqU}).

In view of (\ref{eqgidasEq}), we get
\begin{equation}\label{eqhellyuniform}
U_R''-\frac{n-1}{R-s}U_R'-W'(U_R)=0,\ \ s\in (0,R).
\end{equation}
Making use of (\ref{eqmaxball}), the above equation, elliptic
estimates \cite{Gilbarg-Trudinger}, Arczela-Ascoli's theorem, and a
standard diagonal argument, passing to a subsequence $R_i\to
\infty$, we find that
\begin{equation}\label{eqURi}
U_{R_i}\to {V}\ \textrm{and}\ U_{R_i}'\to {V}', \ \textrm{uniformly\
on\ compact\ intervals\ of}\ [0,\infty), \ \textrm{as}\ i\to \infty,
\end{equation}
where $V\in C^2[0,\infty)$ is nonnegative and satisfies
\begin{equation}\label{eqhellyV}V''=W'(V),\ s>0,\ \textrm{and}\ V(0)=0.\end{equation} Moreover, by
(\ref{eqchen}), (\ref{eqUR}), and (\ref{eqURi}), we see that
\[
[V'(0)]^2=2W(0)>0.
\]
By the uniqueness of solutions of initial value problems of ordinary
differential equations, see for example page 108 in \cite{walter},
we deduce that
\[
V\equiv \textbf{U},
\]
where $\textbf{U}$ is as in (\ref{eqU}). We also used that
$\textbf{U},\ V$ are nonnegative (which implies that
$\textbf{U}'(0),\ V'(0)$ are also nonnegative), and the relation
\begin{equation}\label{eqstar}
[\textbf{U}'(0)]^2=2W(0),
\end{equation}
which follows from the identity
\[
[\textbf{U}'(s)]^2-[\textbf{U}'(0)]^2=2\int_{0}^{s}W'(\textbf{U})\textbf{U}'ds=2W\left(\textbf{U}(s)\right)-2W(0),\
\ s\geq0,
\]
and the fact that $\textbf{U}(s)\to \mu$ as $s\to \infty$, recalling
that $W(\mu)=0$ (otherwise, $\textbf{U}'(s)$ would tend to a nonzero
number and in turn $\left|\textbf{U}(s)\right|$ would diverge, as
$s\to \infty$). Moreover, by the uniqueness of the limiting
function, we infer that the limits in (\ref{eqURi}) hold for
\emph{all}  $R\to \infty$. Consequently, the claimed relation
(\ref{eqclaim}) holds.

Having (\ref{eqchen}), (\ref{eqclaim}) at our disposal, we can now
proceed to the proof of (\ref{eqestimRect}). Let $\epsilon \in
(0,\mu)$ and $D>D'$, where $D'$ is as in (\ref{eqD}). By virtue of
(\ref{eqU}), (\ref{eqUincr}), and (\ref{eqclaim}), there exists a
sufficiently large $R'$, depending only on $\epsilon$, $D$, $W$,
$n$, such that $ U_{R}(D)\geq \mu-\epsilon$, and all the previous
relations continue to hold, for $R>R'$. In other words, via
(\ref{eqUR}), we have that
\begin{equation}\label{eqstarting}
u_R (R-D)=U_R(D)\geq \mu-\epsilon,\ \ R>R'.
\end{equation}
The  fact that $u_R$ is radially decreasing, recall
(\ref{eqmonotonicity}), and the above relation imply the validity of
(\ref{eqestimRect}). As   will be apparent from the Remarks
\ref{remmonotbelowLem1} and \ref{remClearing} that follow, condition
(\ref{eqmonotonicity}) is essential only when dealing with
degenerate situations when  there exists a sequence $t_j\to\mu^-$
such that $W'(t_{2j})W'(t_{2j+1})<0$ for large $j$; an example is a
potential $W$ that coincides with
$(\mu-t)^2\left[\sin\left(\frac{1}{\mu-t}\right)+2 \right]$ near
$\mu$, in which case we can choose $t_j=\mu-\frac{1}{j\pi}$ (note
that $W'(t)\sim \cos\left(\frac{1}{\mu-t}\right)$ as $t\to \mu^-$).
It remains to prove (\ref{eqcaffaLemmmine}). To this end, note that
the nonnegativity of $W$ and (\ref{eqJcomp}) imply that
\[
\int_{r}^{R}s^{n-1}W\left(u_R(s) \right)ds\leq \tilde{C}_1 R^{n-1},\
\ r\in (0,R),
\]
where $\tilde{C}_1$ is independent of $R\geq R'$. It  follows, via
(\ref{eqmonotonicity}), that
\[
\min\left\{W(t)\ :\ t\in \left[0,u_R(r)\right] \right\}(R^n-r^n)\leq
n\tilde{C}_1 R^{n-1},
\]
which clearly implies the validity of (\ref{eqcaffaLemmmine}).


The proof of the lemma is complete.
\end{proof}
\begin{rem}\label{rembelowNew}
Our assumptions on the behavior of $W$ near its global minimum at
$\mu$ are quite weak, and in fact even allow for the potential $W$
to
 have $C^\infty$ contact with zero at the point $\mu$, that is
$W^{(i)}(\mu)=0,\ i\geq 1$. This degeneracy translates into the
absence of decay rates for the convergence of the ``inner''
approximate solution $\textbf{U}(R-|x|)$ (in the sense of singular
perturbation theory, see \cite{fife-arma} and the related references
that can be found in Remark \ref{remSingular} below), where
$\textbf{U}$ is as described in (\ref{eqU}), to the ``outer'' one
$\mu$, away from the boundary of $B_R$, as $R\to \infty$ (see also
the discussion leading to (\ref{eqEV2}) below). This is the main
reason why we have not attempted to apply a perturbation argument,
see for instance \cite{fife-arma} and the related references in
Remark \ref{remSingular} below, in order to study the asymptotic
behavior of $u_R$ as $R\to \infty$. We refer to the recent papers
\cite{vasilevaDiff,vasileva1-,vasileva2} for  singular perturbation
arguments (in one space dimension) in the case where $\mu$ is a root
of $W'$ of finite multiplicity (also allowing for $x$ dependence on
$W'$). From the viewpoint of geometric singular perturbation theory,
the case $W''(\mu)=0$ corresponds to lack of normal hyperbolicity of
the slow manifold corresponding to the equilibria with $(u,
u')=(\mu,0)$ (see \cite{tin}).

If $W''(\mu)>0$, then the convergence of $\textbf{U}$ to $\mu$ is of
order $e^{-\sqrt{W''(\mu)}s}$ as $s\to \infty$ (by the stable
manifold theorem, see \cite{codington}), and one can effectively
interpolate between the outer and inner approximations in order to
construct a smooth global approximation that is valid throughout
$B_R$.
\end{rem}

\begin{rem}\label{remn}
By the well known relations $|B_R|=c_n R^n$, $|\partial B_R|=n c_n
R^{n-1}$, $R>0$, $n\geq 2$, for some explicit constants $c_n$
(independent of $R$), where $|\partial B_R|$ denotes the
$(n-1)$--dimensional measure of $\partial B_R$, we find that
\[
\frac{|\partial B_R|}{\left|B_R\backslash
B_\frac{R}{2}\right|}=\frac{n2^n}{2^n-1}R^{-1},\ \ R>0.
\]
We deduce that the constant $R'$ in Lemma \ref{lem1} diverges (at
least linearly) as $n\to \infty$ (see in particular the relations
leading to (\ref{eqgood})).
\end{rem}

\begin{rem}\label{remmonotbelowLem1}
If in addition to \textbf{(a')} we assume that there exists some
$d\in (0,\mu)$ such that
\begin{equation}\label{eqWmonotone}W'(t)\leq 0,\ \ t\in (\mu-d,\mu),
\end{equation}
(note that this is very natural), then relation (\ref{eqestimRect})
can alternatively be shown, starting from (\ref{eqstarting}),
\emph{without} assuming knowledge of (\ref{eqmonotonicity}), as
follows:  Assuming, without loss of generality, that $2\epsilon<d$,
thanks to   Lemma \ref{lemalikakos} below, we can find a radial
$\tilde{u}\in W^{1,2}(B_{R-D})$ such that
\[
J(\tilde{u};B_{R-D})\leq J(u_R;B_{R-D}), \ \
\tilde{u}(R-D)=u_R(R-D),\ \ \textrm{and}\ \ \tilde{u}(x)\in
[\mu-\epsilon,\mu],\ x\in \bar{B}_{R-D}.
\]
Thus, the function
\[
\hat{u}(x)=\left\{\begin{array}{ll}
                 \tilde{u}(x), & x\in B_{R-D}, \\
                   &   \\
                 u_R(x), & x\in B_R\backslash B_{R-D},
               \end{array}
 \right.
\]
belongs in $W^{1,2}_0(B_R)$ and is a global minimizer of $J(\cdot;
B_R)$ in $W^{1,2}_0(B_R)$ (since $J(\hat{u};B_R)\leq J(u_R;B_R)$).
In particular, it is smooth, radial (and by virtue of its
construction), and solves (\ref{eqgidasEq}). It follows from Lemma
3.1 in \cite{jerison}, which is in the spirit of Lemma
\ref{lemdancer} below, that the function $u_R-\hat{u}$ is either
strictly positive, strictly negative, or identically equal to zero
in $B_R$, and obviously the latter case occurs. For completeness
purposes, as well as for future reference, we will draw the same
conclusion by an alternative and, to our opinion, more elementary
approach: The function \[v\equiv u_R-\hat{u}\] solves the linear
equation
\[
\Delta v+Q(x)v=0,\ \ x\in B_R, \] where
\begin{equation}\label{eqQ}
Q(x)=\left\{
\begin{array}{ll}
  \frac{W'\left(\hat{u}(x)\right)-W'\left(u_R(x)\right)}{u_R(x)-\hat{u}(x)}, & \textrm{if}\ \hat{u}(x)\neq u_R(x), \\
   &  \\
  -W''(u_R(x)), & \textrm{if}\ \hat{u}(x)= u_R(x).
\end{array}
\right.
\end{equation}
On the other hand, since
\[
v(x)=0,\ \ x\in  B_R\backslash B_{(R-D)},
\]
and $Q\in L^\infty(B_R)$, the unique continuation principle (see for
instance \cite{hormander}) yields that \[v(x)=0,\ x\in B_R.\] (In
this simple case of radial symmetry, we can also make use of the
uniqueness theorem of ordinary differential equations to show that
$v\equiv 0$). Therefore, estimate (\ref{eqestimRect}) holds.
  We remark
that, if $W$ was \emph{strictly} decreasing in $(\mu-d,\mu)$, then
(\ref{eqestimRect}) follows at once from the general lemma in
\cite{alikakosReplace} (see also the second assertion of Lemma
\ref{lemalikakos} herein) and (\ref{eqstarting}).

The approach that we just presented makes only partial use of the
radial symmetry of the problem (in order to establish
(\ref{eqstarting})), and may be applied to extend some results in
\cite{danceryanCVPDE} to the general case (without radial symmetry),
see \cite{sourdisDancer}. Moreover, it can be applied for the study
of global minimizers of the analogous  vector-valued energy
functionals, as those appearing in \cite{alikakosARMA}, over $B_R$.
In this case, it is known that global minimizers are radial, see
\cite{lopez}, but monotonicity properties do not hold in general.
\end{rem}

\begin{rem}\label{remClearing}
In the one dimensional case, i.e.,  when $n=1$, the assertion of
Remark \ref{remmonotbelowLem1} can be shown \emph{without} assuming
(\ref{eqWmonotone}). As in the latter remark, we do not assume the
monotonicity property (\ref{eqmonotonicity}) of $u_R$, just that it
is even, and we will start from (\ref{eqstarting}) which clearly
implies that \begin{equation}\label{eqstrauss1}u_R(R-D)\to \mu \ \
\textrm{as}\ \ R\to \infty.
\end{equation}
 Since the energy of $u_R$ is not larger
than that of the even function given by
\begin{equation}\label{eqscrewclear}\check{u}_R(x)=\left\{\begin{array}{ll}
                          u_R(x), &  x\in[R- D,R],\\
                           &  \\
                          \frac{u_R(R-D)-\mu}{D}(x-R+D)+u_R(R-D), &  x\in [R-2D,R-D], \\
                           &  \\
                          \mu, &  x \in[0,R-2D],
                        \end{array}
 \right.
\end{equation}
it follows readily from \textbf{(a')} and (\ref{eqstrauss1}) that
\begin{equation}\label{eqclearing-}
\int_{-R+D}^{R-D}\left\{(u_R')^2+W(u)\right\}dx\to 0 \ \ \textrm{as}
\ R\to \infty.
\end{equation}
Hence, by \textbf{(a')} and  the \emph{clearing--out} Lemma 1 in
\cite{bethuelJDE} (noting that it continues to apply in our possibly
degenerate setting), we have that
\begin{equation}\label{eqclearing}u_R\to \mu, \ \textrm{uniformly\
in}\ [-R+D,R-D],\  \textrm{as}\  R\to \infty.\end{equation} The
intuition behind the latter lemma, as applied in the case at hand,
is  that if the energy is sufficiently small in some place, then
there are no spikes located there.
 Note
that from (\ref{eqmaxball}), (\ref{eqgidasEq}), in arbitrary
dimensions, via standard interior elliptic regularity estimates
\cite{Gilbarg-Trudinger} (see also Lemma A.1 in
\cite{bethuelCVPDE}), applied on balls of radius $\frac{D}{4}$
covering $B_{(R-D)}$, we have that $|\nabla u_R|$ remains uniformly
bounded in $B_{(R-D)}$ as $R\to \infty$ (or see the gradient bound
in (\ref{eqmodicaRadial}) below). Thus, relation (\ref{eqclearing})
can also be derived from \textbf{(a')} and (\ref{eqclearing-})
similarly to  Theorem III.3 in \cite{bethuel}, see also Lemma 3.2 in
\cite{weiMoser} (the point is that the ``bad'' intervals, where
$u_R$ is away from $\mu$ must have size of order one (by the uniform
gradient estimate), as $R\to \infty$, which is not possible by
\textbf{(a')} and (\ref{eqclearing-})). In contrast to the one
dimensional case,
 in
$n\geq 2$ dimensions, by the analog of (\ref{eqclearing-}), i.e.,
\begin{equation}\label{eqanalogNNN}
R^{1-n}J(u_R;B_{(R-D)})\to 0\ \  \textrm{as}\ \  R\to \infty,
\end{equation} arguing again as in Theorem III.3 in \cite{bethuel}, we can
show the weaker property:
\begin{equation}\label{eqclearweak}\textrm{Given}\ \alpha\in (0,1)\  \Rightarrow \ u_R\to \mu,\ \textrm{uniformly\
in} \ \bar{B}_{(R-D)}\backslash B_{\alpha R}, \ \textrm{as}\  R\to
\infty.
\end{equation}


We note that if $W''(\mu)>0$, then (\ref{eqclearing}) follows
directly  from (\ref{eqclearing-}), via (\ref{eqCschoen}) below and
the Sobolev embedding
\[\|\mu-u_R\|_{L^\infty(-R+D,R-D)}\leq
C\|\mu-u_R\|_{W^{1,2}(-R+D,R-D)},\] with constant $C$ independent of
$R\geq 2D$ (see Corollary 5.16 in \cite{adams}). One might be
curious whether this  simple argument can be extended to $n\geq 2$
dimensions. In this direction, we would like to mention that by
using the pointwise estimate
\[
\left(\mu-u_R(r)\right)^2\leq C_n
r^{1-n}\|\mu-u_R\|^2_{W^{1,2}\left(B_{(R-D)}
\right)}+\left(\frac{R-D}{r}
\right)^{n-1}\left(\mu-u_R(R-D)\right)^2,\ \ r\in(0,R-D),
\]
which can be proven similarly as the classical Strauss  radial lemma
(see \cite{strauss}), relation (\ref{eqstrauss1}), and
(\ref{eqanalogNNN}), we arrive again at (\ref{eqclearweak}). On the
other side, as in \cite{alikakosBasicFacts}, fixing $K$ and letting
$R\to \infty$, we see from the monotonicity formula (\ref{eqmodica})
below that $u_R\to \mu$, uniformly on $\bar{B}_K$, as $R\to \infty$
(see also Remark \ref{remsigal} below, and the compactness argument
that follows). Suppose that for a sequence $R\to \infty$, there
exist $r_R\in [0,R-D]$ such that $u_R(r_R)=\mu-2\epsilon$. From
(\ref{eqclearweak}) and our previous comment, we get that
$R-r_R\to\infty$ and $r_R\to \infty$, as $R\to \infty$,
respectively. As in the proof of Theorems 1.3 and 1.4 in
\cite{danceryanCVPDE}, we let $v_R(s)=u_R(r_R+s)$, $s\in
(-r_R,R-r_R)$, note that $v_R(0)=\mu-2\epsilon$. Using
(\ref{eqmaxball}), (\ref{eqgidasEq}), together with standard
elliptic regularity estimates and Sobolev embeddings (see
\cite{Gilbarg-Trudinger}), passing to a subsequence, we find that
$v_R\to V$ in $C^1_{loc}(\mathbb{R})$, where
\begin{equation}\label{eqVdan}
V''=W'(V),\ \ 0\leq V\leq \mu,\ \ s\in \mathbb{R},\ \
V(0)=\mu-2\epsilon.
\end{equation}
Moreover, the solution $V$ is a minimizer of the energy
\[
I(v)=\int_{-\infty}^{\infty}\left[\frac{1}{2}(v')^2+W(v)\right]ds,
\]
in the sense that $I(V+\phi)\leq I(V)$ for every $\phi\in
C^\infty_0(\mathbb{R})$, see page 104 in \cite{danceryanCVPDE}.
Arguing as in the proof of De Giorgi's conjecture in low dimensions
(see \cite{ambrosio3D}, \cite{berestyckiCaffarelli},
\cite{dancerMorseTrans}, \cite{farinaHB}, \cite{ghosubGui},
\cite{pacard}), we can prove that either $V$ is a constant with
$W'(V)=0$, $W''(V)\geq 0$ or $V'$ is nontrivial and has fixed sign.
Since we are assuming  that $W''(\mu)>0$, the first scenario is
ruled out at once from the last condition in (\ref{eqVdan}); in the
second scenario, it follows from phase analysis (see \cite{arnold},
\cite{walter}) that $V$ has to connect two equal wells of the
potential $W$ at respective infinities, one of them being $\mu$, but
this is impossible since $W(t)>0,\ t\in [0,\mu)$. Consequently, if
we assume that $W''(\mu)>0$, assertion (\ref{eqestimRect}) can be
deduced in this manner from (\ref{eqstarting}) \emph{without} making
use of (\ref{eqmonotonicity}) for \emph{all} $n\geq 1$.

\end{rem}

\begin{rem}\label{remmonotbelowLem12}
If $W\in C^{2,\alpha}(\mathbb{R})$, $0<\alpha<1$, satisfies
\textbf{(a')},
\[
W'(\rho_1)=0,\ \ W'(t)<0,\ t\in (\rho_1,\mu),\ \textrm{for\ some}\
\rho_1\in (0,\mu),
\]
 and
(\ref{eqW''}), then Theorem 2 in \cite{sweersEdinburg} tells us that
there exists a $\delta_1\in (0,\mu)$ such that (\ref{eqgidasEq}) has
at most one solution  such that
\[\max_{x\in \bar{B}_R}u(x) \in (\mu-\delta_1,\mu)\ \textrm{and}  \ -\mu<u(x)<\mu,\ x\in B_R,\]
for all $R>0$. Therefore, under these assumptions on $W$, in view of
(\ref{eqmaxball}) and (\ref{eqestimRect}) which hold for all global
minimizers (with the same $R'$), we conclude that there exists a
unique global minimizer of (\ref{eqenergy}), if $R$ is sufficiently
large.

On the other side, if in addition to \textbf{(a')}, the stronger
assumption $W''(\mu)>0$ holds (in other words \textbf{(c)}), then a
simple proof of the uniqueness of the global minimizer, satisfying
(\ref{eqmaxball}), for large $R$, can be given as follows: One first
shows that if a solution of (\ref{eqgidasEq}) satisfies
(\ref{eqmaxball}), (\ref{eqestimRect}), and (\ref{eqclaim}) (recall
(\ref{eqUR})), then it is \emph{asymptotically} stable for large
$R>0$ (we will give a short self--contained proof of this in the
sequel).
Then, suppose that $u_1$ and $u_2$ are two distinct global
minimizers of (\ref{eqenergy}), satisfying (\ref{eqmaxball}). By the
proof of Lemma \ref{lem1}, they satisfy (\ref{eqmaxball}),
(\ref{eqestimRect}), and (\ref{eqclaim}), uniformly (independent of
the choice of minimizers) as $R\to \infty$.
 Thanks to Lemma 3.1 in \cite{jerison} (see also Lemma
\ref{lemdancer} herein), without loss of generality, we may assume
that $u_1(x)<u_2(x)$, $x\in B_R$ (in the problem at hand, we can
also assume this when dealing with stable solutions). On the other
hand, by the mountain pass theorem or the theory of monotone
dynamical systems (see \cite{defiguerdo}, \cite{matano}
respectively, and Section \ref{secExtensions} herein), we infer that
there exists an \emph{unstable} solution $\hat{u}_1$ of
(\ref{eqgidasEq}) such that $u_1(x)<\hat{u}_1(x)<u_2(x),\ x\in B_R$.
In particular, the unstable solution enjoys the asymptotic behavior
of global minimizers, as $R\to \infty$, and thus is asymptotically
stable (by our previous discussion); a contradiction. A related
uniqueness proof, based on a dynamical systems argument (but not of
monotone nature), can be found in \cite{alikakosSimpson}.

Here, for completeness, assuming that $W''(\mu)>0$, we will show
that  solutions $u_R$ of (\ref{eqgidasEq}) which satisfy
(\ref{eqmaxball}), (\ref{eqestimRect}), and (\ref{eqclaim}) are
\emph{asymptotically} stable if $R$ is sufficiently large. Our
argument is inspired from \cite{angenent} where, in particular,
under the additional assumption \textbf{(b)} with strict inequality,
it was applied to (\ref{eqEV1}) below on a smooth bounded domain
with large $\lambda$. We will ague by contradiction. Suppose that,
for a sequence $R\to \infty$, the principal eigenvalue $\mu_R$ of
the linearized operator about $u_R$ is non-positive, i.e,
\begin{equation}\label{eqmuR}
\mu_R\leq 0.
\end{equation}
It is well known that $\mu_R$ is simple and that the corresponding
eigenfunction $\varphi_R$ (modulo normalization) may be chosen to be
positive in $B_R$, see for instance Theorem 8.38 in
\cite{Gilbarg-Trudinger}. We have
\begin{equation}\label{eqkernelstarb}
-\Delta \varphi_R+W''(u_R)\varphi_R=\mu_R\varphi_R\ \textrm{in}\
B_R;\ \varphi_R=0 \ \textrm{on}\ \partial B_R,
\end{equation}
and we   normalize $\varphi_R$ by imposing that
\begin{equation}\label{eqnormal}
\|\varphi_R\|_{L^\infty (B_R)}=1.
\end{equation}
We note that $\varphi_R$  is radially symmetric (and so is every
eigenfunction that is associated to a non-positive eigenvalue, see
\cite{harauxPolacik}, \cite{linNi}, because \textbf{(a')} and Hopf's
boundary point lemma yield that $u_R'(R)<0$). For future reference,
observe that testing (\ref{eqkernelstarb}) by $\varphi_R$ yields the
uniform (in $R$) lower bound:
\begin{equation}\label{eqmulower}
\mu_R\geq -\max_{t\in[0,\mu]}\left|W''(t) \right|.
\end{equation}
Now, by  virtue of (\ref{eqestimRect})  and the positivity of
$W''(\mu)$, there exists a constant $D>0$ such that
\[
W''\left(u_R\right)\geq \frac{W''(\mu)}{2}>0\ \ \textrm{on}\ \
\bar{B}_{(R-D)},
\]
for large $R>0$. So, from (\ref{eqmuR}), (\ref{eqkernelstarb}), and
(\ref{eqnormal}), we obtain that there exist $z_R\in [R-D,R]$ such
that $\varphi_R(z_R)=1$, $\varphi_R'(z_R)=0$, and
$\varphi_R''(z_R)\leq 0$, for large $R$ (along the sequence).
 As in the proof of Lemma
\ref{lem1}, making use of (\ref{eqclaim}), (\ref{eqmuR}),
(\ref{eqkernelstarb}),  (\ref{eqnormal}), and (\ref{eqmulower}),
passing to a subsequence, we get that $\varphi_{R_i}(R_i-\cdot)\to
\Phi(\cdot)$ in $C^1_{loc}[0,\infty)$, $\mu_{R_i}\to \mu_*\leq 0$,
and $R_i-z_{R_i}\to \textbf{z}\in [0,D]$, as $i\to \infty$, such
that
\begin{equation}\label{eqskribble}
-\Phi''+W''\left(\textbf{U}(r)\right)\Phi=\mu_* \Phi, \ r\in
(0,\infty);\ \Phi(0)=0,\
\Phi(\textbf{z})=\|\Phi\|_{L^\infty(0,\infty)}=1,
\end{equation}
where $\textbf{U}$ is as in (\ref{eqU}). On the other hand,
differentiating (\ref{eqU}), multiplying the resulting identity by
$\frac{\Phi^2}{\textbf{U}'}$ (recall (\ref{eqUincr})) and
integrating by parts over $(0,\infty)$, we arrive at $\mu_*\geq 0$
(see also Proposition 3.1 in \cite{pacard}); to be more precise, one
first multiplies by $\frac{\zeta_m^2}{\textbf{U}'}$, with
$\zeta_m\in C_0^\infty(0,\infty)$ such that $\zeta_m\to \Phi$ in
$W^{1,2}_0(0,\infty)$, and then lets $m\to \infty$. A different way
to see that $\mu_*\geq 0$ is to note that the linear operator
defined by the lefthand side of (\ref{eqskribble}) is an unbounded,
self-adjoint operator  in $L^2(0,\infty)$ with domain
$W^{1,2}_0(0,\infty)\cap W^{2,2}(0,\infty)$, having as continuous
spectrum the interval $[W''(\mu), \infty)$ and principal eigenvalue
zero (by the positivity of $\textbf{U}'$), see also Remark 2.8 in
\cite{stefanopoulos} or Proposition 1 in \cite{henryCMP} or
\cite{schatman}.
 In other words, recalling (\ref{eqmuR}), we have
\[
-\Phi''+W''\left(\textbf{U}(r)\right)\Phi=0,\ \Phi>0, \ r\in
(0,\infty);\ \Phi(0)=0,\
\Phi(\textbf{z})=\|\Phi\|_{L^\infty(0,\infty)}=1.
\]
 The above linear second
order equation has the following two independent solutions:
\[
\textbf{U}'(r)\ \ \textrm{and}\ \
\textbf{U}'(r)\int_{0}^{r}\frac{1}{\left[\textbf{U}'(s)
\right]^2}ds,
\]
see for example Lemma 3.2 in \cite{batesZhang}. It is easy to see
that the second solution grows unbounded as $r\to \infty$ (plainly
apply l'hospital's rule), and thus $\Phi$ has to be
$\|\textbf{U}'\|^{-1}_{L^\infty(0,\infty)}\textbf{U}'$. Since
$\Phi(0)=0$, whereas $\textbf{U}'(0)=\sqrt{2W(0)}>0$, we have
reached a contradiction.


For further information on ``asymptotic'' uniqueness of positive
solutions, in arbitrary domains, we refer to Remark
\ref{remuniqDancer} below.

 In the
case where uniqueness of a stable solution, satisfying
(\ref{eqmaxball}), holds for $R>R_0\geq 0$ (recall Remark
\ref{remuniq} and see Remark \ref{remuniqDancer} below), it is easy
to see that the family $\{u_R\}_{R>R_0}$ is nondecreasing with
respect to $R$, namely
\begin{equation}\label{eqnondecreasing}
u_{R_2}(x)>u_{R_1}(x),\ \ x\in \bar{B}_{R_1},\ \ \forall\
R_2>R_1>R_0,
\end{equation}
see Lemma 1 in \cite{fifesaddle}. Moreover, as in Lemma 2 in
\cite{fifesaddle}, we have that
\begin{equation}\label{eqfifebarrier}
u_R(R-r)\leq \textbf{U}(r),\ \ r\in [0,R],\ \forall R>R_0,
\end{equation}
(plainly observe that, thanks to (\ref{eqU}) and (\ref{eqUincr}),
the function $\textbf{U}(R-r)$ is a weak upper solution to
(\ref{eqgidasEq}) in the sense of \cite{Berestyckilion}).
 We note
that, arguing as in Remark \ref{remmonotbelowLem1} (see also Lemma
\ref{lemdancer} in Appendix \ref{secappenda}), it follows that
(\ref{eqnondecreasing}) holds \emph{without} assuming uniqueness
(plainly observe that $J\left(
\max\{u_{R_2},u_{R_1}\};B_{R_2}\right)\leq J(u_{R_2},B_{R_2})$, see
also Lemma 5.3 in \cite{guiAnnals}). Moreover, similarly to
Proposition 1.5 in \cite{cabre}, based on the gradient estimate
(\ref{eqmodicaRadial}), one can show (\ref{eqfifebarrier}) without
assuming uniqueness.
\end{rem}

\begin{rem}\label{remscrew}
If in addition to \textbf{(a')}, we assume that $W'(0)=0$,
$W''(0)<0$, and $W\in C^3$ ($W'''$ bounded for $t>0$ small is
enough), then (\ref{eqgidasEq}) admits a nontrivial positive
solution, which is a global minimizer of $J(\cdot;B_R)$ in
$W^{1,2}_0(B_R)$, as long as $R>R_c$, where
\begin{equation}\label{eqRc}
R_c=\sqrt{-\frac{\lambda_1}{W''(0)}},
\end{equation}
and $\lambda_1$ denotes the principal eigenvalue of $-\Delta$ in
$W^{1,2}_0(B_1)$ (an analogous result holds for (\ref{eqEV1})
below). To see this, let $\varphi_1$ denote the associated
eigenfunction with the normalization $\varphi_1(0)=1$ ($\varphi_1$
is radially decreasing). Then, the pair
\begin{equation}\label{eqhellyphiR}\lambda_R=\lambda_1R^{-2}\ \
\textrm{and}\ \  \varphi_R(x)=\varphi_1(R^{-1}x)\end{equation} is
the principal eigenvalue and eigenfunction of $-\Delta$ in
$W^{1,2}_0(B_R)$ such that $\varphi_R(0)=1$. Now, the desired
conclusion follows at once by noting that
\[
J\left(\varepsilon \varphi_R;B_R \right)=J\left(0;B_R
\right)+\frac{\varepsilon^2}{2}\int_{B_R}^{}\varphi_R^2\left(\lambda_1R^{-2}+W''(0)+\mathcal{O}(\varepsilon)
\right)dx\ \ \textrm{as}\  \varepsilon\to 0^+,
\]
which implies that zero is not a global minimizer if $R>R_c$ (see
also Example 5.11 in \cite{malchiodicambridge}, Theorem 2.19 in
\cite{badiale}, Lemma 2.1 in \cite{delpinoScrew} and Proposition
1.3.3 in \cite{dupaigneBook}; note also that $\varepsilon\varphi_R$,
with $R>R_c$, is a lower solution to (\ref{eqgidasEq}) for small
$\varepsilon>0$). (If $W$ is even, one can construct a plethora of
sign-changing solutions, for large $R$, not necessarily radial, by
noting that $J(u;B_R)<J(0;B_R)$ for $u\in
\textrm{Span}\left\{\varphi_1(R^{-1}x),\cdots,\varphi_k(R^{-1}x)\right\}$,
$k\geq 1$, and $\|u\|_{L^2(B_R)}$ sufficiently small, where
$\varphi_i$ denote eigenfunctions of the Laplacian in
$W^{1,2}_0(B_1)$ (normalized so that
$\|\varphi_i(R^{-1}x)\|_{L^2(B_R)}=1$ and
$\int_{B_1}\varphi_i\varphi_j dx=0$ if $i\neq j$), corresponding to
the first $k$ eigenvalues (counting multiplicities), and applying
Theorem 8.10 in \cite{rabcbms}; see also \cite{alikakosPhilips} and
Theorem 10.22 in \cite{malchiodicambridge}).

 If we further assume that
\begin{equation}\label{eqdelpinobelowRc}
W'(t)\geq W''(0)t,\ \ t\geq 0,
\end{equation}
then (\ref{eqgidasEq}), for $R\in (0,R_c)$, has no positive solution
as can be seen by testing the equation by $\varphi_R$.

Under some different conditions, which are compatible with
\textbf{(a')}, and are satisfied for example by the nonlinearity in
(\ref{eqwei}), there exists an $R'_c>0$ such that (\ref{eqgidasEq})
has exactly one positive solution for $R=R_c'$ and exactly two for
$R>R_c'$, the one is a global minimizer while the other is a
mountain pass of the associated energy (see \cite{shiOyang},
\cite{shiBook}, \cite{weiProc}).
\end{rem}

\begin{rem}\label{remcoarea}
By (\ref{eqcoarea}), via the coarea formula (see
\cite{evansGariepy}), it follows that there exists a $\xi_R \in
\left(\frac{R}{2}, R \right)$ such that
\[
\int_{\partial B_{\xi_R}}^{}\left\{\frac{1}{2}|\nabla
u_R|^2+W(u_R)\right\}dS\leq 2 C_1 R^{n-2},\ \ R\geq C_1 C_2^{-1}+2.
\]
This observation makes no use of the radial symmetry of $u_R$, and
is motivated from the proof of the corollary in
\cite{alikakosReplace}. In regard to the latter comment, it might be
useful to recall our Remark \ref{remClearing} and compare with the
arguments of \cite{alikakosReplace}.
\end{rem}

\begin{rem}\label{rempoho}
In  case a $C^2$ potential $W$ satisfies $W(0)=0$ and the domain
$\Omega$ has $C^1$ boundary, is bounded, and star-shaped with
respect to some point in its interior, the well known Pohozaev
identity easily implies that there does not exist a nontrivial
solution of (\ref{eqEq}) such that $W(u(x))\geq 0, \ x\in \Omega$
(see for instance relation (11) in \cite{alikakosfaliagas}, a
reference which is in accordance with our notation).
Actually, relation (11) in the latter reference holds true for the
elliptic system that corresponds to (\ref{eqEq}) (with the obvious
notation), and an analogous nonexistence result holds in that
situation as well.
\end{rem}

\begin{rem}\label{remalikakossimple}
Under the stronger assumptions \textbf{(a)} (or more generally
\textbf{(a')}), \textbf{(b)}, and \textbf{(c)}, considered in
\cite{fuscoTrans} (recall the introduction herein), motivated from
the proof of Lemma 3 in \cite{lin} (see also \cite{schoen} and the
remarks following Lemma 2.1 in \cite{caffareliLin}),
 we
can give a streamlined proof of relation (\ref{eqsweers}) as
follows: Note first that, thanks to \textbf{(a')} and \textbf{(c)},
there exists a positive constant $c_0$ such that
\begin{equation}\label{eqCschoen}
W(t)\geq c_0 (\mu-t)^2,\ \ 0\leq t\leq \mu.
\end{equation}
Then, bounds (\ref{eqmaxball}), (\ref{eqJcomp}), and the above
relation yield that
\begin{equation}\label{eqL2alik}
\int_{B_R}^{}(\mu-u_R)^2dx\leq c_1 R^{n-1},\ \ R\geq 2,
\end{equation}
where the positive constant $c_1$ depends only on $W$ and $n$.
 Next, note that assumption $\textbf{(b)}$, bound (\ref{eqmaxball}), and the equation in (\ref{eqgidasEq}), imply
that the function $\mu-u_R$ is subharmonic in $B_R$, and thus we
have
\[
\Delta (\mu-u_R)^2\geq 0\ \ \textrm{in}\ \ B_R,\ \ R\geq 2.
\]
In other words, the function $(\mu-u_R)^2$ is also subharmonic in
$B_R$. Consequently, by (\ref{eqL2alik}) and the mean value
inequality of subharmonic functions (see Theorem 2.1 in
\cite{Gilbarg-Trudinger}) together with a simple covering argument
(see also the general Theorem 9.20 in \cite{Gilbarg-Trudinger} and
Chapter 5 in \cite{morrey}),
we deduce that
\begin{equation}\label{eqmorrey}
\max_{\bar{B}_{\frac{R}{2}}}(\mu-u_R)^2\leq c_2
R^{-n}\int_{B_R}^{}(\mu-u_R)^2dx \leq c_3 R^{-1},\ \ R\geq 2,
\end{equation}
where the positive constants $c_2,\ c_3$ depend only on $W$ and $n$.
The latter inequality clearly implies the validity of
(\ref{eqsweers}). In passing, we note that the spherical mean of
$(\mu-u_R)^2$ appearing in the above inequality is nondecreasing
with respect to $R$, because of the subharmonic property, see
\cite{QitnerrIma}.

The above argument makes no  use of the fact that $u_R$ is radially
symmetric. Moreover, it works equally well if instead of
(\ref{eqCschoen}) we had  $W(t)\geq c(\mu-t)^p,\ t\in [0,\mu]$, for
some constants $c>0$ and $p> 2$. In Appendix \ref{appenAlik}, we
will adapt this approach in order  to simplify some arguments from
Section 6  of the recent paper \cite{alikakosNewproof}, where the De
Giorgi oscillation lemma for subharmonic functions was employed
instead of the mean value inequality. The former lemma roughly says
that  if a positive subharmonic function is smaller than one in
$B_1$ and is ``far from one" in a set of non trivial measure, it
cannot get too close to one in $B_{\frac{1}{2}}$ (see for example
\cite{caffarelliDegiorgi}). An intriguing application of the
techniques in the current remark is given in the following Remark
\ref{remscrew2}.
\end{rem}

\begin{rem}\label{remscrew2}
When seeking solutions of (\ref{eqentire}) in $\mathbb{R}^3$ which
are invariant under screw motion and whose nodal set is a
helico\"{i}d, assuming that $W$ is even, by introducing cylindrical
coordinates, one is led to study positive solutions of
\begin{equation}\label{eqscrewleft}
\partial_r^2U+\frac{1}{r}\partial_r U+\left(
1+\frac{\lambda^2}{\pi^2r^2}\right)\partial_s^2U-W'(U)=0,
\end{equation}
in the infinite half strip $\{(r,s)\in (0,\infty)\times (0,\lambda)
\}$, vanishing  on the boundary of $[0,\infty)\times [0,\lambda]$,
where $\lambda$ corresponds to a dilation parameter of a fixed
helico\"{i}d. More specifically, such solutions $U$ give rise to
solutions $u$ of (\ref{eqentire}) which vanish on the helico\"{i}d
that is parameterized by
\[
 \left\{\left(r\cos\theta,r\sin\theta,
z \right)\in \mathbb{R}^3\ :\ z=\frac{\lambda}{\pi}\theta\right\},
\]
see \cite{delpinoScrew} for the details. In the latter reference,
assuming that $W''(0)<0$ and (\ref{eqdelpinobelowRc}), it was shown
that there exists an explicit constant $\lambda^*>0$ such that the
above problem has a positive solution $U_\lambda$  if and only if
$\lambda>\lambda^*$ ($\lambda^*$ is actually equal to $2R_c$, where
$R_c$ is given from (\ref{eqRc}) with $n=1$).

Here, motivated from our previous Remark \ref{remalikakossimple}, we
will study this problem for large values of $\lambda$ under
complementary conditions on $W$ (in particular, without assuming
that $W''(0)<0$). Due to the presence of singularities in the
equation (\ref{eqscrewleft}) at $r=0$, as in Lemma 3.4 in
\cite{duNakas}, we will first consider the approximate (regularized)
problem
\begin{equation}\label{eqscrewleft2}
\Delta_{\mathbb{R}^2} U+\left(
1+\frac{\lambda^2}{\pi^2|x|^2}\right)\partial_s^2U-W'(U)=0,
\end{equation}
in $\{\xi<|x|,\ s\in (0,\lambda) \}$ with zero conditions on
$\{|x|=\xi, \ s\in [0,\lambda]\}$ and $\{\lambda=0,\ |x|\geq \xi
\}$, $\{\lambda=1,\ |x|\geq \xi \}$, with $\xi$ small (this was
skipped in \cite{delpinoScrew}).  Then, we consider equation
(\ref{eqscrewleft2}) in the annular cylinder $\{\xi<|x|<R,\ s\in
(0,\lambda)\}$, imposing that $U$ also vanishes on $|x|=R$. Assuming
\textbf{(a')}, as in Lemma \ref{lem1}, by minimizing the energy
\begin{equation}\label{eqcellbiol}
E(V)=\frac{1}{2}\int_{}^{}\left\{|\nabla_x V|^2+\left(
1+\frac{\lambda^2}{\pi^2|x|^2}\right)|\partial_sV|^2+2W(V)
\right\}dxds
\end{equation}
in $W^{1,2}_0\left((B_R\backslash B_\xi)\times (0,\lambda) \right)$
(with the obvious notation), but this time \emph{in the radially
symmetric class} with respect to $|x|$ (minimizers in this class are
critical points in the usual sense, see \cite{palais}),
 we find a solution $U_{\xi,R,\lambda}$
of (\ref{eqscrewleft2}), satisfying the prescribed Dirichlet
boundary conditions, such that $0\leq U_{\xi,R,\lambda}(|x|,s)\leq
\mu$ on $\left(\bar{B}_R\backslash B_\xi\right)\times [0,\lambda]$
(see  Lemma \ref{lemalikakos} below). Moreover, as in the proof of
Lemma \ref{lem1}, we have
\begin{equation}\label{eqscrewE}
E\left(U_{\xi,R,\lambda} \right)\leq CR\lambda, \ \ R\geq 2,\
\lambda\geq 2, \ \xi \leq 1,
\end{equation}
with $C$ independent of $\xi,\ R,\ \lambda$ (for this, it is
convenient to use a separable test function of the form
$\eta(r)\vartheta(s)$, see also \cite{delpinoScrew}). Hence, again
as in Lemma \ref{lem1}, we have that $0< U_{\xi,R,\lambda}(|x|,s)<
\mu$, if $\xi\leq |x|\leq R,\ s\in [0,\lambda]$, for all $\xi\leq
1$, $\lambda\geq 2$, provided that $R$ is sufficiently large (note
that $E(0)=\lambda \pi (R^2-\xi^2)W(0)$). Using the standard
compactness argument, letting $\xi\to 0$ and $R\to \infty$ (along a
sequence), we are left with a solution $U_\lambda$ of
(\ref{eqscrewleft}) in the infinite half strip $(0,\infty)\times
(0,\lambda)$, with zero conditions on its boundary, such that $0\leq
U_\lambda\leq \mu$ on the half strip. The latter relation leaves
open the possibility of $U_\lambda$ being identically zero. However,
$U_\lambda$ is a minimizer of the energy in (\ref{eqcellbiol}), in
the sense of (\ref{eqJer}) below (since it is the limit of a family
of minimizers, see also page 104 in \cite{danceryanCVPDE}). So, with
the help of a suitable energy competitor (see for example
(\ref{eqscrewclear}) or (\ref{eqlinear})), for any two--dimensional
ball ${B}_\frac{\lambda}{3}(q)$ of radius $\frac{\lambda}{3}$ that
is contained in $(\lambda,\infty)\times (1,\lambda-1)$, we have
\[
\int_{{B}_\frac{\lambda}{3}(q)}^{}W(U_\lambda)drds\leq C\lambda^2,
\]
with constant $C>0$ independent of large $\lambda$. If we further
assume that conditions \textbf{(b)} and \textbf{(c)} hold, noting
that
\[
\partial_r^2(\mu-U_\lambda)+\frac{1}{r}\partial_r (\mu-U_\lambda)+\left(
1+\frac{\lambda^2}{\pi^2r^2}\right)\partial_s^2(\mu-U_\lambda)\geq
0,
\]
and that the coefficients of the elliptic operator above satisfy
\[
\frac{1}{r}\leq \lambda^{-1},\ \ 1
\leq1+\frac{\lambda^2}{\pi^2r^2}\leq 1+\pi^{-2}\ \ \textrm{on}\
B_\frac{\lambda}{3}(q),
\]
the arguments in Remark \ref{remalikakossimple} can be applied to
show that
\[
U_\lambda \to\mu,\ \textrm{uniformly\ on}\
\bar{B}_\frac{\lambda}{6}(q),\ \textrm{as}\ \lambda\to \infty.
\]
Since $q$ was any point with coordinates $r>\frac{4\lambda}{3}$ and
$s\in \left(\frac{\lambda}{3}+1, \frac{2\lambda}{3}-1 \right)$, we
deduce
 that
\[
U_\lambda \to\mu,\ \textrm{uniformly\ on}\
[2\lambda,\infty)\times\left[\frac{\lambda}{6}+1,\frac{5\lambda}{6}-1\right],\
\textrm{as}\ \lambda\to \infty.
\]
Studying the existence and asymptotic behavior of $U_\lambda$, as
$\lambda\to \infty$, assuming \emph{only} \textbf{(a')}, is left as
an interesting open problem.
\end{rem}

For future reference, let us prove here the following lemma.
\begin{lem}\label{lemhelly}
Suppose that $W\in C^2$ satisfies \textbf{(a')},
(\ref{eqpacardW''}), and (\ref{eqpacardW'}). Let $u_R$ denote a
family of solutions to (\ref{eqgidasEq}), not necessarily global
minimizers, such that (\ref{eqmaxball}) holds. Then, we have that
$u_R$ are radial as well as  the validity of relations
(\ref{eqestimRect}), (\ref{eqmonotonicity}), (\ref{eqchen}), and
(\ref{eqclaim}), \emph{uniformly} with respect to the family $u_R$.
\end{lem}
\begin{proof}
Since $u_R$ is positive, thanks to \cite{gidas}, we have that $u_R$
is radial and that  (\ref{eqmonotonicity}) holds. Similarly to Lemma
\ref{lem1}, the functions $U_R$, defined through (\ref{eqUR}),
satisfy (\ref{eqURi}) for some $V$ which solves (\ref{eqhellyV})
with $0\leq V(s)\leq \mu$, $s\geq 0$.

We claim that $V$ is nontrivial. In the case where $W'(0)<0$, this
is clear. If $W'(0)=0$ and $W''(0)<0$ (keep in mind
(\ref{eqpacardW''})), we argue  as follows. Let $\lambda_R$,
$\varphi_R$ be as in (\ref{eqhellyphiR}). In view of
(\ref{eqpacardW'}), we have
\[
W'(t)\leq -ct,\ \ t\in \left[0,\frac{\mu}{2} \right],
\]
for some $c>0$ (necessarily $c\leq -W''(0)$). Observe that the
functions $\tau \varphi_R$, with $\tau \in \left[0,\frac{\mu}{2}
\right]$, satisfy
\[
  -\Delta (\tau\varphi_R)+W'(\tau\varphi_R)  \leq  \frac{\lambda_1}{R^2}\tau\varphi_R-c\tau\varphi_R\leq 0 \
  \ \textrm{in}\ B_R,
\]
if $R>\sqrt{\frac{\lambda_1}{c}}$ (we also used that
$\varphi_R(x)\leq\varphi_R(0)=1$, $x\in B_R$). Consider a ball $B_R$
with $R>3\sqrt{\frac{\lambda_1}{c}}$ and another ball
$B_{\sqrt{\frac{4\lambda_1}{c}}}(p)\subset B_R$ such that they touch
at one point on $\partial B_R$. By Serrin's sweeping technique (see
the references in the first proof of Theorem \ref{thmmine} below),
keeping in mind that $u_R'(R)<0$ (by Hopf's lemma), it follows that
\[
u_R(x)\geq \frac{\mu}{2}\varphi_{\sqrt{\frac{4\lambda_1}{c}}}(x-p),\
\ x\in B_{\sqrt{\frac{4\lambda_1}{c}}}(p).
\]
(In fact, since $u_R$ is radially symmetric, the above bound holds
for all $p\in B_R$ such that $\textrm{dist}(p,\partial
B_R)=2\sqrt{\frac{\lambda_1}{c}}$). This lower bound certainly
ensures that $V$ is nontrivial. Then, by the strong maximum
principle, we deduce that
\begin{equation}\label{eqhellyMax}
0<V(s)<\mu,\ \ s>0.
\end{equation}

On the other hand, from \textbf{(a')}, (\ref{eqpacardW'}), and the
phase portrait of the ordinary differential equation (see for
example \cite{arnold}), the only solution of (\ref{eqhellyV}) which
satisfies (\ref{eqhellyMax}) is $\textbf{U}$, as described in
(\ref{eqU}). By the uniqueness of the limiting function, we infer
that (\ref{eqURi}) holds for $R\to \infty$. So, we have proven that
 (\ref{eqclaim}) and in turn (\ref{eqestimRect}), (\ref{eqchen}) hold for each such
family of solutions.

The fact that they hold \emph{uniformly} with respect to the family
$\{u_R\}$ follows plainly from the observation that every such
family is uniformly bounded in $C^2(\bar{B}_R)$ with respect to $R$.
The latter property follows from the fact that $0<u_R<\mu$ in $B_R$
and a standard bootstrap argument involving elliptic regularity (the
gradient bounds for $u_R$ follow from elliptic estimates
\cite{Gilbarg-Trudinger} applied on balls of radius one covering
$B_R$).

The proof of the lemma is complete.
\end{proof}

An extension of Lemma \ref{lem1} can be shown, allowing the
possibility $W'(0)\geq 0$, provided that the potential $W$
satisfies:
\begin{description}
  \item[(a'')] There exist constants $\mu_-\leq0$ and $\mu>0$ such that
  \[
0=W(\mu)<W(t),\  t\in [\mu_-,\mu),\ \ W(t)\geq 0,\ t\in \mathbb{R},
  \]
  \[
W(2\mu_--t)\geq W(t), \ t\in [\mu_-,\mu]\ \textrm{or}\ W'(t)<0,\
t<\mu_-.
  \]
\end{description}
Note that \textbf{(a'')} reduces to \textbf{(a')} when $\mu_-=0$. We
point out that the existence of $\textbf{U}$, as in (\ref{eqU}),
also holds under \textbf{(a'')}.

Below, we state such a result which seems to be new and of
independent interest. In particular, it will be useful in Sections
\ref{secmatano} and \ref{secfarina}.
\begin{lem}\label{lem1Sign}
Assume that $W\in C^2$ satisfies condition \textbf{(a'')}. Let
$\epsilon\in (0,\mu)$ and $D>D'$, where $D'$ is as in (\ref{eqD}).
Then, there exists a positive constant $R'>D$, depending only on
$\epsilon$, $D$, $W$, and $n$, such that there exists a global
minimizer $u_R$ of the energy functional in (\ref{eqenergy}) which
satisfies (\ref{eqmaxball}), (\ref{eqestimRect}), and
(\ref{eqcaffaLemmmine}), provided that $R\geq R'$. (As before, we
assume that $W$ has been appropriately extended outside of a large
compact interval). (We have chosen to keep some of the notation from
Lemma \ref{lem1}).
\end{lem}
\begin{proof}
The existence of a minimizer $u_R$, which solves (\ref{eqgidasEq}),
and satisfies
\[
\mu_-<u_R(x)<\mu,\ x\in B_R,
\]
 follows as in the proof of Lemma \ref{lem1}.
 The main difference
with the proof of Lemma \ref{lem1}
 is that the above relation does not exclude the possibility of the
minimizer $u_R$  taking non-positive values. In particular, the
method of moving planes (see \cite{brezis}, \cite{dancePlanes},
\cite{gidas}) is not applicable in order to show  that $u_R$ is
radially symmetric and decreasing. (Nevertheless, it is known that
nonnegative solutions of (\ref{eqgidasEq}), with $n\geq 2$, are
actually positive in $B_R$ and so the method of moving planes is
still applicable in that situation, see \cite{polacik} and the
references therein). Not all is lost however. As we have already
remarked in the proof of Lemma \ref{lem1}, if $n\geq 2$, the
stability of $u_R$ (as a global minimizer) implies that it is
radially symmetric, see Lemma 1.1 in \cite{alikakosbates},  Remark
3.3 in \cite{cabreDiscrete}, Proposition 2.6 in
\cite{danceryanCVPDE}; for an elegant proof that exploits the fact
that $u_R$ is a global minimizer, see Corollary II.10 in
\cite{lopez} (see also \cite{gurtin} and Appendix C in
\cite{willem}). In \cite{comte}, see also Proposition 10.4.1 in
\cite{cazevane} and Proposition 3.4 in \cite{ni}, it has
additionally been shown that stable solutions have constant sign,
and hence are radially monotone by the method of moving planes. For
the reader's convenience, we will show that $u_R(r)$ is a decreasing
function of $r$, namely that (\ref{eqmonotonicity}) holds true, by a
far more elementary argument. In view of (\ref{eqgood}), which still
holds for the case at hand (by virtue of radial symmetry alone), it
suffices to show that $u_R'(r)\neq 0,\ r\in (0,R]$. We will follow
the part of the proof of Lemma 2 in \cite{cabreRadial} which dealt
with problem (\ref{eqentire}) with $n\geq 3$ (see also Proposition
1.3.4 in \cite{dupaigneBook}), and in fact show that it continues to
apply for $n\leq 2$. To this end, we have not been able to adapt the
approach of Lemma 1 in \cite{alikakosSimpson}, which basically
consists in multiplying (\ref{eqcabrelinearized}) below by
$V^+\equiv\max\{V,0\}\in W^{1,2}(B_R)$ and integrating the resulting
identity  by parts over $B_R$, since in the problem at hand
$V(R)=u_R'(R)$ may be positive. Let
\[V\equiv u_R',\] and suppose, to the contrary, that
$V(R_0)=0$ for some $R_0\in (0,R]$. We will show that the function
\begin{equation}\label{eqVtilda}
\tilde{V}(r)=\left\{\begin{array}{ll}
                      V(r), & r\in [0,R_0], \\
                       &  \\
                      0 & r\in [R_0,R],
                    \end{array}
 \right.
\end{equation}
belonging in $W^{1,2}_0(B_R)$, satisfies
\begin{equation}\label{eqbilinear}
\int_{B_R}^{}\left\{|\nabla
\tilde{V}|^2+W''(u_R)\tilde{V}^2\right\}dx <0,
\end{equation}
which  clearly contradicts the stability of $u_R$. Differentiating
(\ref{eqgidasEq}) with respect to $r$, we arrive at
\begin{equation}\label{eqcabrelinearized}
-\Delta V+W''(u_R)V+\frac{n-1}{r^2}V=0,\ \ x\in B_R\backslash\{0\}.
\end{equation}
Let $\zeta$ be a smooth  function such that
\[
\zeta(t)=\left\{\begin{array}{ll}
                  0, & t\in [0,1], \\
                    &   \\
                  1, & t\in [2,\infty).
                \end{array}
 \right.
\]
Multiplying (\ref{eqcabrelinearized}) by $\zeta
\left(\frac{r}{\varepsilon} \right)V(r)$, with $\varepsilon>0$
small, and integrating the resulting identity by parts over
$B_{R_0}$ (recall that $V(R_0)=0$), we find that
\begin{equation}\label{eqcabrecomplex}
\int_{B_{R_0}}^{}\left\{\zeta \left(\frac{r}{\varepsilon}
\right)|\nabla V|^2+\frac{1}{\varepsilon}V\zeta'
\left(\frac{r}{\varepsilon} \right)\left(\frac{x}{r}\cdot\nabla
V\right)+\zeta \left(\frac{r}{\varepsilon} \right)W''(u_R)V^2+\zeta
\left(\frac{r}{\varepsilon} \right)\frac{n-1}{r^2}V^2\right\}dx=0.
\end{equation}
Note that \[ \left|\int_{B_{R_0}}^{}\frac{1}{\varepsilon}V\zeta'
\left(\frac{r}{\varepsilon} \right)\left(\frac{x}{r}\cdot\nabla
V\right)dx\right|\leq
C\varepsilon^{-1}\int_{\varepsilon}^{2\varepsilon}r^{n-1}dr\to 0\ \
\textrm{as}\ \ \varepsilon\to 0,\] since the constant $C>0$ does not
depend on $\varepsilon$. (Note that we have silently assumed that
$N\geq 2$, since in the case $N=1$  we can plainly multiply
(\ref{eqcabrelinearized}) by $V$ and then  integrate by parts over
$(-R_0,R_0)$). So,  letting $\varepsilon\to 0$ in
(\ref{eqcabrecomplex}), and employing Lebesgue's dominated
convergence theorem (see for instance page 20 in
\cite{evansGariepy}), it readily follows that
\[
\int_{B_{R_0}}^{}\left\{|\nabla
V|^2+W''(u_R)V^2+\frac{n-1}{r^2}V^2\right\}dx=0,
\]
where in order to obtain  the last term we used that $|V(r)|\leq
C'r$, $r\in [0,R]$, with constant $C'>0$ depending only on $R$ (keep
in mind that $u_R\in C^2[0,R]$ with
$u_R''(0)=\frac{1}{n}W'\left(u_R(0) \right)$, see for instance page
72 in \cite{walter}). From the above relation, via (\ref{eqVtilda}),
we get (\ref{eqbilinear}). We have thus arrived at the desired
contradiction. Consequently, the monotonicity relation
(\ref{eqmonotonicity}) also holds for the more general case at hand.
The rest of the argument follows word by word the proof of Lemma
\ref{lem1}, and is therefore omitted.

 The proof of the lemma is complete.
\end{proof}

\begin{rem}\label{remMonoto}
Suppose that $u_R$ is as in Lemma \ref{lem1} or Lemma
\ref{lem1Sign}, and $E_R$ as defined in (\ref{eqER1}). From
(\ref{eqER'}), it follows that
\[
E_R(r)<E_R(0)=-W\left(u_R(0) \right)<0,\ \ r\in [0,R],
\]
i.e.,
\begin{equation}\label{eqmodicaRadial}
\frac{1}{2}[u_R'(r)]^2<W(u_R),\ \ r\in (0,R],
\end{equation}
recall \textbf{(a')} and that $u_R'(0)=0$, see also Remark 4 in
\cite{alikakosSimpson} for a related discussion. In passing, we note
that every bounded solution of (\ref{eqentire}) satisfies
\begin{equation}\label{eqmodicafarina}
\frac{1}{2}|\nabla u|^2\leq W(u),\ \ x\in \mathbb{R}^n,
\end{equation}
provided that $W$ is nonnegative.
 The proof of this gradient  bound, originally due to L.
Modica \cite{modica}, is much more complicated than that of its
radially symmetric counterpart (\ref{eqmodicaRadial}). We refer the
interested  reader to \cite{cafamodica} and Lemma 4.1 in
\cite{chenDG} (see also Proposition \ref{procarbouModica} herein).
In turn, making use of the gradient bound (\ref{eqmodicaRadial}), we
can establish the monotonicity formula
\begin{equation}\label{eqmodica}
\frac{d}{dr}\left(\frac{1}{r^{n-1}}\int_{B_r}^{}\left\{\frac{1}{2}|\nabla
u_R|^2+ W(u_R) \right\}dx\right)>0,\ \ r\in (0,R),
\end{equation}
see \cite{alikakosBasicFacts} for a modern approach as well as the
older references therein which include \cite{cafamodica}. In
passing, we note that a similar monotonicity formula holds true for
solutions of (\ref{eqentire}), and a weaker one (with the exponent
$n-1$ replaced by $n-2$) holds in the case of the corresponding
systems, see again \cite{alikakosBasicFacts} and the references
therein or \cite{caffareliLin}. Now, making use of (\ref{eqJcomp})
and the above relation, we find that
\[
\frac{1}{K^{n-1}}\int_{B_K}^{}\left\{\frac{1}{2}|\nabla u_R|^2+
W(u_R) \right\}dx < C_1\ \ \ \forall \  K\in (0,R),\ \ R\geq 2.
\]
We have therefore provided a proof (of a sharper version) of
(\ref{eqK}). It also follows from (\ref{eqmodica}) that
$R^{1-n}J(u_R;B_R)$ remains bounded from below by some positive
constant, as $R\to \infty$ (compare with (\ref{eqJcomp})).
If $W''(\mu)>0$, making use of (\ref{eqclaim}), it is not hard to
determine a constant to which $R^{1-n}J(u_R;B_R)$ converges as $R\to
\infty$ (see \cite{andre2} and \cite{guiAnnals}), recall also the
last part of Remark \ref{rembelowNew} (in order to avoid confusion,
we point out that we have \emph{not} shown that the latter function
is increasing in $R$). In this regard, we also refer to Theorem 7.10
in \cite{braides} where functionals of the form (\ref{eqenergy}) are
shown to converge (in an appropriate variational sense) to
functionals involving the perimeter of the domain.
\end{rem}

\begin{rem}\label{remsigal}
Here, for completeness, we sketch an argument related to the proof
of Lemma \ref{lem1Sign}.
 By (\ref{eqmaxball}), elliptic estimates (see \cite{Gilbarg-Trudinger}), and a
standard compactness argument, it follows readily that $u_R$
converges, up to a subsequence $R_i\to \infty$, uniformly on compact
subsets of $\mathbb{R}^n$ to a radially symmetric solution   $U$ of
(\ref{eqentire}) such that $0\leq U(x)\leq \mu,\ x\in \mathbb{R}^n$.
Moreover, arguing as in page 104 of \cite{danceryanCVPDE}, this
solution is a global minimizer of (\ref{eqentire}) in the sense of
(\ref{eqJer}) below, with $\Omega =\mathbb{R}^n$, see also
\cite{cabre}, \cite{jerison}.

On the other hand, it is known that (\ref{eqentire}), for \emph{any}
$W\in C^2$, does not have nonconstant bounded, radial global
minimizers (see \cite{villegasCpaa}). This property is also related
to the nonexistence of nonconstant  ``bubble'' solutions to
(\ref{eqentire}) with $W\geq 0$ vanishing nondegenerately at a
finite number of  points, namely solutions that tend to one of these
points as $|x|\to \infty$,
 see Theorem 2 in \cite{carbou}, Chapter 4 in \cite{sigal} and
recall Remark \ref{rempoho} (keep in mind that stable solutions of
(\ref{eqentire}) are radially monotone and tend to a local or global
minimum of $W$, as $r\to \infty$, see \cite{cabreRadial}). In
passing, we note that if $n\leq 10$ then nonconstant radial
solutions of (\ref{eqentire}), with $W\in C^2$ arbitrary, are
unstable (see \cite{cabreRadial}). Under certain assumptions on $W$,
satisfied by the Allen-Cahn potential (\ref{eqAllen}) for example,
it was shown in \cite{guiRadial} (see also \cite{birindeli}) that
nonconstant radial solutions of (\ref{eqentire}) tend in an
oscillatory manner  to zero as $r \to \infty$ and thus are unstable.
 More generally,
the nonexistence of nonconstant finite energy solutions to
(\ref{eqentire}) with $W\geq 0$ holds, see \cite{alikakosBasicFacts}
or \cite{henryCMP} where this property is refereed to as a theorem
of Derrick and Strauss. Related nonexistence results for nonnegative
solutions can be found in Sections \ref{secBoundsonEntire} and
\ref{secmatano} herein.

 Obviously $U\equiv \mu$ (recall \textbf{(a')})
and, by the uniqueness of the limit, the convergence holds for
\emph{all} $R\to \infty$.   We conclude that, given any $K>1$, we
have $ u_R\to \mu, \ \textrm{uniformly\ in}\ B_K, \ \textrm{as}\
R\to \infty$. The main advantage of this approach is that it
continues to work  when (\ref{eqgidasEq}) is replaced by $\Delta
u=F_R(|x|,u)$, with a suitable $F_R(|x|,u)$ which converges
uniformly over compact sets of $[0,\infty)\times \mathbb{R}$  to a
$C^1$ function $F(u)$ (the point being that $\frac{d}{dr}F_R(r,u)$
may be negative somewhere, and (\ref{eqmonotonicity}) may fail in
$B_R$).
\end{rem}
The following lemma is motivated from Lemma 3.3 in
\cite{cafareliPacard}.
\begin{lem}\label{lemcafa}
Assume that $W$ satisfies conditions \textbf{(a'')} and
(\ref{eqW''}). Let $\epsilon\in (0,\mu)$ be any number such that
\begin{equation}\label{eq115}W''(t)\geq 0 \ \ \textrm{on}\ \ [\mu-\epsilon,\mu].
\end{equation}
Then, the global minimizers $u_R$ that are   provided by Lemmas
\ref{lem1} and \ref{lem1Sign} satisfy
\begin{equation}\label{eqcafa}
-W'\left(u_R(0) \right)\leq \tilde{C}R^{-2},
\end{equation}
where the constant $\tilde{C}>0$ depends only on $n$,
 provided that $R\geq R'$, where $R'$ is as in the latter lemmas.
\end{lem}
\begin{proof}
Let $D$ be as in the assertions of Lemmas \ref{lem1} and
\ref{lem1Sign}. Thanks to (\ref{eqmaxball}), (\ref{eqestimRect}),
(\ref{eqgidasEq}), (\ref{eqmonotonicity}), and (\ref{eq115}),  we
have
\[
\Delta u_R=W'\left(u_R\right)\leq W'\left(u_R(0)\right) \ \
\textrm{on}\ \ \bar{B}_{(R-D)},
\]
if $R\geq R'$.

For such $R$, let $Z_R$ be the solution of
\[
\Delta Z_R=W'\left(u_R(0)\right)\ \ \textrm{in}\ B_{(R-D)};\ \ Z_R=0
\ \ \textrm{on}\ \partial B_{(R-D)}.
\]
By scaling, one finds that
\[
\max_{|x|\leq R-D} Z_R(x)=Z_R(0)=-z(0)W'\left(u_R(0)\right)(R-D)^2,
\]
for $R\geq R'$, where $z$ is the solution of
\[
\Delta z=-1\ \ \textrm{in}\ B_{1};\ \ z=0 \ \ \textrm{on}\
\partial B_{1}.
\]
By the maximum principle, we deduce that
\[
Z_R(x)\leq u_R(x)<\mu,\ \ x\in B_{(R-D)},\ \ R\geq R'.
\]
In particular, by setting  $x=0$ in the above relation, we get
(\ref{eqcafa}).

The proof of the lemma is complete.
\end{proof}

\begin{rem}\label{remnaka1}
In the special case where $W$ satisfies (\ref{eqpacardW'}),
$W'(0)=0$, $W''(0)<0$, $W'(\mu)=0$, and $W''(\mu)>0$, the estimate
of Lemma \ref{lemcafa} becomes that of Lemma 3.1 in
\cite{nakashimaNonradial}.
\end{rem}

 Under conditions \textbf{(a'')} and (\ref{eqW''}), the global
minimizers that are provided by Lemmas \ref{lem1} and \ref{lem1Sign}
are \emph{asymptotically} stable, if $R$ is sufficiently large. This
property is a direct consequence of the following proposition, which
will play an essential role in our proofs of Theorems
\ref{thmmatano}-\ref{thmmatano2} and Propositions
\ref{promatanointerior1}-\ref{promatanointerior2} below, as well as
in our first proof of Theorem \ref{thmmine}.
\begin{pro}\label{proUniq}
Assume that \textbf{(a'')} and (\ref{eqW''}) hold, then any solution
of (\ref{eqgidasEq}) which satisfies  (\ref{eqmaxball}),
(\ref{eqestimRect}), and (\ref{eqURi}) with $V=\textbf{U}$ for
\emph{all} $R\to \infty$ (keep in mind (\ref{eqUR})) is linearly
non--degenerate if $R$ is sufficiently large.
\end{pro}
\begin{proof}
We remark that in the case where $W''(\mu)>0$, we have already  seen
in Remark \ref{remmonotbelowLem12} that any such solution is in fact
asymptotically stable for large $R$.

To prove this proposition, we will argue once more by contradiction.
Suppose that there exists a sequence $R\to \infty$ and solutions
$u_R$ of (\ref{eqgidasEq}), as in the assertion of the proposition,
such that there are nontrivial solutions $\varphi_R$ of
(\ref{eqkernelstarb}) with $\mu_R=0$. Without loss of generality, we
may assume that the normalization (\ref{eqnormal}) holds.
 This time, the $\varphi_R$'s
may change sign but they are still radially symmetric (see again
\cite{harauxPolacik}, \cite{linNi}, and note that (\ref{eqURi})
implies that $u_R'(R)<0$ for large $R$). By Lemma 2.1 in
\cite{korman}, the following identity holds
\begin{equation}\label{eqkorman}
R^{-n}\int_{0}^{R}W'\left(u_R(r)
\right)\varphi_R(r)r^{n-1}dr=-\frac{1}{2}u_R'(R)\varphi_R'(R).
\end{equation}
In order to make the presentation as self-contained as possible, let
us mention that a direct proof of (\ref{eqkorman}) can be given by
observing that the function
\[
\zeta(r)=r^n\left[u'\varphi'-W'(u)\varphi
\right]+(n-2)r^{n-1}u'\varphi,\ \ r\in [0,R],
\]
(having dropped the subscripts for the moment), introduced in
\cite{tang}, satisfies
\[
\zeta'(r)=-2W'(u)\varphi r^{n-1},\ \ r\in (0,R);
\]
see also Chapter 1 in \cite{kormanBook} (a perhaps simpler proof was
given in \cite{korman2}). Since $W'(\mu)=0$, by (\ref{eqestimRect}),
we deduce that
\[
R^{-n}\int_{0}^{R}\left[W'\left(u_R(r) \right)\right]^2r^{n-1}dr\to
0\ \ \textrm{as}\ R\to \infty.
\]
 Hence, recalling (\ref{eqnormal}), via the
Cauchy--Schwarz inequality, we find that the lefthand side of
(\ref{eqkorman}) tends to zero as $R\to \infty$ (along the
sequence). On the other side, from our assumption that
(\ref{eqURi}), with $V=\textbf{U}$, holds for all $R\to \infty$, we
know that
\[
u_R'(R)\to -\sqrt{2W(0)}<0\ \ \textrm{as}\ \ R\to \infty.
\]
So, from (\ref{eqkorman}), we get that $\varphi_R'(R) \to 0$ as
$R\to \infty$ (along the contradicting sequence). By the continuous
dependence theory for systems of ordinary differential equations
\cite{arnold,walter} (applied to $\varphi_R(R-r)$ in
(\ref{eqkernelstarb})), making use of (\ref{eqURi}) with
$V=\textbf{U}$ for all $R\to \infty$, we infer that for any $D>0$ we
have
\begin{equation}\label{eqcontra1}
\left|\varphi_R(r)\right|+\left|\varphi_R'(r)\right|\leq
\frac{1}{2},\ \  r\in [R-D,R],
\end{equation}
provided that $R$ is sufficiently large (along this sequence). On
the other hand, if $D$ is chosen so that $W''(u_R)\geq 0$ on
$\bar{B}_{(R-D)}$, which is possible by  (\ref{eqW''}),
(\ref{eqmaxball}) and (\ref{eqestimRect}), it follows from
(\ref{eqkernelstarb}) with $\mu_R=0$ that
\[
\varphi_R\Delta \varphi_R=W''(u_R)\varphi_R^2\geq 0 \ \ \textrm{on}
\ \bar{B}_{(R-D)},
\]
for such large $R$. In particular, we find that $\varphi_R$ cannot
vanish in $B_{(R-D)}\backslash \{0\}$ (using the radial symmetry,
and integrating by parts over $B_z$ if $\varphi_R(z)=0$).
Furthermore, it cannot vanish at the origin by virtue of the
uniqueness theorem for ordinary differential equations, which still
holds despite of the singularity at $r=0$ (see \cite{fifesaddle},
\cite{peletierSerrin}, \cite{walter}). Therefore, without loss of
generality, we may assume that $\varphi_R> 0$ in $B_{(R-D)}$. Hence,
the positive function $\varphi_R$ is subharmonic in $B_{(R-D)}$, and
not greater than $\frac{1}{2}$ on $\partial B_{(R-D)}$ (recall
(\ref{eqcontra1})), for $R$ large along the contradicting sequence.
The maximum principle (see for example Theorem 2.3 in
\cite{Gilbarg-Trudinger}) yields that $0<\varphi_R\leq\frac{1}{2}$
on $\bar{B}_{(R-D)}$. The latter relation together with
(\ref{eqcontra1}) clearly contradict (\ref{eqnormal}), and we are
done.

The proof of the proposition is complete.
\end{proof}
\begin{rem}\label{remshipoho}
We note that identity (\ref{eqkorman}) has been generalized in Lemma
2.3 in \cite{shiOyang} for the case of solutions of (\ref{eqEq}) on
an arbitrary smooth, bounded star-shaped domain (see also Theorem
1.6 in \cite{kormanBook}). This leaves open the possibility that
Proposition \ref{proUniq} above can be generalized accordingly.
\end{rem}

 The following corollary is a simple consequence of the
maximum principle.
\begin{cor}\label{corbrezis}
If $W''(\mu)>0$, then the solutions provided by Lemmas \ref{lem1}
and \ref{lem1Sign} satisfy
\[
\mu-u_R(r)\leq C_5 e^{-C_6(R-r)},\ \ r\in [0,R-2D] \ \ \textrm{for}\
\ R\geq R',
\]
and some positive  constants $C_5, C_6$, depending on $W$ and $n$.
\end{cor}
\begin{proof}
Let $\varphi\equiv \mu-u_R$, where $u_R$ is as in Lemma \ref{lem1}
or \ref{lem1Sign}. By virtue of \textbf{(a')}, (\ref{eqmaxball}),
and (\ref{eqestimRect}), we can choose $\epsilon$ sufficiently small
such that that \[W'(u_R(x))\leq \frac{W''(\mu)}{2}\left(u_R(x)-\mu
\right),\ \ x\in B_{(R-D)},\] provided that $R\geq R'$, where $D,\
R'$ are as in the previously mentioned lemmas (having increased the
value of $R'$, if necessary, but still depending only on $\epsilon$,
$D$, $W$, and $n$). It follows from (\ref{eqgidasEq}) that
\[
-\Delta \varphi+\frac{W''(\mu)}{2}\varphi\leq 0\ \ \textrm{in} \
B_{(R-D)},\ R\geq R'.
\]
 Now, the desired assertion of the
corollary follows from a standard comparison argument, see Lemma 2
in \cite{bethuelCVPDE} or Lemma 4.2 in \cite{fife-arma} (see also
Lemma 2.5 in \cite{duNakas} and Lemma 5.3 in \cite{guiAnnals}).

The proof of the corollary is complete.
\end{proof}
\begin{rem}\label{remdancerwei}
A special case of  Theorem 2.1 in \cite{dancerSpikes} shows that the
assertion of Corollary \ref{corbrezis} above can be considerably
refined to
\[
\lim_{R\to \infty}R^{-1}\ln\left(\mu-u_R(Rs)
\right)=-(1-s)\sqrt{W''(\mu)},\ \ \forall\ s\in [0,1],
\]
see also \cite{perthame}.
\end{rem}
%

\subsection{Proof of Theorem \ref{thmmine}}\label{subsectionproof} Once Lemma \ref{lem1} is
established, the proof of Theorem \ref{thmmine} proceeds in a rather
standard way. We will present two different approaches, and leave it
to the reader's personal taste. The first approach is based on the
method of upper and lower solutions, while the second one is based
on variational arguments.

\texttt{First proof of Theorem \ref{thmmine}:}  We will adapt an
argument from the proof of Theorem 2.1 in \cite{dancer-fields}, and
prove existence of the desired solution to (\ref{eqEq}) by the
method of upper and lower solutions (see for instance \cite{matano},
\cite{sattinger}). Let $\epsilon\in (0,\mu)$ and $D>D'$, where $D'$
is as in (\ref{eqD}), and $R'$ be the positive constant, depending
only on $\epsilon$, $D$, $W$, and $n$, that is described in Lemma
\ref{lem1}. Suppose that $\Omega$ contains a closed ball of radius
$R'$. We use $\bar{u}(x)\equiv \mu,\ x\in \Omega$, as an upper
solution (recall that $W'(\mu)=0$), and as lower solution the
function
\begin{equation}\label{eqlower}
\underline{u}_P(x)\equiv\left\{\begin{array}{ll}
                       u_{\textrm{dist}(P,\partial\Omega)}(x-P), & x\in B_{\textrm{dist}(P,\partial\Omega)}(P),  \\
                         &    \\
                       0, & x\in \Omega\backslash B_{\textrm{dist}(P,\partial\Omega)}(P),
                     \end{array}
 \right.
\end{equation}
for some $P\in \Omega_{R'}$ (considered fixed for now), where
$u_{R}$ is as in Lemma \ref{lem1} (here we used that $W'(0)\leq 0$
and Proposition 1 in \cite{Berestyckilion} to make sure that
$\underline{u}_P$ is a lower solution, see also Proposition 1 in
\cite{kurata}). In view of (\ref{eqmaxball}) and
(\ref{eqestimRect}), keeping in mind that
\begin{equation}\label{eqP}
\textrm{dist}(P,\partial \Omega)>R',
\end{equation}
it follows  that
\begin{equation}\label{eqmaxballunder}
\underline{u}_P(x)<\bar{u}(x)\equiv \mu,\ x\in \Omega, \ \
\textrm{and}\ \ \mu-\epsilon<\underline{u}_P(x),\ \ x\in
B_{\left(\textrm{dist}(P,\partial \Omega)-D\right)}(P).
\end{equation}
In the case where $\Omega$ is bounded,  it follows immediately from
the method of monotone iterations, see Theorem 2.3.1 in
\cite{sattinger} (this is the only place where we use the smoothness
of $\partial \Omega$, see however Remark \ref{remcafalipschitz}
below), that there exists a solution $u \in C^2(\Omega)\cap
C(\bar{\Omega})$ of (\ref{eqEq}) such that
\begin{equation}\label{eqPsliding}
\underline{u}_P(x)<  u(x)< \bar{u}(x)\equiv \mu,\ \ x\in \Omega,
\end{equation}
(keep in mind that the solution $u$ depends on the choice of the
center $P$). The same property can also be shown in the case where
$\Omega$ is unbounded, by exhausting it with a sequence of bounded
domains, see
 Theorem 2.10 in
\cite{niIndiana} (also recall our discussion following the statement
of Theorem \ref{thmmine}), see also \cite{noussair,ogata}. We have
thus established the existence of a solution $u$ to (\ref{eqEq})
that satisfies (\ref{eq12}), and  the lower bound (\ref{eq14+}) in
the region $B_{\left(\textrm{dist}(P,\partial \Omega)-D\right)}(P)$
(recall (\ref{eqmaxball}), (\ref{eqestimRect}), and (\ref{eqP})), or
equivalently in $P+B_{\left({\textrm{dist}(P,\partial
\Omega)-D}\right)}\supseteq P+B_{(R'-D)}$. It remains to show that
the latter lower bound is valid in $\Omega_{R'}+B_{(R'-D)}$.
 To this end, as we will see in a
moment, in a more complicated setting, we can translate the
compactly supported function $u_{R'}(x-Q)$, for $x\in B_{R'}(Q)$
(extended by zero otherwise), $Q\in \Omega_{R'}$, starting from
$Q=P$, by means of Serrin's sweeping technique and the sliding
method, to obtain the lower bound (\ref{eq14+}) via
(\ref{eqestimRect}).

In the remainder of this proof, unless specified otherwise, we will
assume that relation (\ref{eqW''}) holds.
  Observe
that as we vary the point $P$ in $\Omega_{R'}$, assuming for the
moment that $\Omega_{R'}$ has a single arcwise connected component,
the functions $\underline{u}_P$'s continue to be lower solutions of
(\ref{eqEq}). Consequently, by Serrin's sweeping principle (see
\cite{clemente,dancerProcLondon,Jang,sattinger}, and the last part
of the proof of Proposition \ref{propacard} herein), we deduce that
\begin{equation}\label{eqsweep}\underline{u}_Q(x)<u(x),\ \ x\in
\Omega,\ \ \forall \  Q\in \Omega_{R'},\end{equation} (see also the
proof of Lemma 3.1 in \cite{clemente}, and note that
$\underline{u}_Q$ varies continuously with respect to $Q$ because of
the connectedness of $\partial \Omega$; by the implicit function
theorem \cite{malchiodicambridge}, we obtain that $u_R$ varies
smoothly with respect to $R$, provided that $R$ is sufficiently
large so that Proposition \ref{proUniq} is applicable).

In fact, this is more in the spirit of the celebrated sliding method
\cite{berestyckiBrazil}: Let $\gamma(s),\ s\in [0,1]$, be a smooth
curve, lying entirely in $\Omega_{R'}$, such that $\gamma(0)=P$ and
$\gamma(1)=Q$ ($Q\in \Omega_{R'}$ arbitrary). It follows from
(\ref{eqPsliding}) that $\underline{u}_{\gamma(0)}<u$ in $\Omega$.
We intend to show that \[\underline{u}_{\gamma(s)}\leq u\
\textrm{in}\  \Omega\ \textrm{for\ all}\ s\in [0,1].\] Call
\[
t_*=\sup \left\{t\in [0,1]\ :\ \underline{u}_{\gamma(s)}\leq u\
\textrm{on}\ \bar{\Omega}\ \ \forall\ s\in [0,t] \right\}.
\]
It is easily seen that
\[
\underline{u}_{\gamma(t_*)}\leq u\ \ \textrm{on}\ \bar{\Omega}.
\]
Suppose, to the contrary, that $t_*<1$. Then, there exists a
sequence $t_j>t_*$, satisfying $t_j\to t_*$, and a sequence $x_j\in
\bar{\Omega}$, such that
\[
\underline{u}_{\gamma(t_j)}(x_j)>u(x_j).
\]
Clearly, we have that
\[
x_j\in B_{\textrm{dist}\left(\gamma(t_j),\partial \Omega
\right)}\left(\gamma(t_j)\right).
\]
We may therefore assume that
\[
x_j\to x_*\in \bar{B}_{\textrm{dist}\left(\gamma(t_*),\partial
\Omega \right)}\left(\gamma(t_*)\right).
\]
Furthermore, we obtain that
\[\underline{u}_{\gamma(t_*)}(x_*)=\lim_{j\to \infty}
\underline{u}_{\gamma(t_j)}(x_j)\geq \lim_{j\to\infty}u(x_j)=u(x_*).
\]
Hence, we get that $\underline{u}_{\gamma(t_*)}(x_*)=u(x_*)$. Since
$\underline{u}_{\gamma(t_*)}=0$ on $\partial
{B}_{\textrm{dist}\left(\gamma(t_*),\partial \Omega
\right)}\left(\gamma(t_*)\right)$, while $u>0$ in $\Omega$, we infer
that $x_*\in {B}_{\textrm{dist}\left(\gamma(t_*),\partial \Omega
\right)}\left(\gamma(t_*)\right)$ or $x_*\in \partial \Omega$, and
is a point of local minimum for $u-\underline{u}_{\gamma(t_*)}$.
But, we have
\begin{equation}\label{eqsmallvol}
\Delta \left(u- \underline{u}_{\gamma(t)}\right)+q_t(x) \left(u-
\underline{u}_{\gamma(t)}\right)\leq 0\ \ \textrm{weakly\ in}\
\Omega, \ \textrm{with}\ q_t\in L^\infty(\Omega),\ t\in [0,1].
\end{equation}
If $x_*\in {B}_{\textrm{dist}\left(\gamma(t_*),\partial \Omega
\right)}\left(\gamma(t_*)\right)$, the strong maximum principle
implies that $u\equiv {u}_{\gamma(t_*)}$ therein. However, this is
not possible, since $u>{u}_{\gamma(t_*)}=0$ at points on
$\partial{B}_{\textrm{dist}\left(\gamma(t_*),\partial \Omega
\right)}\left(\gamma(t_*)\right)$ which are not on $\partial
\Omega$. In passing, we remark that a similar argument, in the case
where the radius of the sliding ball is kept fixed, appears in the
proof of Lemma 3.1 in \cite{cafareliPacard}. It remains to consider
the case where $x_*\in \partial \Omega$, and
$u>\underline{u}_{\gamma(t_*)}$ in $\Omega$. For simplicity, we will
assume that $u$ and $\underline{u}_{\gamma(t_*)}$ touch only at
$x_*$, since the general case can be treated analogously. Given
$\rho>0$, there exist $\delta,d>0$ such that
\[
u-\underline{u}_{\gamma(t_*)}\geq d\ \ \textrm{on}\
\bar{B}_{\left(\textrm{dist}\left(\gamma(t_*),\partial \Omega
\right)+\delta\right)}\left(\gamma(t_*) \right)\cap
\bar{\Omega}\setminus B_\rho(x_*),
\]
(by the imposed regularity on $\partial \Omega$, we may assume that
${B}_{\left(\textrm{dist}\left(\gamma(t_*),\partial \Omega
\right)+\delta\right)}\left(\gamma(t_*) \right)\cap {\Omega}$ and
$B_\rho(x_*)\cap  \Omega$ contain only one connected component
respectively). Hence, for $t$ close to $t_*$ (how close depends on
the smallness of $\rho>0$), we have
\[
u-\underline{u}_{\gamma(t)}\geq 0\ \ \textrm{on}\
\bar{B}_{\textrm{dist}\left(\gamma(t),\partial \Omega
\right)}\left(\gamma(t) \right)\setminus B_\rho(x_*).
\]
In turn, the latter relation clearly implies that
\[
u-\underline{u}_{\gamma(t)}\geq 0\ \ \textrm{on}\
\bar{\Omega}\setminus B_\rho(x_*).
\]
In particular, we find that $u-\underline{u}_{\gamma(t)}\geq 0$ on
the boundary of $B_\rho(x_*)\cap \Omega$ for $t-t_*>0$ small.
Decreasing the value of $\rho>0$, if necessary, we can apply the
maximum principle for small domains in (\ref{eqsmallvol}) (see
\cite{brezis}, \cite{saddlecabre3solo}, \cite{dancePlanes}), to
infer that $u-\underline{u}_{\gamma(t)}\geq 0$ in $B_\rho(x_*)\cap
\Omega$. We have thus arrived at $u-\underline{u}_{\gamma(t)}\geq 0$
on $\bar{\Omega}$ for $t-t_*>0$ sufficiently small (depending on the
smallness of $\rho>0$), which contradicts the assumption that
$t_*<1$. Consequently, we have that $t_*=1$, as desired.

The validity of the lower bound (\ref{eq14+}), over the whole
specified domain, now follows from (\ref{eqestimRect}),
(\ref{eqlower}), (\ref{eqP}), and (\ref{eqsweep}). In the case where
the domain $\Omega_{R'}$ has numerably  many arcwise connected
components, we can use the function $ \max\{\underline{u}_{P_i}, \
i=1,\cdots \} $ as a lower solution, where the
$\underline{u}_{P_i}$'s are as in (\ref{eqlower}) with each center
$P_i$ belonging to a different component of $\Omega_{R'}$. (We use
again Proposition 1 in \cite{Berestyckilion}, keep in mind that the
maximum is essentially chosen among finitely many functions). The
case where there are denumerably  many arcwise connected components
of $\Omega_{R'}$ can be treated similarly. The proof of
(\ref{eqcaffaThmmine}), which does not require assumption
(\ref{eqW''}), is postponed until  the second proof of Theorem
\ref{thmmine} that follows.

If $W''(\mu)>0$, the validity of (\ref{eq13}) for $x\in \Omega_{R'}$
follows at once from Corollary \ref{corbrezis} and relations
(\ref{eqlower}), (\ref{eqsweep}). If $\textrm{dist}(x,\partial
\Omega)\leq R'$, then plainly observe that
\begin{equation}\label{eqplaindecay}
\mu-u(x)\leq \mu=\mu e^{R'}e^{-R'}\leq\mu
e^{R'}e^{-\textrm{dist}(x,\partial \Omega)}.
\end{equation}
If relation (\ref{eqW''}) holds, then the validity of
(\ref{eqcafathm}) follows from (\ref{eqcafa}), (\ref{eqlower}), and
(\ref{eqsweep}), keeping in mind that $\mu-\epsilon\leq
\underline{u}_P(P)\leq u(P)$, via (\ref{eqW''}), implies that
$-W'\left(u(P)\right)\leq -W'\left(\underline{u}_P(P)\right)$. We
postpone  the proof of relation (\ref{eqfinal}) until  Subsection
\ref{submodel}.

The first proof of the theorem, with the exception of
(\ref{eqfinal}), is complete.\ \ \ \ \ \ $\Box$

\begin{rem}\label{remcafalipschitz}
It is stated in page 1107 of \cite{cafareliPacard}, unfortunately
without providing a reference, that the method of upper and lower
solutions works also in the case of merely Lipschitz domains (at
least for Dirichlet boundary conditions). If this is true, then our
Theorem \ref{thmmine} holds for $\Omega$ Lipschitz.
\end{rem}

\begin{rem}\label{remstable}
Since it is constructed by the method of upper and lower solutions,
we know that the obtained solution $u$ is stable (with respect to
the corresponding parabolic dynamics), see \cite{matano,sattinger},
namely the principal eigenvalue of
\[
-\Delta \varphi+W''(u)\varphi=\lambda \varphi,\ x\in \Omega;\ \
\varphi=0,\ x\in \partial \Omega,
\]
in nonnegative.
 In the case of unbounded domains, some extra care is needed in the
definition of stability, see
\cite{cabreRadial,dancerMorseTrans,dupaigneBook}.
\end{rem}

\texttt{Second  proof of Theorem \ref{thmmine}:} Assume first that
$\Omega$ is bounded. As in the proof of Lemma \ref{lem1}, there
exists a global minimizer $u_{min}$ of the energy
\[
J(v;\Omega)=\int_{\Omega}^{}\left\{\frac{1}{2}|\nabla v|^2+W(v)
\right\}dx,\ \ v\in \ W^{1,2}_0(\Omega),
\]
which furnishes a classical solution of (\ref{eqEq}) such that
$0\leq u_{min}(x)<\mu$, $x\in \Omega$. Again, by the strong maximum
principle, either $u_{min}$ is identically equal to zero or it is
strictly positive in $\Omega$. We intend to show that there exists
an $R_*>0$, depending only on $W$ and $n$, such that $u_{min}$ is
nontrivial, provided that $\Omega$ contains some closed ball of
radius $R_*$.

For the sake of our argument, suppose that $u_{min}$ is the trivial
solution. Then, motivated from Proposition 1 in \cite{alama} (see
also \cite{cabre}, \cite{lassouedmironescu} and
\cite{nousairPacific}), assuming without loss of generality that
$\bar{B}_{R+2}\subset\Omega$ for some $R>0$, we consider the
function
\begin{equation}\label{eqlinear}
Z(x)=\left\{\begin{array}{ll}
           0,   & x\in \Omega\backslash B_{R+1}, \\
            &   \\
       \mu(R+1-|x|),     & x\in B_{R+1}\backslash B_R,\\
       & \\
\mu, & x\in B_R.
         \end{array}
 \right.
\end{equation}
Since $Z\in W^{1,2}_0(\Omega)$, from the relation $J(0;\Omega)\leq
J(Z;\Omega)$, and recalling that $W(\mu)=0$, we obtain that
\begin{equation}\label{eqcabre}
J(0;B_{R+1})\leq \int_{B_{R+1}\backslash
B_R}^{}\left\{\frac{1}{2}|\nabla Z|^2+W(Z) \right\}dx\leq C_0
R^{n-1},
\end{equation}
with $C_0$ depending only on $W$ and $n$. In turn, the above
relation implies that
\[
|B_{R+1}|W(0)\leq C_0 R^{n-1},
\]
which cannot hold if $R\geq R_*$ is sufficiently large, depending on
$W$ and $n$. Consequently, the minimizer $u_{min}$ is nontrivial,
provided that $\Omega$ contains some closed ball of radius $R_*$.
From our previous discussion, we therefore conclude that $u_{min}$
 satisfies (\ref{eq12}).

Let $\epsilon\in (0,\mu)$ and $D>D'$, where $D'$ is as in
(\ref{eqD}). Suppose that $\Omega$ contains a closed ball of radius
$R'$, where $R'$ is as in the assertion of Lemma \ref{lem1}; without
loss of generality, we may assume that $R'>R_*$.
 Relation (\ref{eq14+}) now follows by
applying Lemma \ref{lemdancer} below, over every closed  ball of
radius $R'$ contained in $\Omega$, and recalling Lemma \ref{lem1}.
Note that, as in Remark \ref{remmonotbelowLem1}, the unique
continuation principle implies that
\begin{equation}\label{eqJer}u_{min}\  \textrm{minimizes}\  J(v;\mathcal{D})\  \textrm{in}\  v-u_{min}\in
W^{1,2}_0\left(\mathcal{D}\right)\  \textrm{for\ every\ smooth\
bounded\ domain}\ \mathcal{D}\subset\Omega.
\end{equation}
The case where $\Omega$ is unbounded can be treated by exhausting it
by an infinite sequence of bounded ones, where the above
considerations apply (see also \cite{nousairPacific}). The
minimizers over the bounded domains (extended by zero outside)
converge locally uniformly to a solution $u$ of (\ref{eqEq}) that
satisfies  (\ref{eq12}) (the latter solution is nontrivial by virtue
of the lower bound  $u(x)\geq \mu-\epsilon,\ x\in B_{(R'-D)}(x_0)$
for some $x_0 \in \Omega_{R'}$,  which is valid since we may assume
that each one of the bounded domains contains the same closed ball
$\bar{B}_{R'}(x_0)$). This solution of (\ref{eqEq}), on the
unbounded domain $\Omega$, found in this way, may have infinite
energy but is still  a global minimizer in the sense of Definition
1.2 in \cite{jerison}, namely satisfies (\ref{eqJer}). As before, it
satisfies (\ref{eq14+}).

The validity of (\ref{eqcaffaThmmine}) follows from
(\ref{eqcaffaLemmmine}) and Lemma \ref{lemdancer} below (applied on
every ball $B_{\textrm{dist}(x,\partial\Omega)}(x),\ x\in
\Omega_{R'}$). Similarly, if $W''(\mu)>0$, the validity of
(\ref{eq13}) follows from Corollary \ref{corbrezis}, Lemma
\ref{lemdancer} below, and the observation in (\ref{eqplaindecay}).
The validity of relation (\ref{eqcafathm}) follows in the same
manner, making use of (\ref{eqcafa}). We postpone  the proof of
relation (\ref{eqfinal}) until  Subsection \ref{submodel}.

The second proof of the theorem, with the exception of
(\ref{eqfinal}), is complete.\ \ \ \ \ \ $\Box$
\begin{rem}\label{remuniqsweers2ap}
If $W$ is as in Remark \ref{remmonotbelowLem12}, and $\Omega$ is
bounded with smooth boundary (at least $C^3$), in view of the latter
remark and Theorem 2 in \cite{sweersEdinburg}, the solutions of
Theorem \ref{thmmine} that we found by the two different approaches
are actually the same, if $\epsilon$ is chosen sufficiently small.
\end{rem}
\begin{rem}
The first proof of Theorem \ref{thmmine} provides the additional
information of the existence of a minimal and maximal solution of
(\ref{eqEq}).
\end{rem}

\begin{rem}\label{remconvex}
Assume that the domain $\Omega$ is symmetric with respect to the
hyperplane $x_i=0$. Then, since the solution of (\ref{eqEq}),
provided by the second proof of Theorem \ref{thmmine}, is a global
minimizer of the associated energy (in the sense described above, in
case $\Omega$ is unbounded), it follows from Theorem II.5 in
\cite{lopez} (applied on symmetric bounded domains, with respect to
the hyperplane $x_i=0$, exhausting $\Omega$) that the latter
solution is symmetric with respect to this hyperplane. Note that, if
in addition the domain $\Omega$ is bounded and convex in the $x_i$
direction, this assertion holds true for \emph{any} positive
solution of (\ref{eqEq}) by virtue Theorem 2 in \cite{brezis} or
Theorem 1 in \cite{dancePlanes} (proven by the method of moving
planes). Clearly, if uniqueness holds for positive solutions of
(\ref{eqEq}) (recall Remark \ref{remuniq}), these assertions follow
at once (see also Remark 1.3 in \cite{fuscoTrans}).
\end{rem}
\begin{rem}\label{kajikiya}
In the case where $\overline{\Omega}$ is the complement in
$\mathbb{R}^n$ of a smooth \emph{convex} domain $\mathcal{D}$, the
existence of the desired solution to (\ref{eqEq}) can be proven by
noting that the function
\begin{equation}\label{eqkajikaya}
\underline{u}(x)=\textbf{U}\left(\textrm{dist}(x,\partial
\mathcal{D})\right),
\end{equation}
with $\textbf{U}$ as in (\ref{eqU}), is a lower solution to
(\ref{eqEq}). This follows from  (\ref{eqUincr}), and the property
that the distance function $\rho(x)=\textrm{dist}(x,\partial
\mathcal{D})$ satisfies $|\nabla \rho|=1$ and $\Delta \rho\geq 0$
(see \cite{kajikiya} and a related discussion in \cite{savin}).
Actually, it has been proven recently in \cite{hardy1,hardy2} that
these  properties also hold for \emph{mean convex} domains
$\mathcal{D}$. Keep in mind that \emph{$\bar{u}=\mu$ is always an
upper solution}.

In the case where $\Omega$ is the quarter-plane $\{x_1>0,\ x_2>0\}$
(recall our discussion in the introduction about saddle solutions),
and \emph{$W$ also satisfies (\ref{eqbrezisOz})}, it was observed in
\cite{schatman} that the function
\[
\frac{1}{\mu}\textbf{U}\left(\frac{x_1}{\sqrt{2}}
\right)\textbf{U}\left(\frac{x_2}{\sqrt{2}} \right)
\]
is a lower solution to (\ref{eqEq}). We note that if the first
$\textbf{U}$ in the above product is replaced by $u_R$, as provided
by Lemma \ref{lem1} with $n=1$, the resulting function becomes a
lower solution to (\ref{eqEq}) in the semi-infinite strip
$(-\sqrt{2}R,\sqrt{2}R)\times [0,\infty)$; in this regard, recall
our discussion on ``tick'' saddle solutions. Similarly, we can
construct lower solutions in a rectangle (recall the so called
``lattice'' solutions). Analogous constructions hold in arbitrary
dimensions. However, it does not seem likely that one can play this
game for the so called ``pizza'' solutions.
\end{rem}
\begin{rem}
In the case where $\Omega$ is convex,  the function
\[
\bar{u}(x)=\textbf{U}\left(\textrm{dist}(x,\partial \Omega)\right),
\]
is a (weak) upper solution to (\ref{eqEq}) (see the comments
following (\ref{eqkajikaya})). Therefore, if uniqueness of positive
stable solutions holds, we can generalize (\ref{eqfifebarrier}).
\end{rem}
\section{Uniform estimates for positive solutions without specified  boundary
conditions}\label{secpacard} In this section, we will assume
conditions \textbf{(a')},  (\ref{eqpacardW''}), and
(\ref{eqpacardW'}). Under these assumptions,
 we
will establish uniform estimates for solutions of
\begin{equation}\label{eqEqnobdry}
\Delta u=W'(u),
\end{equation}
provided that they are positive and less than $\mu$ over a
sufficiently large set.
 Our motivation comes from Lemmas 3.2--3.3
in \cite{cafareliPacard}, Lemma 4.1 in \cite{kowalczykliupacard},
and Lemma 6.1 in \cite{shiTams} (see also Lemma 2.4 in
\cite{farinaRendi} and \cite{kiselev}).

%
The next proposition and the corollary that follows refine the
latter results, pretty much as (\ref{eqestimRect}) refined
(\ref{eqsweers}). In particular, the approach that we apply for
their proofs will be used crucially in the proof of Theorem
\ref{thmfarmine} below.

\begin{pro}\label{propacard}
Suppose that $W\in C^3$ satisfies \textbf{(a')},
(\ref{eqpacardW''}), and (\ref{eqpacardW'}). Let $\epsilon\in
(0,\mu)$ and $D>D'$, where $D'$ is given from (\ref{eqD}). There
exists a positive constant $R'$, depending on $\epsilon$, $D$, $W$,
$n$, such that for \emph{any} solution of (\ref{eqEqnobdry}) which
satisfies
\begin{equation}\label{eqdummy}
0<u(x)<\mu, \ \ x\in \bar{B}_R(P),\ \textrm{for \ some}\ P\in
\mathbb{R}^n,\ \ \textrm{and}\ R\geq R',
\end{equation}
we have
\begin{equation}\label{eqjmj}
u(x)\geq \mu-\epsilon,\ \ x\in \bar{B}_{(R-D)}(P).
\end{equation}

If $W''\geq 0$ on $[\mu-\epsilon,\mu]$, we have that
\begin{equation}\label{eqcaffaLemmmineNobdry}
\min\left\{W(t)\ :\ t\in \left[0,u(x)\right] \right\}\leq
\frac{C}{R-|x-P|},\ \ x\in B_R(P),
\end{equation}
for some positive constant $C$ that depends only on $W$ and $n$, and
\begin{equation}\label{eqcafaP}
-W'\left(u(P) \right)\leq \tilde{C}(R-|x-P|)^{-2},\ \ x\in B_R(P),
\end{equation}
for some constant $\tilde{C}>0$ that depends only on $n$.
\end{pro}
\begin{proof}
Before we go into the proof, let us make some remarks. The point of
this proposition
  is
that we do not assume that the solutions under consideration are
global minimizers, a case which can be handled similarly to the
second proof of Theorem \ref{thmmine}. The argument that was used
for the related results in \cite{cafareliPacard},
\cite{kowalczykliupacard}, \cite{shiTams} essentially consists in
constructing a family of positive lower solutions of
(\ref{eqEqnobdry}) of the form $s\varphi_R$, $s>0$, where
$\varphi_R$ is the eigenfunction associated to the principal
eigenvalue of the negative Dirichlet Laplacian over a fixed ball
$B_R$, and sweeping \'{a} la Serrin with respect to $s$ (see also
Lemma \ref{lemhelly} herein, Lemma 3.1 in \cite{angenent}, Lemma 2
in \cite{dancerProcLondon}, Theorem 2.1 in \cite{sabina}, and
Proposition 3.1 in \cite{ros}). On the other hand, our argument
consists in constructing a family of nonnegative lower solutions of
(\ref{eqEqnobdry}) from the global minimizing solutions of
(\ref{eqgidasEq}) that are provided by Lemma \ref{lem1}, and
sweeping with respect to the radius of the ball (a similar idea can
be found in \cite{fifesaddle}, see also the comments after
Proposition 2.2 in \cite{guiEven}). Borrowing an expression from
\cite{kiselev}, this type of argument may appropriately be called
"ballooning" (as opposed to "sliding"). The main advantage of our
approach will become clear in Theorem \ref{thmfarmine} below.

Observe that if $u$ solves (\ref{eqgidasEq}), the function
$v(y)\equiv u(Ry)$, $y\in B_1$, satisfies
\begin{equation}\label{eqEVPkorman}
\Delta v=R^2 W'(v),\ y\in B_1;\ \ v(y)=0,\ y\in \partial B_1.
\end{equation}
Since $W'(0)\leq 0$ (recall \textbf{(a')}), it follows from the
results in \cite{korman} (see also Chapter 1 in \cite{kormanBook}),
which are based on the identity (\ref{eqkorman}), that solutions of
(\ref{eqEVPkorman}) lie on smooth curves in the $(R,v)$ ``plane'',
i.e. either solutions of (\ref{eqEVPkorman}) can be continued in $R$
or else there are simple turning points (see also
\cite{kielhoferBook} for the definitions and functional set up). We
will distinguish two cases according to  $W'(0)$:
\begin{itemize}
  \item If
$W'(0)<0$,  by a classical global result of Leray and Schauder (see
\cite{leray} or page 65 in \cite{malchiodicambridge}), there exists
an unbounded connected   branch $\mathcal{C}_+
\subseteq(0,\infty)\times C(\bar{B}_1)$ of positive solutions to
(\ref{eqEVPkorman}) that meets $(0,0)$ (see also Lemma 5.1 in
\cite{AmbrozettiHess}). (As we have already discussed, thanks to
\cite{gidas}, all positive solutions of (\ref{eqEVPkorman}) are
radially symmetric and decreasing). In fact, the detailed behavior
of that branch as $R\to 0^+$ is described in Theorem 3.2 of
\cite{shi2}. By the strong maximum principle, we deduce that the
solutions on $\mathcal{C_+}$ take values strictly less than $\mu$
(by a continuity argument, since they do so for small $R$, see also
Lemma 1 in \cite{hess}). Thus, the projection of $\mathcal{C}_+$
onto $(0,\infty)$ is unbounded, namely coincides with $(0,\infty)$.
  \item If $W'(0)=0$ and $W''(0)<0$ (recall (\ref{eqpacardW''})),  there is a global
connected solution curve $\mathcal{C}_+$ in $(0,\infty)\times
C(\bar{B}_1)$, emanating from $(R_c,0)$, where $R_c$ was defined in
(\ref{eqRc}), due to a bifurcation from a simple eigenvalue as $R$
crosses $R_c$ (see \cite{kielhoferBook}). As before, the solutions
on that branch are positive and strictly less than $\mu$.
 It follows readily from Rabinowitz's global bifurcation theorem \cite{rabiglobaL} (see also Chapter 4 in \cite{malchiodicambridge})
that the projection of $\mathcal{C}_+$ onto $(0,\infty)$ is an
unbounded interval (for this point, which relies on the radial
symmetry of solutions, see the appendix in \cite{shibiology} for
instance). (Keep in mind that $\mathcal{C}_+$ is bounded away from
$R=0$, as can easily be seen by testing (\ref{eqEVPkorman}) by the
principal eigenfunction $\varphi_1$ in (\ref{eqhellyphiR}) (see
Lemma  6.17 in \cite{shi2}); in fact, if (\ref{eqdelpinobelowRc})
holds, the projection of $\mathcal{C}_+$ onto $(0,\infty)$ is
$(R_c,\infty)$).
\end{itemize}


 As in
\cite{dancerCatalysis,korman}, we can parameterize smoothly
$\mathcal{C_+}$ by $\left\{\left(R_\tau,v_\tau\right),\ \tau\in
(0,\tau')\right\}$, for a maximal interval $(0,\tau')\subseteq
(0,\mu)$, where $\tau$ is the maximum of $v_\tau$, namely
$v_\tau(0)=\tau$. The functions
\begin{equation}\label{equv}\texttt{u}_\tau(r)\equiv v_\tau\left(R_\tau^{-1}r\right),\ \
\tau\in (0,\tau'),\ \textrm{where}\ r=|x|,\end{equation}  define a
smooth, with respect to $\tau$, family of solutions to
(\ref{eqgidasEq}), satisfying (\ref{eqmaxball}), with $R=R_\tau$.
Note that
\begin{equation}\label{equrto0RRRRRRR}
R_\tau\to 0, \ \textrm{as}\ \tau\to 0,\ \textrm{if}\ W'(0)<0;\
R_\tau\to R_c, \ \textrm{as}\ \tau\to 0,\ \textrm{if}\
(\ref{eqpacardW''}) \ \textrm{holds},
\end{equation}
and the range of $R_\tau$, $\tau\in(0,\tau')$, covers $(0,\infty)$
and $(R_c,\infty)$ respectively (in the latter case, the covered
interval might be $[R_1,\infty)$ with $R_1<R_c$, but strictly
positive as can be seen by testing by the principal eigenfunction).
In view of Lemma \ref{lemhelly}, it follows that \[\tau'=\mu.\]
 On the other side, from the definition of $\tau$, we see that
\begin{equation}\label{equrto0}
 \texttt{u}_{\tau}\to 0,\ \textrm{uniformly\ on}\ \bar{B}_{R_\tau}, \
\textrm{as}\ \tau\to 0.
\end{equation}

 Given $\epsilon\in
(0,\mu)$ and $D>D'$, where $D'$  as in (\ref{eqD}), let $R'$ be as
in (\ref{eqestimRect}). Suppose that a solution $u$ of
(\ref{eqEqnobdry}) satisfies (\ref{eqdummy}) for some
$\texttt{R}>R'$ and $P\in \mathbb{R}^n$. The family of functions
\[
\underline{\texttt{u}}_{\tau,P}(x)=\left\{\begin{array}{ll}
                            \texttt{u}_\tau(x-P), &x\in B_{R_\tau}(P),  \\
                             &   \\
                            0, & \textrm{elsewhere},
                          \end{array}
\right.
\]
are lower solutions to (\ref{eqEqnobdry}) in $\mathbb{R}^n$ for all
$\tau\in (0,\mu)$ (as the maximum of two lower solutions, recall
that $W'(0)\leq 0$, see \cite{Berestyckilion}). Moreover, we have
\[
\underline{\texttt{u}}_{\tau,P}(x)=0<u(x),\ \ x\in \partial
B_\texttt{R}(P),\ \ \tau \in (0,\tau_*],
\]
where $\tau_*$ is the smallest number  such that
\[R_{\tau_*}=\texttt{R},\]
(such a number exists since $R_\tau$ is smooth and $R_{\tau_i}\to
\infty$ for some sequence $\tau_i\to \infty$).
 Also, thanks to (\ref{equrto0RRRRRRR})
and (\ref{equrto0}), we get
\[
\underline{\texttt{u}}_{\tau,P}(x)<u(x),\ \ x\in
\bar{B}_\texttt{R}(P), \ \textrm{for} \ \tau\ \textrm{close\ to}\ 0.
\]
Consequently, by Serrin's sweeping principle (see
\cite{clemente,dancerProcLondon,sattinger}), we deduce that
\[\underline{\texttt{u}}_{\tau_*,P}(x)\leq u(x),\ \ x\in
\bar{B}_\texttt{R}(P).\] In turn, this implies that
\begin{equation}\label{eqabove}
u_\texttt{R}(x-P)=\texttt{u}_{\tau_*}(x-P)=\underline{\texttt{u}}_{\tau_*,P}(x)\leq
u(x),\ \ x\in \bar{B}_\texttt{R}(P),
\end{equation}
where $u_{\texttt{R}}$ is some solution to (\ref{eqgidasEq}) that
satisfies (\ref{eqmaxball}) with $R=\texttt{R}$. To prove this, we
let $\tilde{\tau}=\sup \left\{\tau\in (0,\tau_*]\ :\ u\geq
\underline{\texttt{u}}_{\tau,P}\ \textrm{on}\
\bar{B}_{\texttt{R}}(P)\right\}$, note that $u(x)\geq
\texttt{u}_{\tilde{\tau}}(x-P)$,
$x\in\bar{B}_{R_{\tilde{\tau}}}(P)$, and apply the strong maximum
principle to $u(x)-\texttt{u}_{\tilde{\tau}}(x-P)$ to deduce that
this function has a positive lower bound on
$\bar{B}_{R_{\tilde{\tau}}}(P)$ if $\tilde{\tau}<\tau_*$; this
implies that the same holds true for the function
$u-\underline{\texttt{u}}_{\tilde{\tau},P}$ on
$\bar{B}_\texttt{R}(P)$ which contradicts the maximality of
$\tilde{\tau}$ if $\tilde{\tau}<\tau_*$.
 Relation (\ref{eqabove}), by virtue of
Lemma \ref{lemhelly} (recall that $\texttt{R}>R'$), clearly implies
the validity of (\ref{eqjmj}).

If $W''\geq 0$ on $[\mu-\epsilon,\mu]$, from Remark
\ref{remuniqDancer} below, we know that (\ref{eqgidasEq}) has a
unique solution satisfying (\ref{eqmaxball}) for large $R$. In
particular, the solution $u_{\texttt{R}}$ in (\ref{eqabove}) is the
global minimizer of Lemma \ref{lem1}, provided that $\texttt{R}$ is
sufficiently large.
 The validity of relation
(\ref{eqcaffaLemmmineNobdry}) now follows at once from
(\ref{eqcaffaLemmmine}) and (\ref{eqabove}). Finally, relation
(\ref{eqcafaP}) follows immediately  from  (\ref{eqcafa}) and
(\ref{eqabove}).

The proof of the proposition is complete.
\end{proof}

\begin{rem}\label{rempacardBrezOz}
In the case where condition (\ref{eqbrezisOz}) holds, relation
(\ref{eqabove}) follows directly from Serrin's sweeping principle.
Indeed, it is easy to verify that the functions
$tu_\texttt{R}(x-P)$, $0\leq t\leq 1$, fashion a family of lower
solutions to (\ref{eqgidasEq}) which vanish along $\partial
B_\texttt{R}(P)$.
\end{rem}

\begin{rem}\label{remW''=0pacard}
The assumption (\ref{eqpacardW''}) is essential for our approach.
Indeed, if $W'(0)=0$ and $W''(0)=0$ then there are no arbitrarily
uniformly
 small positive
 solutions of (\ref{eqEVPkorman}) for any $R>0$ (thanks to the implicit function theorem, see for example \cite{kielhoferBook}). In fact, for the
 case $W'(t)=rt^p-t^q$, $t\geq 0$, with $p>q>1$, $r>0$, which satisfies \textbf{(a')},
 (\ref{eqpacardW'}), and (\ref{eqW''})  but not (\ref{eqpacardW''}),
 the
global bifurcation diagram of
 positive solutions of (\ref{eqEVPkorman}) has been shown in
 \cite{shiOyang} to be qualitatively the same as the one
 corresponding to
 (\ref{eqwei}) that we described at the end of Remark \ref{remscrew}, namely $\subset$-shaped.
 It might also be useful to see the condition on the behavior of
 $W'$ near the origin for the so-called  ``hair--trigger effect'' to take place in
 the parabolic equation $u_t=\Delta u-W'(u)$, see \cite{aronson}.
\end{rem}

\begin{rem}\label{remdrawbak2222}
In \cite{cafareliPacard}, the assumption (\ref{eqpacardW''}) was
replaced by the weaker one: $W'$ being Lipschitz continuous and
$-W'(t)\geq \delta_0 t$ on $[0,t_0]$ for some $\delta_0,t_0>0$. A
possible ``cure'' for this could be the use of bifurcation theory
for H\"{o}lder continuous nonlinearities (see Appendix B in
\cite{shi2} and the references therein).
\end{rem}

\begin{rem}\label{remangenent}
The proof of the above proposition may be adapted to provide an
alternative proof of Lemma 3.1 in \cite{angenent}. Therefore, one
can estimate the width of the boundary layer of positive solutions
to the spatially inhomogeneous singular perturbation  problem
(\ref{eqinhomogeneous}) below (with $W(\cdot,x)$ essentially
satisfying the assumptions of this section for each $x$) only in
terms of $W$, without involving the principal Dirichlet eigenvalue
of the Laplacian in the unit ball of $\mathbb{R}^n$ (see also Remark
\ref{remSingular} below).
\end{rem}

The following corollary can be deduced from Proposition
 \ref{propacard} by making use of the celebrated sliding method.
\begin{cor}\label{corpacard}
Suppose that $W \in C^3$ satisfies \textbf{(a')} and
(\ref{eqpacardW'}). Let $\epsilon\in (0,\mu)$ and $D>D'$, where $D'$
is given from (\ref{eqD}). There exists an $R'>D$, depending only on
$\epsilon$, $D$, $W$, and $n$, such that any solution $u$ of
(\ref{eqEqnobdry}) which satisfies (\ref{eq12}) in a domain
$\Omega\neq \mathbb{R}^n$ (open and connected set), containing some
closed ball of radius $R'$, satisfies (\ref{eq14+}),
(\ref{eqcaffaThmmine}), and (\ref{eqcafathm}). In the case where
$\Omega=\mathbb{R}^n$, the only solutions of (\ref{eqentire}) such
that $0\leq u(x)\leq \mu$, $x\in \mathbb{R}^n$, are the constant
ones, namely $u\equiv 0$ and $u\equiv \mu$.
 \end{cor}
\begin{proof}
Suppose that $\epsilon$, $D$, $\Omega \neq \mathbb{R}^n$, $u$  are
as in the first assertion of the corollary. From our assumptions, we
know that $\Omega$ contains some closed ball $\bar{B}_\texttt{R}(P)$
for some $\texttt{R}\geq R'$ and $P\in \Omega_\texttt{R}$.
 Since $u$
satisfies (\ref{eq12}) and (\ref{eqEqnobdry}) in $\Omega$, it
follows from the proof of Proposition \ref{propacard} that relation
(\ref{eqabove}) holds. As in the first proof of Theorem
\ref{thmmine}, we can use the sliding method to show that the latter
relation actually holds for all $P\in \Omega_\texttt{R}$. We point
out that here we do not need that the boundary of $\Omega$ is
continuous, since the radius of the ball is fixed, and we can apply
directly Lemma 3.1 in \cite{cafareliPacard}. The validity of the
first assertion of (\ref{eq14+}) now follows at once form
(\ref{eqestimRect}) (keep in mind that $\texttt{R}$ could also be
chosen as $R'$). Now, let $Q\in \Omega_{R'}+B_{(R'-D)}$. From the
proof of Proposition \ref{propacard}, using Serrin's sweeping
principle, we have that $u(x)\geq u_{\textrm{dist}(Q,\partial
\Omega)}(x-Q)$ in $B_{\textrm{dist}(Q,\partial \Omega)}(Q)$. By
means of (\ref{eqcaffaLemmmine}) and (\ref{eqcafaP}), we infer that
$u$ also satisfies (\ref{eqcaffaThmmine}) and (\ref{eqcafathm})
respectively. Consequently, we have established the first assertion
of the corollary.

The second assertion of the corollary follows easily. By the strong
maximum principle, we deduce that either $u\equiv 0$,  $u\equiv
\mu$, or $0<u(x)<\mu$, $x\in \mathbb{R}^n$. We will show that the
latter alternative cannot happen. Suppose to the contrary that
$0<u(x)<\mu$, $x\in \mathbb{R}^n$. Then,  we get that
(\ref{eqabove}) holds for every $\texttt{R}>0$ and $P\in
\mathbb{R}^n$. By fixing $P$ and letting $\texttt{R}\to \infty$,
making use of (\ref{eqsweers}), we obtain that $u\geq \mu$ in
$\mathbb{R}^n$; a contradiction.

The proof of the corollary is complete.
\end{proof}
\begin{rem}
The second assertion of the corollary is a Liouville type theorem,
and was originally proven in \cite{aronson} by parabolic methods
(see also \cite{berestyckiLiouville} for a simpler proof of a more
general result,  using elliptic techniques, in the spirit of
\cite{cafareliPacard}; see also Theorem 2.7 in \cite{efediev}).
\end{rem}
\begin{rem}
If we additionally   assume that $W''(\mu)>0$, then Proposition
\ref{propacard} and Corollary \ref{corpacard} can be derived from
the exponential decay estimates of Lemma 4.2 in
\cite{kowalczykliupacard}, Proposition 1 in
\cite{nakashimaNonradial}, and Lemma 6.2 in \cite{shiTams} (see also
\cite{ghosubGui} and Lemma 2.4 in \cite{guitraveling}).
\end{rem}
\begin{rem}\label{remcafa}
Estimate (\ref{eqcafathm}) represents a slight improvement over
estimate $(3.3)$ in \cite{cafareliPacard} (see also relation
(\ref{eqcafaImproved}) below). We remark that the latter relation
was shown in \cite{cafareliPacard} without making use of
(\ref{eqW''}).
\end{rem}
\section{Algebraic singularity decay estimates in the
case of pure power nonlinearity, and completion of the proof of
Theorem \ref{thmmine}}\label{secdoubling} The potential that comes
first to mind when looking at \textbf{(a')} is
\begin{equation}\label{eqmodel}
W(t)=|t-\mu|^{p+1},
\end{equation}
where $p\geq1$.

If $p=1$, the solutions provided by Theorem \ref{thmmine} satisfy
the exponential decay estimate (\ref{eq13}). In this section, we
will show that a universal  algebraic decay estimate holds for
\emph{all} solutions of (\ref{eqEqnobdry}) with potential as in
(\ref{eqmodel}), provided that $p>1$. Although our arguments in this
section rely on the specific form of the potential $W$, our results
may be used together with a comparison argument to cover a broader
class of potentials. In particular, as we will show in the following
subsection, we can establish the remaining relation (\ref{eqfinal})
from the proof of Theorem \ref{thmmine}. Moreover, our main estimate
in this section suggests that there is room for improvement over a
result of the celebrated paper \cite{cafareliPacard} by Berestycki,
Caffarelli and Nirenberg, see Remark \ref{rembcn} below. We believe
that the results of this section can have applications in the study
of elliptic singular perturbation problems of the form
(\ref{eqinhomogeneous}) below in the case where the degenerate
equation $W(u,x)=0$ has a root $u=u_0(x)$ of finite multiplicity
(see the recent papers \cite{vasilevaDiff,vasileva1-,vasileva2}).
 This section is self-contained and can be studied
independently of the rest of the paper.

The main result of this section is
\begin{pro}\label{proModel}
Let $W$ be given from (\ref{eqmodel}), with $p>1$, and let
$\Omega\neq \mathbb{R}^n$ be a domain of $\mathbb{R}^n$. There
exists a positive constant $C$, depending only on $p,\ n$, such that
\emph{any} solution $u$ of (\ref{eqEqnobdry}) in $\Omega$ satisfies
\begin{equation}\label{eqmodelAssert}
|\mu-u|+|\nabla u|^{\frac{2}{p+1}}\leq
C\textrm{dist}^{-\frac{2}{p-1}}(x,\partial \Omega),\ \ x\in \Omega.
\end{equation}
In particular, if $\Omega$ is an exterior domain, i.e., $\Omega
\supset \{x\in \mathbb{R}^n\ :\ |x|>R\}$ for some $R>0$, then
\begin{equation}\label{eqmodelAssertExt}
|\mu-u|+|\nabla u|^{\frac{2}{p+1}}\leq C|x|^{-\frac{2}{p-1}},\ \
|x|\geq 2R.
\end{equation}
\end{pro}
\begin{proof}
Our proof is modeled after that of Theorem 2.3 in \cite{pqs} which
dealt with focusing nonlinearities, making use of scaling (blow-up)
arguments, inspired from \cite{gidasSpruck}, combined with a key
``doubling'' estimate. The main difference with \cite{pqs} is in the
Liouville type theorem that we will use to conclude, see Remark
\ref{rempolacikfocus} below.

We will argue by contradiction. Suppose that estimate
(\ref{eqmodelAssert}) fails. Then, there exist sequences of domains
$\Omega_k\neq \mathbb{R}^n$, functions $u_k$, and points $y_k\in
\Omega_k$, such that $u_k$ solves (\ref{eqEqnobdry}) in $\Omega_k$
and the functions
\[
M_k\equiv |\mu-u_k|^{\frac{p-1}{2}}+|\nabla u_k|^\frac{p-1}{p+1},\ \
k=1,2,\cdots,
\]
satisfy
\[
M_k(y_k)>2k\textrm{dist}^{-1}(y_k,\partial \Omega_k),\ \
k=1,2,\cdots.
\]
From the \emph{Doubling Lemma} of Polacik, Quittner and Souplet, see
Lemma 5.1 and Remark 5.2 (b) in \cite{pqs} or Lemma
\ref{lemDoubling} in the appendix below, it follows that there exist
$x_k\in \Omega_k$ such that
\begin{equation}\label{eqdouble1}
M_k(x_k)\geq M_k(y_k),\ \ M_k(x_k)>2k\textrm{dist}^{-1}(x_k,\partial
\Omega_k),
\end{equation}
and
\begin{equation}\label{eqdouble2}
M_k(z)\leq 2 M_k(x_k),\ \ |z-x_k|\leq k M^{-1}_k(x_k),\ \
k=1,2,\cdots.
\end{equation}
Note that $B_{k M^{-1}_k(x_k)}(x_k)\subset \Omega_k$.  Now, we
rescale $u_k$ by setting
\begin{equation}\label{eqdouble3}
v_k(y)\equiv \lambda_k^{\frac{2}{p-1}}\left[\mu-u_k(x_k+\lambda_k
y)\right],\ \ |y|\leq k,\ \ \textrm{with}\ \lambda_k=M_k^{-1}(x_k).
\end{equation}
The function $v_k$ solves
\[
\Delta v_k(y)=(p+1) v_k(y)\left|v_k(y)\right|^{p-1},\ \ |y|\leq k.
\]
Moreover,
\begin{equation}\label{eqdoublingnontriv}
\left[ |v_k|^{\frac{p-1}{2}}+|\nabla
v_k|^\frac{p-1}{p+1}\right](0)=\lambda_k M_k(x_k)=1,
\end{equation}
and
\[
\left[ |v_k|^{\frac{p-1}{2}}+|\nabla
v_k|^\frac{p-1}{p+1}\right](y)\leq 2,\ \ |y|\leq k.
\]
By using elliptic $L^q$ estimates and standard imbeddings (see
\cite{Gilbarg-Trudinger}), we deduce that some subsequence of $v_k$
converges in $C^1_{loc}(\mathbb{R}^n)$ to a (classical) solution
$\textbf{V}$ of
\begin{equation}\label{eqbrezisLiouville}
\Delta v=(p+1) v(y)\left|v(y)\right|^{p-1},\ \ y\in \mathbb{R}^n.
\end{equation}
Moreover, thanks to (\ref{eqdoublingnontriv}), we have
\[
\left[ |\textbf{V}|^{\frac{p-1}{2}}+|\nabla
\textbf{V}|^\frac{p-1}{p+1}\right](0)=1,
\]
so that $\textbf{V}$ is nontrivial. On the other hand, by a result
of Brezis \cite{brezisLiouville}, we know that there does not exist
a nontrivial solution of (\ref{eqbrezisLiouville}) in
$L^p_{loc}(\mathbb{R}^n)$ (in the sense of distributions), see also
Theorems 4.6--4.7 in \cite{farinaHB} or Theorem \ref{thmFarina}
below. Consequently, we have arrived at the desired contradiction.

The proof of the proposition is complete.
\end{proof}
\begin{rem}\label{remDoublingOptimal}
The powers $2/(p-1)$ and $(p+1)/(p-1)$ in (\ref{eqmodelAssert}) (for
$u$ and $|\nabla u|$ respectively) are sharp for $p\in
\left(1,\frac{n}{n-2} \right)$ if $n\geq 3$ and $p>1$ if $n=2$, as
can be seen from the explicit solution
$u(x)=c(p,n)|x|^{-\frac{2}{p-1}}$ of
\begin{equation}\label{eqbrezisARMA}
\Delta u= u^p,
\end{equation}
in $\mathbb{R}^n\backslash \{ 0\}$, see for example
\cite{brezis-oswald}.

In the latter reference, it was shown that every nonnegative
solution $u\in C^2(B_R)$ of (\ref{eqbrezisARMA}), with $p>1$,
satisfies
\[
u(0)\leq C(p,n)R^{-\frac{2}{p-1}},
\]
where $C(p,n)$ is determined explicitly. This result, minus the
explicit dependence of the constat on $p,n$, follows as a special
case of our Proposition \ref{proModel} if we choose $\mu=0$.
Moreover, it was shown in the same reference that every  nonnegative
solution $u\in C^2\left(B_R\backslash \{0 \}\right)$ of
(\ref{eqbrezisARMA}), with $p\in \left(1,\frac{n}{n-2} \right)$ if
$n\geq 3$ and $p>1$ if $n=2$, satisfies
\[
u(x)\leq
l(p,n)|x|^{-\frac{2}{p-1}}\left[1+\frac{C(p,n)}{l(p,n)}\left(\frac{|x|}{R}
\right)^\beta \right],\ \ 0<|x|\leq \frac{R}{2},
\]
where $\beta= \frac{4}{p-1}+2-n>\frac{2}{p-1}$, and $C(p,n), l(p,n)$
some explicitly determined constants. The validity of this estimate,
minus the explicit dependence of the constat on $p,n$, follows for
\emph{all} the range $p>1$ from our Proposition \ref{proModel} with
$\mu=0$. It was shown in  \cite{brezisveronARMA} that, if $n\geq 3$
and $p\geq \frac{n}{n-2}$, there exists a constant $A=A(p,n)>0$ such
that every nonnegative solution $u\in C^2\left(B_1\backslash \{0
\}\right)$ of (\ref{eqbrezisARMA}) satisfies
\[
u(x)\leq \frac{A}{|x|^{n-2}},\ \ 0<|x|<\frac{1}{2}.
\]
In turn, the latter estimate was used to show that the solution $u$
has a removable singularity at the origin. Clearly, the above
estimate follows from (\ref{eqmodelAssert}) with $\mu=0$. The proofs
in \cite{brezisveronARMA}, \cite{brezisARMA}, and
\cite{brezisLiouville} (where we referred to towards  the end of the
proof of Proposition \ref{proModel}), are based on the explicit
knowledge of positive, radially symmetric upper solutions of the
equation $-\Delta u + |u|^{p-1}u = 0$ on arbitrary open balls, with
the further property that these functions blow up at the boundary of
the considered balls. This fact is crucially related to the shape of
the nonlinear function $|t |^{p-1}t$ and it does not easily extend
to more general functions. We refer to \cite{farinaHB} for a
different approach for establishing the Liouville type theorem of
\cite{brezisLiouville}, that we used towards the end of the proof of
Proposition \ref{proModel}, with the advantage to apply to a larger
class of problems (see Theorem \ref{thmFarina} below).

To the best of our knowledge, this is the first time that the
doubling lemma of \cite{pqs} has been used in relation with the
previously mentioned papers.
\end{rem}

\begin{rem}\label{rempolacikfocus}
In the problems studied in \cite{gidasSpruck}, \cite{pqs} (see also
\cite{QitnerrIma}), the blowing-up argument leads to a positive
solution of the whole space problem $\Delta v+v^p=0$, which does not
exist for the range of exponents $p\in
\left(1,\frac{n+2}{n-2}\right)$ if $n\geq 3$, $p>1$ if $n=2$.
\end{rem}

\begin{rem}\label{rembcn}
Assume that the potential $W\in C^2$ satisfies (\ref{eqpacardW'}),
$W'(t)\geq 0$ for $t\geq \mu$, $-W'(t)\geq \delta_0 t$ on $[0,t_0]$
for some $\delta_0,t_0>0$, and (\ref{eqW''}). Clearly, these
conditions are satisfied by the model examples (\ref{eqAllen}) and
(\ref{eqmodel}). Let $\Omega$ be the entire \emph{epigraph}:
\begin{equation}\label{eqOmegaBCN}
\Omega=\left\{x\in \mathbb{R}^n\ :\ x_n>\varphi(x_1,\cdots,x_{n-1})
\right\},
\end{equation}
where $\varphi:\mathbb{R}^{n-1}\to \mathbb{R}$ is a globally
Lipschitz continuous function. It was shown in Lemma 3.2 in
\cite{cafareliPacard} that there are constants $\varepsilon_1,R_0>0$
with $R_0$, depending only on $n,\delta_0$, such that
 any positive bounded solution of
(\ref{eqEq}) satisfies $u<\mu$ in $\Omega$, and
\[
u(x)>\varepsilon_1\ \ \textrm{if}\ \ x\in \Omega_{R_0},\
\textrm{i.e.}\ \textrm{dist}(x,\partial \Omega)>R_0.
\]
Moreover, setting
\[
\delta(x)=\min\left\{-W'(t)\ :\ t\in \left[\varepsilon_1,u(x)
\right] \right\},\ \ x\in \Omega_{R_0},
\]
there exists a constant $C_1$, depending only on $n$, such that
\begin{equation}\label{eqcafaImproved}
C_1\delta(x)\leq \left[\textrm{dist}(x,\partial \Omega)-R_0
\right]^{-2},\ \ x\in \Omega_{R_0},
\end{equation}
recall also estimate (\ref{eqcafathm}) herein. In the case of the
potential (\ref{eqmodel}), the function $\delta(x)$ is plainly $
\delta(x)=(p+1)\left(\mu-u(x) \right)^p$, and estimate
(\ref{eqcafaImproved}) says that
\[
\mu-u(x)\leq \left[C_1  (p+1)
\right]^{-\frac{1}{p}}\left[\textrm{dist}(x,\partial \Omega)-R_0
\right]^{-\frac{2}{p}},\ \ x\in \Omega_{R_0}.
\]
Observe that our estimate (\ref{eqmodelAssert}) is an improvement of
the above estimate, since $\frac{2}{p-1} >\frac{2}{p}$. Moreover,
our estimate holds for \emph{all} solutions, possibly sign changing
or unbounded, without specified boundary conditions.
 Note also
that these observations reveal that estimate (\ref{eqcaffaThmmine})
is far from optimal.
\end{rem}

As in Theorem 2.1 in \cite{pqs}, we can generalize our Proposition
\ref{proModel} to
\begin{pro}\label{propolacikGeneral}
Let $p>1$, and assume that the smooth $W$ satisfies
\begin{equation}\label{eqduScale}
\lim_{|t|\to \infty} t^{-1}|t|^{1-p}W'(t+\mu)=\ell \in (0,\infty).
\end{equation}
Let $\Omega$ be an arbitrary domain of $\mathbb{R}^n$. Then, there
exists a constant $C=C(n,W')>0$ (independent of $\Omega$ and $u$)
such that, for any solution of (\ref{eqEqnobdry}), there holds
\begin{equation}\label{eqpolacik}
|\mu-u|+|\nabla u|^\frac{2}{p+1}\leq C\left(
1+\textrm{dist}^{-\frac{2}{p-1}}(x,\partial \Omega)\right),\ \ x\in
\Omega.
\end{equation}
In particular, if $\Omega=B_R\backslash \{0\}$ then
\[
|\mu-u|+|\nabla u|^\frac{2}{p+1}\leq C\left(
1+|x|^{-\frac{2}{p-1}}\right),\ \ 0<|x|\leq \frac{R}{2}.
\]
\end{pro}
\begin{proof}
Assume that estimate (\ref{eqpolacik}) fails. Keeping the same
notation as in the proof of Proposition \ref{proModel}, we have
sequences $\Omega_k,\ u_k,\ y_k\in \Omega_k$ such that $u_k$ solves
(\ref{eqEqnobdry}) in $\Omega_k$ and
\[
M_k(y_k)>2k\left(1+\textrm{dist}^{-1}(y_k,\partial
\Omega_k)\right)>2k\textrm{dist}^{-1}(y_k,\partial \Omega_k).
\]
Then, formulae (\ref{eqdouble1})--(\ref{eqdouble3}) are unchanged
but now the function $v_k$ solves
\[
\Delta_yv_k(y)=f_k\left(v_k(y)
\right)\equiv-\lambda_k^\frac{2p}{p-1}W'\left(\mu-\lambda_k^{-\frac{2}{p-1}}
v_k(y)\right),\ \ |y|\leq k.
\]
Note that, since $M_k(x_k)\geq M_k(y_k)>2k$, we also have
\[
\lambda_k\to 0,\ \ k\to \infty.
\]
Since there exists a constant $C>0$ such that $\left|W'\left(\mu-t
\right) \right|\leq C(1+|t|^p)$, $t\in \mathbb{R}$, due to
(\ref{eqduScale}) (and $W'$ being continuous), it follows that
\[
\left|f_k(t) \right|\leq C\lambda_k^\frac{2p}{p-1}+C|t|^p,\ \ t\in
\mathbb{R},\ k\geq 1.
\]
By using elliptic $L^q$ estimates, standard imbeddings, and
(\ref{eqduScale}), we deduce that some subsequence of $v_k$
converges in $C^1_{loc}(\mathbb{R}^n)$ to a  classical solution
$\textbf{V}$ of $\Delta v=\ell v|v|^{p-1}$ in $\mathbb{R}^n$.
Moreover, we have that $ |\textbf{V}(0)|^{\frac{p-1}{2}}+|\nabla
\textbf{V}(0)|^\frac{p-1}{p+1}=1$, so that $\textbf{V}$ is
nontrivial. As in Proposition \ref{proModel}, since $\ell>0$,  this
contradicts the Liouville theorem in \cite{brezis}, \cite{farinaHB},
in particular Theorem \ref{thmFarina} below.

The proof of the proposition is complete.
\end{proof}
\begin{rem}\label{remfarinapola}
The same assertion of Proposition \ref{propolacikGeneral} holds, if
we assume that the righthand side of (\ref{eqduScale}) is as the
function $f$ in Theorem \ref{thmFarina} below.
\end{rem}
\subsection{Proof of relation (\ref{eqfinal})}\label{submodel}
Based on Proposition \ref{proModel}, via a comparison argument, we
can show relation (\ref{eqfinal}) and thus complete the proof of
Theorem \ref{thmmine}.

\texttt{Proof of (\ref{eqfinal}):} Clearly, estimate (\ref{eqfinal})
holds if $\textrm{dist}(x,\partial \Omega)\leq R'$.

In any connected component $\mathcal{A}$ of
$\Omega_{R'}+B_{(R'-D)}$, thanks to (\ref{eq14+}) and
(\ref{eqWmonotThm}) (assuming that $\epsilon<d$), we have
\[
\Delta u\leq -c (\mu-u)^{p}\  \textrm {in}\  \mathcal{A},\ \ u\geq
\mu-\epsilon \ \textrm{on}\ \partial \mathcal{A}.
\]

 Let $0<v<\mu$ be the solution of
\[
\Delta v=-c(\mu-v)^{p}\ \ \textrm{in}\ \mathcal{A};\ \ v=0\ \
\textrm{on}\
\partial \mathcal{A},
\]
as provided by Theorem \ref{thmmine} (keep in mind the second part
of Remark \ref{remuniq} which implies uniqueness), where $c>0,\ p>1$
are as in (\ref{eqWmonotThm}). From Proposition \ref{proModel}, we
know that $v$ satisfies
\[
\mu-v\leq \hat{K}\textrm{dist}^{-\frac{2}{p-1}}(x,\partial
\mathcal{A}),\ \ x\in \mathcal{A},
\]
for some constant $\hat{K}>0$ that depends only on $c,p,$ and $n$.
Since $\textrm{dist}(\partial \mathcal{A},\partial \Omega)\leq D$,
we have
\[
\textrm{dist}(x,\partial \mathcal{A})\geq \textrm{dist}(x,\partial
\Omega)-\textrm{dist}(\partial \mathcal{A},\partial \Omega)\geq
\textrm{dist}(x,\partial \Omega)-D.
\]
So, we arrive at
\[
\mu-v\leq \tilde{K}\textrm{dist}^{-\frac{2}{p-1}}(x,\partial
\Omega),\ \ x\in \mathcal{A},
\]
for some constant $\tilde{K}>0$ that depends only on $c,p,n$ and
$W$.

We intend to show that
\begin{equation}\label{eqvu}
v\leq u\ \ \textrm{in}\ \mathcal{A},
\end{equation}
from where relation (\ref{eqfinal}) follows at once.
 Let $w=u-v$. We have
\[
\Delta w\leq q(x)w \ \ \textrm{in}\ \mathcal{A},
\]
where
\[
q(x)=cp\int_{0}^{1}\left(\mu-s u-(1-s)v \right)^{p-1}ds.
\]
This is a bit meshy but what  matters is that $q$ is continuous and
nonnegative in $\mathcal{A}$.
 Note that $w> 0$ on $\partial \mathcal{A}$ and $w$ is bounded in
 $\mathcal{A}$ ($|w|\leq \mu$ to be more precise).
Therefore, the assumption that $\bar{\Omega}$ is disjoint from the
closure of an infinite open connected cone (and so is
$\bar{\mathcal{A}}$), or $n=2$ and $\bar{\Omega}\neq \mathbb{R}^2$,
allows us to apply the maximum principle, even in the case where
$\mathcal{A}$ is unbounded, to deduce that (\ref{eqvu}) holds (see
Lemma 2.1 and the remark following it in \cite{cafareliPacard}, and
also Lemma 6.2 in \cite{kiselev}).

The proof of relation (\ref{eqfinal}) is complete.\ \ \ \  \
$\square$

\begin{rem}
The proof of Lemma 2.1 in \cite{cafareliPacard} is based on the
property that, for the domains $\Omega$ in the previous class, there
exists a positive super-harmonic function $g$ in $\Omega$ (i.e.
$\Delta g\leq 0$ in $\Omega$) such that $g(x)\to \infty$ if $x\in
\Omega$ and $|x|\to \infty$. If $\Omega$ is contained in the set
$\left\{x_n\geq h(x_1,\cdots,x_{n-1}),\ (x_1,\cdots,x_{n-1})\in
\mathbb{R}^{n-1} \right\}$, for some $h\in C(\mathbb{R}^{n-1})$ such
that $h(y)\to \infty$ as $|y|\to \infty$, then clearly we can take
\[g(x)=x_n-\min_{\mathbb{R}^{n-1}}h+1.\]
In light of the comparison function of Theorem 2 in
\cite{dancePlanes}, it follows readily that Lemma 2.1 in
\cite{cafareliPacard} also applies in the case where ${\Omega}$ is
contained in  some  strip
\[\left\{|x_i|\leq M,\ i=1,\cdots,n-m; \
x_j\in \mathbb{R},\ j=n-m+1,\cdots,n\right\},\] where $M>0$ and
$m\in \{1,\cdots, n-1\}$.
 Therefore, relation
(\ref{eqfinal}) also holds for such domains.
\end{rem}


\section{Bounds on entire solutions of $\Delta
u=W'(u)$}\label{secBoundsonEntire} In this subsection, we will
assume that the    $C^2$ potential $W$ satisfies (\ref{eqduScale})
for some $\mu \in \mathbb{R}$, and there exist $\mu_-<\mu_+$ such
that
\begin{equation}\label{eqdumatano}
W'(\mu_-)=W'(\mu_+)=0,\ \ W'(t)<0,\ t<\mu_-;\ W'(t)>0,\ t>\mu_+.
\end{equation}
We will utilize Propositions \ref{propacard} and
\ref{propolacikGeneral}, together with the corresponding parabolic
flow to (\ref{eqentire}), in order  obtain the following result:
\begin{pro}\label{produ}
Under the above assumptions, we have that \emph{every}  solution
$u\in C^2(\mathbb{R}^n)$ of (\ref{eqentire}), which is not
identically equal to $\mu_-$ or $\mu_+$, satisfies
\begin{equation}\label{eqdustrict}
\mu_-<u(x)<\mu_+,\ \ x\in \mathbb{R}^n.
\end{equation}
\end{pro}
\begin{proof}
From Proposition \ref{propolacikGeneral} with $\Omega=\mathbb{R}^n$,
i.e. $\textrm{dist}^{-\frac{2}{p-1}}(x,\partial \Omega)=0$  $\forall
\ x\in \mathbb{R}^n$, we know that there exists a constant
$C=C(W',n)>0$ such that every solution of (\ref{eqentire}) satisfies
\[
|u(x)|\leq C,\ \ x\in \mathbb{R}^n.
\]
We will show that
\begin{equation}\label{eqdu2}
\mu_-\leq u(x)\leq \mu_+,\ \ x\in \mathbb{R}^n.
\end{equation}
Indeed, as in \cite{angenent}, \cite{duMa}, by the parabolic maximum
principle (this is possible since all the functions under
consideration are bounded, see \cite{protterBook}), we infer that
\begin{equation}\label{eqdutime}
u_-(t)\leq u(x)\leq u_+(t),\ \ t\geq 0,
\end{equation}
where $u_\pm$ are the solutions of the initial value problems
\[
\dot{u}_\pm=W'(u_\pm),\ \ t> 0,\ \ u_\pm(0)=\pm 2C.
\]
Note that $u_\pm(t)$ are solutions of $u_t=\Delta u-W'(u)$ on
$\mathbb{R}^n \times (0,\infty)$, as is $u(x)$. From our assumptions
on $W$, it follows that $u_-(t)$ is increasing and $u_+(t)$ is
decreasing with respect to $t>0$. Furthermore, it is easy to show
that $u_\pm(t)\to \mu_\pm$ as $t\to \infty$, see also \cite{arnold},
\cite{walter}. Hence, letting $t\to \infty$ in (\ref{eqdutime}), we
find that relation (\ref{eqdu2}) holds. For a similar argument,
which allows for the last assumption in (\ref{eqdumatano}) to be
weakened (allowing $W'$ to vanish), we refer to Theorem 2.7 in
\cite{efediev}.
 Alternatively, we could
argue as in Corollary \ref{corpacard} by considering the function
$2C-u$.
By the strong maximum principle, it follows that (\ref{eqdustrict})
holds unless $u\equiv \mu_-$ or $u\equiv \mu_+$.

 The proof of the proposition is complete.
\end{proof}
\begin{rem}\label{remparabolicobstc}
With trivial modifications, Proposition \ref{produ} can be applied
in the case where there is an obstacle in $\mathbb{R}^n$, as in
problem (\ref{eqmatano}) below (see also Remark \ref{remmatanodu}).
\end{rem}

As a corollary to the above proposition, we can give a short proof
of a Liouville type result in \cite{duMaSqueeeze} (see Theorem 1.1
therein), where a squeezing argument involving boundary blow-up
solutions (recall the discussion related to \cite{brezisLiouville}
at the end of Remark \ref{remDoublingOptimal}) was used instead (see
also \cite{duMa}, \cite{duBook}).
\begin{cor}
Let $\lambda \in (-\infty,\infty)$, $p>1$, and $u\in
C^2(\mathbb{R}^n)$ be a nonnegative solution of
\[
\Delta u=u^p-\lambda u\ \ \textrm{in}\ \mathbb{R}^n.
\]
Then, the solution $u$ must be a constant.
\end{cor}
\begin{proof}
If $\lambda \leq 0$, we have that $-\Delta u+u^p\leq0$. Since $p>1$,
it follows from Keller-Osserman theory \cite{keller,osserman} that
$u\leq 0$ on $\mathbb{R}^n$ (see Theorem \ref{thmFarina} below).
Hence, in this case, the solution $u$ is identically zero.

If $\lambda>0$, it follows readily  from Proposition \ref{produ}
that either $u\equiv 0$ or $u\equiv \lambda^\frac{1}{p}$ or
$0<u(x)<\lambda^\frac{1}{p}$, $x\in \mathbb{R}^n$. However, the
latter alternative cannot occur, because of the second assertion of
Corollary \ref{corpacard}.

The proof of the proposition is complete.
\end{proof}
\begin{rem}\label{remdegiorgidu}
Our method of proof,  as well as that of
\cite{duMaSqueeeze,duMa,duBook}, work for a broader class of
nonlinearities. In the special case of the Allen-Cahn equation
\begin{equation}\label{eqallendu}
\Delta u=u^3-u\ \ \textrm{in}\ \mathbb{R}^n,
\end{equation}
by making use of Kato's inequality and Keller-Osserman theory,
it was shown in \cite{brezisMa} (see also \cite{farinaDiffInt},
\cite{ma}) that all solutions of this equation satisfy $|u(x)|\leq
1$, $x\in \mathbb{R}^n$ (for different proofs, see Lemma 1 in
\cite{carbou} and Lemma 4.1 in \cite{chenDG}). A parabolic version
of this result can be found in \cite{ma}.

The importance of the above results is that they imply that there is
no need for the boundedness assumption is the well known statement
of the famous De Giorgi's conjecture: \emph{Let $u$ be a bounded
solution of equation (\ref{eqallendu}) such that $u_{x_n} > 0$. Then
the level sets $\{u = \lambda \}$ are all hyperplanes, \textbf{at
least} for dimension $n\leq 8$.} The motivation behind this
conjecture came from the classical Bernstein problem in the  theory
of minimal surfaces, which explains the restriction in the
dimension. There has been tremendous activity in the last years, and
this conjecture has been completely resolved in dimensions $n\leq 3$
(see \cite{ghosubGui}, \cite{ambrosio3D}; see also
\cite{modicaDeGiorgi} for an earlier related proof in two dimensions
under additional assumptions), and essentially in dimensions $4\leq
n\leq 8$ (assuming that $u\to \pm 1$ pointwise as $x_n\to \pm
\infty$, see \cite{savin} and also \cite{wangSavin}), while a
counterexample which shows that the conjecture is false for $n\geq
9$ has been constructed in \cite{delpinoAnnals}. For more details,
we refer the interested reader to the review article
\cite{farinaState} (some more recent proofs in two and three
dimensions can also be found in \cite{savinValdinoci}).

\end{rem}
\section{Nonexistence of nonconstant solutions with Neumann boundary conditions}
 \label{secmatano}

In this section, motivated from a Liouville-type theorem in
\cite{matanoObstacle}, we will consider some situations where the
equation $\Delta u =W'(u)$, in a possibly unbounded domain, with
Neumann boundary conditions, has only the (obvious) constant
solution.

 \subsection{A Liouville theorem arising in the study of
traveling waves around an obstacle} In Theorem 6.1 of their article
\cite{matanoObstacle}, H. Berestycki, F. Hamel, and H. Matano proved
the following Liouville type result:

\begin{thm}\label{thmmatano}
Let $\Omega$ be a smooth, open, connected subset of $\mathbb{R}^n$,
$n\geq 2$, with outward unit normal $\nu$, and assume that
$K=\mathbb{R}^n\backslash \Omega$ is compact. Let $\mu_-\leq u \leq
\mu$ be a classical solution of
\begin{equation}\label{eqmatano}
\left\{\begin{array}{ll}
         \Delta u=W'(u) & \textrm{in}\ \Omega, \\
           &   \\
          \nu \nabla u=0 & \textrm{on} \ \partial \Omega, \\
          &  \\
        u(x)\to \mu   & \textrm{as}\ |x|\to \infty,
       \end{array}
 \right.
\end{equation}
where $W\in C^2$ satisfies conditions \textbf{(a'')} (defined prior
to Lemma \ref{lem1Sign}) with $W(\mu_-)=0$ allowed, and
(\ref{eqW''}). If $K$ is star-shaped, then
\begin{equation}\label{eqmatanoassert}
u\equiv \mu\ \ \textrm{on}\ \bar{\Omega}.
\end{equation}
\end{thm}

In the study conducted in \cite{matanoObstacle} the set $K$ plays
the role of an obstacle. The prototypical example for the $W$ in the
above theorem is (\ref{eqAllen}) (in that case we have $\mu_-=-1$
and $\mu=1$).

Below, we will provide an alternative proof of the above theorem. We
remark that the statement in \cite{matanoObstacle}, adapted to our
notation, also requires that $W(\mu_-)>0$ and $W'(\mu_-)=0$. In our
statement, we assume that $W\in C^2$ in order to apply the implicit
function theorem to the equation in (\ref{eqgidasEq}). Nevertheless,
with just a slight modification, our proof works also for Lipschitz
$W'$ (see Remark \ref{remmatanoLipschitz} below), as was the
original assumption in \cite{matanoObstacle}. Moreover, as we will
see in the same remark, we can easily dispense of assumption
(\ref{eqW''}).

  Loosely speaking, the approach of
\cite{matanoObstacle} consists in using a sweeping family of lower
solutions of (\ref{eqmatano}), having as building block the solution
$\textbf{U}$ of (\ref{eqU}). Our proof is in the same spirit, but we
build  lower solutions out of one dimensional solutions of
(\ref{eqgidasEq}), capitalizing on the results of Subsection
\ref{subsectionBalls}. In our opinion, our proof is simpler (having
knowledge of Lemma \ref{lem1} and Proposition \ref{proUniq} herein)
and more intuitive. In particular, our proof of Theorem
\ref{thmmatano2} below is considerably simpler than the
corresponding one of \cite{matanoObstacle}, and does not require
that $W'$ is non-decreasing near $\mu$. As will become apparent from
the proofs, the main advantage of our approach is that we use lower
solutions that stay away from $\mu$ (individually). In contrast, the
lower solutions of \cite{matanoObstacle} tend to $\mu$, as $|x|\to
\infty$, causing technical difficulties.

\texttt{Proof of Theorem \ref{thmmatano}:} Up to a shift of the
origin, one can assume without loss of generality that $K$, if not
empty, is star-shaped with respect to $0$. By the strong maximum
principle and Hopf's boundary point lemma \cite{Gilbarg-Trudinger}
(keep in mind that $W'(\mu_-)\leq 0$), and the asymptotic behavior
of $u$ as $|x|\to \infty$, we deduce that
\[\inf_{x\in \bar{\Omega}}u(x)>\mu_-.\]

Under our current assumptions on $W$, it is easy to see that
analogous assertions to those of Lemma \ref{lem1} hold for
minimizers of the energy $J(v;B_R)$ with $v-\mu_--\delta\in
W^{1,2}_0(B_R)$, where $\delta>0$ is chosen sufficiently small so
that $\mu_-+\delta<\inf_{\bar{\Omega}}u$ and $W'(\mu_-+\delta)\leq
0$ (the point is that we have $W(\mu_-+\delta)>0$; if $W(\mu_-)>0$
then we can take $\delta=0$). This is also the case with Proposition
\ref{proUniq}. Abusing notation, we will still denote these
minimizers by $u_R$. From Proposition \ref{proUniq}, there exists an
$R_0>0$ such that these $u_R$'s with $n=1$ are non-degenerate for
$R\geq R_0$ (abusing notation again). Thus, via the implicit
function theorem (see \cite{kielhoferBook}), we can find a
continuous  family of such minimizing solutions $u_R$ (for $R\geq
R_0$, with respect to the uniform topology, as described in
Corollary 2.2 in \cite{Jang}); in this regard, see also Remark
\ref{remmatanoextension} below. Increasing the value of $R_0$, if
necessary, we may assume that
\begin{equation}\label{eqmatanoW'}
W'\left(u_R(0) \right)\leq 0,\ \ R\geq R_0,
\end{equation}
recall  (\ref{eqW''}), $\mu_-+\delta\leq u_R<\mu$,
(\ref{eqestimRect}), and (\ref{eqmonotonicity}). By virtue of the
asymptotic behavior in (\ref{eqmatano}), it follows at once that
there exists a large $T>R_0$ such that
\begin{equation}\label{eqR0matano}
u(x)>u_{R_0}(0)=\max_{\bar{B}_{R_0}}u_{R_0},\ \ x\in
\mathbb{R}^n\backslash B_{(T-R_0)},\ \ \textrm{and}\ \
\bar{K}\subset B_{(T-R_0)},
\end{equation}
(this is the main advantage of our proof in comparison to
\cite{matanoObstacle}).
Let
\begin{equation}\label{eqlaura}
\underline{u}_R(r)=\left\{\begin{array}{ll} \mu_-+\delta, & r\in \left(0,\max\{T-R,0\} \right), \\
& \\
                         u_R(r-T), & r\in \left[\max\{T-R,0\}, T\right], \\
                           &   \\
                         u_R(0), &r> T,
                       \end{array}
 \right.
\end{equation}
with $r=|x|$, $x\in \mathbb{R}^n\backslash \{0\}$. Since
$u_R'(0)=0$, it follows that $\underline{u}_R\in C^1(\bar{\Omega})$.
Using the equation in (\ref{eqgidasEq}) (with $n=1$), we find that
\begin{equation}\label{eqmatanolower0}
\Delta \underline{u}_R- W'(\underline{u}_R)=\left\{\begin{array}{ll}
-W'(\mu_-+\delta),& r\in \left(0,\max\{T-R,0\} \right), \\
& \\
                                                     \frac{n-1}{r}u_R'(r-T), & r\in \left(\max\{T-R,0\}, T\right], \\
                                                      &  \\
                                                     -W'\left(u_R(0)\right), & r> T.
                                                   \end{array}
 \right.
\end{equation}
In particular, recalling that $W'(\mu_-+\delta)\leq 0$,
(\ref{eqmonotonicity}) and (\ref{eqmatanoW'}), we find that
\begin{equation}\label{eqmatanolower}\underline{u}_R\ \ \textrm{is\ a\
weak\ lower\ solution\ of} \ (\ref{eqEqnobdry})\ \textrm{in}\
\Omega,\ \textrm{if}\ R\geq R_0.
\end{equation}

We claim that
\begin{equation}\label{eqmatanoclaim} \underline{u}_R\leq u \ \
\textrm{on}\ \ \bar{\Omega}, \ \textrm{for\ all}\ R\geq R_0.
\end{equation}
Suppose that the claim is false, and let  \[R_*=\sup \left\{ R>R_0\
:\ \underline{u}_s<u\  \textrm{on}\ \bar{\Omega},\ s\in
(R_0,R)\right\}<\infty,\] (recall (\ref{eqR0matano}) which implies
that the set of such numbers $s$ is nonempty). The set
$\bar{\Omega}$ is not compact, so there need not be a point $x\in
\bar{\Omega}$ for which $\underline{u}_{R_*}(x)=u(x)$.
 However, there exists a
sequence of points $x_i \in \bar{\Omega}$ such that
$\underline{u}_{R_*}(x_i)-u(x_i)$ tends to zero as $i\to \infty$.
Since $u(x)\to \mu$ as $|x|\to \infty$, whereas
$\underline{u}_{R_*}(x)\to u_{R_*}(0)<\mu$ as $|x|\to \infty$, it
follows at once that the $x_i$'s remain bounded (this is the main
advantage of our proof in comparison to \cite{matanoObstacle}).
Passing to a subsequence, we find that $x_i\to x_*\in \bar{\Omega}$
with $\underline{u}_{R_*}(x_*)=u(x_*)$.
 In
view of (\ref{eqmatano}) and (\ref{eqmatanolower}), we find that
\begin{equation}\label{eqmatanoQ}
\Delta (u-\underline{u}_{R_*})\leq Q(x)(u-\underline{u}_{R_*})\ \
\textrm{weakly\ in}\ \Omega,
\end{equation}
where $Q$ is a continuous function of the form (\ref{eqQ}). Since
$u\geq u_{R_*}$ on $\bar{\Omega}$, the weak Harnack inequality (see
 \cite{Gilbarg-Trudinger}) tells us that the point $x_*$ must lie
on the boundary of $\Omega$ (otherwise, $\underline{u}_{R_*}\equiv
u$ which is not possible by (\ref{eqmatanolower0})); note also that
at $x_*$ we have that $\underline{u}_{R_*}$ is smooth so we can
apply the strong maximum principle. Since $x_*\in
\partial \Omega=
\partial K$, by (\ref{eqmatanoQ}) and Hopf's boundary point lemma, we get that
\begin{equation}\label{eqmatanoHopf}
0>\nu \nabla (u-\underline{u}_{R_*})=-\nu \nabla
\underline{u}_{R_*}=-(\nu\cdot
x_*)\frac{\underline{u}_{R_*}'(|x_*|)}{|x_*|}\ \ \textrm{at}\ x_*,
\end{equation}
(here $\nu=\nu_{x_*}$ denotes the outward unit normal to $\partial
\Omega$ at $x_*$). On the other hand, since $K$ is star-shaped with
respect to the origin, we have that
\begin{equation}\label{eqstarshapedInner}
x\cdot\nu_x\leq 0,\ \ x\in \partial K.
\end{equation}
Also, relation (\ref{eqmonotonicity}) implies that
\[
\underline{u}_{R_*}'(|x_*|)=u_{R_*}'(|x_*|-T)>0,\ \ x\in
\mathbb{R}^n\backslash \{0\}.
\]
The above two relations contradict (\ref{eqmatanoHopf}).
Consequently, claim (\ref{eqmatanoclaim}) holds.

Now, letting $R\to \infty$ in (\ref{eqmatanoclaim}), thanks to
(\ref{eqestimRect}), we arrive at the sought for relation
(\ref{eqmatanoassert}).

The proof of the theorem is complete. \ \ \ $\square$

\begin{rem}
In dimension $n=1$, with  $K$ a closed bounded interval, the same
arguments can be adapted straightforwardly, and the conclusion of
Theorem \ref{thmmatano} continues to  hold. In this special case,
however, it is better to use Proposition \ref{prohalfmatano} below.
\end{rem}

\begin{rem}\label{remmatanoLipschitz}
One can avoid making use of Proposition \ref{proUniq} in the proof
of Theorem \ref{thmmatano}, and thus have its validity for $W'$
Lipschitz, as follows. Given $\epsilon\in (0,\mu-\mu_-)$ such that
$W'\leq 0$ on $[\mu-\epsilon,\mu]$, we can find $R>0$ and
 minimizer $u_R$ (as in the above proof) such that $u_R(0)\geq
 \mu-\epsilon$ (this assertion of Lemma \ref{lem1} holds for $W\in
 C^{1,1}$). For such $R$ and $u_R$, let $T>R$ be such that
 (\ref{eqR0matano}) holds with $R$ in place of $R_0$, namely $u\geq
 \underline{u}_{R,T}$ $\forall \ T> R$, where $\underline{u}_{R,T}$
 as in (\ref{eqlaura}) (with the obvious meaning). Now, in contrast
 to the proof of Theorem \ref{thmmatano}, we can let $T\to 0$ (perform
 sliding) and find that $u\geq \mu-\epsilon$ on $\bar{\Omega}$.
 Since $\epsilon$ can be taken arbitrarily small, we conclude that
 $u\equiv \mu$, as desired. The same argument can also be applied to
 Theorem \ref{thmmatano2} below, but does not seem to be usable in
 Propositions \ref{promatanointerior1}-\ref{promatanointerior2}
 that follow.

 Clearly, as $T$ decreases, the functions $u$ and
 $\underline{u}_{R,T}$ cannot touch at an $x\in \mathbb{R}^n\setminus
 \bar{B}_T$ (nor at an $x \in \bar{B}_{T-R}$, if $T>R$, as a matter of fact). Hence, there
 is no need for imposing (\ref{eqW''}) in order to have that $W'\left(u_R(0) \right)<0$
for large $R>0$. In other words, the assumption (\ref{eqW''}) can
also be removed from the statement of Theorem \ref{thmmatano}.
\end{rem}

\begin{rem}
If we plainly use an $n$-dimensional minimizer from Lemma \ref{lem1}
(minimizing over $(\mu_-+\delta)+W_0^{1,2}(B_R)$), making use of
Proposition \ref{proUniq}, and the sliding argument, we can
potentially simplify the proof of the related Proposition 2.1 in
\cite{bouhours}.
\end{rem}

\begin{thm}\label{thmmatano2}
If in Theorem \ref{thmmatano} we assume that $N\geq 1$ and  the
obstacle $K$  to be directionally convex instead of star--shaped,
then conclusion (\ref{eqmatanoassert}) still holds.
\end{thm}
\begin{proof}
Without loss of generality, we may assume that $K$ is convex in the
$x_1$ direction, which implies that
\begin{equation}\label{eqnablaconvex}
(x_1,\cdots,0)\nu_x\leq 0\ \ \forall\ x=(x_1,\cdots,x_n)\in \partial
K,
\end{equation}
where $\nu$ denotes again the unit outer normal to $\partial \Omega$
(i.e. inner to $\partial K$).
 The proof proceeds along the same lines as that of Theorem
\ref{thmmatano}. As in the latter theorem, let $u_R$ denote a
minimizing solution to the equation in (\ref{eqgidasEq}) with $n=1$
and $u_R=\mu_-+\delta$ on $\partial B_R$ (this $\delta>0$ is
completely analogous  to the one in the proof of Theorem
\ref{thmmatano}). For $R,\ T>0$, let
\[
\underline{u}_R(x)=\left\{
\begin{array}{ll}
  u_R(x_1-T), & x_1 \in \left(\max\{T-R,0\}, T\right), \\
    &   \\
u_R(x_1+T), & x_1 \in \left(-T, \min\{-T+R,0\}\right), \\
&\\
u_R(0),&|x|\geq T,\\
&\\
  \mu_-+\delta, & \textrm{otherwise}.
\end{array}
\right.
\]
From the equation in (\ref{eqgidasEq}) and (\ref{eqmonotonicity}),
we have that $\underline{u}_R$ is a weak lower solution of
(\ref{eqEqnobdry}) in $\Omega$ (see again \cite{Berestyckilion}). As
before, there exist large $R_0,\ T>R_0$ such that
$\underline{u}_{R_0}<u$ on $\bar{\Omega}$, and the minimizers $u_R$
vary smoothly with respect to $R\geq R_0$.

We claim that
\begin{equation}\label{eqmatanoclaim2}
\underline{u}_R\leq u \ \ \textrm{on} \ \bar{\Omega}\ \ \textrm{for\
all}\ \ R\geq R_0.
\end{equation}
Arguing by contradiction, as in the proof of Theorem
\ref{thmmatano}, we get the existence of analogous  $R_*>R_0$ and
$x_*\in \bar{\Omega}$ (as before, the corresponding sequence
$\{x_i\}$ is bounded). To reach a contradiction, it boils down to
exclude the case $x_*\in
\partial K$. Firstly, note that $x_*$ cannot be on the hyperplane
$\{x_1=0\}$. Indeed, in that case, we would have $R_*>T$, and
observe that the function
\[
g(t)=(u-\underline{u}_{R_*})(x_*+te),\ \ e=(1,\cdots,0),
\]
would be well defined for small $|t|$ and
\[
g'(0^-)-g'(0^+)=u_{R_*}'(-T)-u_{R_*}'(T)
=-2u_{R_*}'(T)\stackrel{(\ref{eqmonotonicity})}{>}0,
\]
which is not possible because $g$ has a global minimum at $t=0$.
Now, since $x_*\in \partial \Omega\backslash \{x_1=0 \}$, Hopf's
boundary point lemma tells us that (\ref{eqmatanoHopf}) holds. On
the other hand, recalling (\ref{eqmonotonicity}), at the point $x_*$
we have that
\[
(x_1,\cdots ,0)\nabla \underline{u}_{R_*}= x_1
\partial_{x_1}\underline{u}_{R_*}=
\left\{\begin{array}{ll}
         x_1 u_{R_*}'(x_1-T) & \textrm{if}\ \ 0<x_1<T, \\
           &   \\
         x_1 u_{R_*}'(x_1+T) & \textrm{if}\ \ -T<x_1\leq 0.
       \end{array}
 \right.
\]
Hence, relation (\ref{eqmonotonicity}) yields that $ (x_1,\cdots
,0)\nabla \underline{u}_{R_*}\geq 0\ \ \textrm{at}\ x_*. $ However,
from (\ref{eqmatanoHopf}) and the latter relation, we get that
\[
\nu \nabla \underline{u}_{R_*}=\left\{\begin{array}{ll}
                            \nu\cdot (x_1,\cdots,0)\frac{1}{x_1}u_{R_*}'(x_1-T) & \textrm{if}\ x_*\in \partial\Omega\cap \{x_1>0\}, \\
                              &  \\
                              \nu\cdot (x_1,\cdots,0)\frac{1}{x_1}u_{R_*}'(x_1+T) & \textrm{if}\ x_*\in \partial\Omega\cap
                              \{x_1<0\},
                          \end{array}
 \right.
\]
at $x_*$, i.e., $\nu \nabla \underline{u}_{R_*}\leq 0$ at $x_*$; a
contradiction. We have therefore shown that claim
(\ref{eqmatanoclaim2}) holds.

Letting $R\to \infty$ in (\ref{eqmatanoclaim2}), as before, we
arrive at (\ref{eqmatanoassert}).

The proof of the theorem is complete.
\end{proof}
\begin{rem}\label{remmatanodu}
If in addition $W$ satisfies relations (\ref{eqduScale}), and
(\ref{eqdumatano}) with $\mu_+=\mu$, then there is no need to assume
that $\mu_-\leq u\leq \mu$ in the assertions of Theorems
\ref{thmmatano}, \ref{thmmatano2} (recall the proof of Proposition
\ref{produ}).
\end{rem}
\begin{rem}
In Theorems \ref{thmmatano} and \ref{thmmatano2}, we assumed that
the obstacle is smooth for the purposes of applying Hopf's boundary
point lemma. In this regard, we refer to \cite{gidas} for a
generalization of the latter lemma to domains with corners (see also
the proof of our Proposition \ref{procarbouMine} below).
\end{rem}

\begin{rem}\label{remmatanoextension}
If $W$ satisfies \textbf{(a'')} and $W'(t)<0$,  $t\in (\mu_-,\mu)$,
then the assertions of Theorems \ref{thmmatano} and \ref{thmmatano2}
can be proven easily (recall Remark \ref{remparabolicobstc}).
\end{rem}

\subsection{A Liouville-type theorem in a convex epigraph}

Adapting the proof of Theorem \ref{thmmatano}, using the
$n$-dimensional $u_R$, we can show the following proposition. In the
special case where $\Omega$ is the half-space $\mathbb{R}_+^n$, this
proposition  will come in handy when dealing with a class of mixed
boundary value problems in Section \ref{secmixed}; in fact,  in this
special case, this proposition is contained in Theorem
\ref{thmrigidity} below (via a reflection argument)).

\begin{pro}\label{prohalfmatano}
Assume that $W\in C^2$ satisfies condition \textbf{(a'')} with
$W(\mu_-)=0$ allowed. Let $\Omega$ be an entire epigraph of the form
(\ref{eqOmegaBCN}), with $\varphi$ convex  and $\|\nabla
\varphi\|_{C^{1,\alpha}(\mathbb{R}^{n-1})}\leq C$, for some
$\alpha\in (0,1)$ and $C>0$.
 Then $u\equiv \mu$ is the only classical
solution (in $C^2(\bar{\Omega})$) to the  problem
\begin{equation}\label{eqNeumanMat}
\left\{\begin{array}{ll}
         \Delta u=W'(u) & \textrm{in}\ \Omega, \\
           &   \\
          \nu \cdot \nabla u=0   & \textrm{on}\ \partial \Omega,
       \end{array}
 \right.
\end{equation}
where $\nu$ denotes $\partial \Omega$'s outer unit normal, such that
$\mu_-\leq u \leq \mu$ and
\begin{equation}\label{equnifdog}u(x',x_n)\to \mu,\ \textrm{uniformly\ in}\
\mathbb{R}^{n-1},\ \textrm{as}\  x_n-\varphi(x')\to \infty.
\end{equation}
\end{pro}

\begin{proof}
As before, by the strong maximum principle and Hopf's boundary point
lemma, we deduce that $u>\mu_-$ on $\overline{\Omega}$. In fact, we
claim that
\begin{equation}\label{eqdoginf}
\inf_{x\in \overline{\Omega}}u(x)>\mu_-,
\end{equation}
(this is not needed in the case where $W(\mu_-)>0$). To show this,
we will argue by contradiction (in the spirit of Lemma 3.4 in
\cite{cafareliPacard}, see also \cite{farinaAdvMath} and
\cite{angenent,ghosubGui}), namely assume that
\[
u(y_j',y_j)\to \mu_-\ \textrm{for\ some}\ y_j'\in \mathbb{R}^{n-1}\
\textrm{and} \ y_j\geq \varphi(y_j').
\]
Note that (\ref{equnifdog}) implies that there exists an $L>0$ such
that
\begin{equation}\label{eqlowerdog}
u(x',x_n)\geq \frac{\mu_-+\mu}{2} \ \textrm{if}\ x'\in
\mathbb{R}^{n-1}\ \textrm{and}\ x_n\geq \varphi(x')+L.
\end{equation}
It follows that $\varphi(y_j')\leq y_j\leq\varphi(y_j')+ L$ for
large $j$, and, passing to a subsequence, we find that
\begin{equation}\label{eqYinfty}y_j-\varphi(y_j')\to Y_\infty\in [0,L].\end{equation}
Using the Neumann boundary conditions, as in \cite{niTakagi}, we can
extend $u$ to a $C^2$ function in a neighborhood of $\Omega$. Then,
applying interior regularity estimates (see
\cite{Gilbarg-Trudinger}), we infer that there exists a constant
$C>0$ such that
\[
\|u\|_{C^{2,\alpha}(\bar{\Omega})}\leq C,
\]
(for this, as explained in \cite{farinaAdvMath}, it is important
that $\|\varphi\|_{C^{1,\alpha}(\mathbb{R}^{n-1})}$ is finite). By
Lemma 6.37 in \cite{Gilbarg-Trudinger}, we can take $\tilde{u}\in
C^{2,\alpha}(\mathbb{R}^n)$ to be a smooth extension of $u$, that is
$\tilde{u}=u$ in $\Omega$, such that
\[
\|\tilde{u}\|_{C^{2,\alpha}(\mathbb{R}^n)}\leq
C\|u\|_{C^{2,\alpha}(\bar{\Omega})}\leq C,
\]
see also \cite{farinaAdvMath}. Now, let
\[
v_j(z)=\tilde{u}(z'+y_j',z_n+y_j),\ z=(z',z_n)\in \mathbb{R}^n.
\]
Each $v_j$ solves (\ref{eqNeumanMat}) in
\begin{equation}\label{eqphij}\Omega_j=\left\{z_n>\varphi_j(z')\equiv \varphi(z'+y_j')-y_j
\right\},\end{equation} there exists $C>0$ such that
\begin{equation}\label{eqvjC2}
\|{v_j}\|_{C^{2,\alpha}(\mathbb{R}^n)}\leq C \ \textrm{for\ every}\
j,
\end{equation}
\begin{equation}\label{eqvjmu}
\mu_-\leq v_j\leq \mu,\ \textrm{and}\ v_j(0,0)\to \mu_-.
\end{equation}
Moreover, thanks to (\ref{eqlowerdog}), we have
\begin{equation}\label{eqlowerdogII}
v_j\left(0,\varphi(y_j')-y_j+L
\right)=u\left(y_j',\varphi(y_j')+L\right)\geq \frac{\mu_-+\mu}{2}.
\end{equation}
In view of (\ref{eqvjC2}), and the usual diagonal argument, passing
to a subsequence, we find that
\[
v_j\to v_\infty\  \ \textrm{in}\
C_{loc}^{2,\alpha}\left(\mathbb{R}^n\right),
\]
for some  $v_\infty$ with
$\|{v_\infty}\|_{C^{2,\alpha}(\mathbb{R}^n)}\leq C$ (this $\alpha$
is in fact the same as in (\ref{eqvjC2}), see Lemma 6.1.6 in
\cite{henry}), and
\begin{equation}\label{eqvinfty}
\mu_-\leq v_\infty\leq \mu\  \textrm{in}\ \mathbb{R}^n,\
\textrm{and}\ v_\infty(0,0)=\mu_-,
\end{equation}
(recall (\ref{eqvjmu})). Furthermore, in view of (\ref{eqYinfty})
and (\ref{eqlowerdogII}), we get that
\begin{equation}\label{eqvinftygeq}
v_\infty(0,L-Y_\infty)\geq \frac{\mu_-+\mu}{2},
\end{equation}
which implies that $v_\infty$ is not identically equal to $\mu$ in
$\Omega_\infty$. From (\ref{eqYinfty}), and (\ref{eqphij}), we find
that
\[
\varphi_j(0)\to -Y_\infty \in [-L,0].
\]
Moreover, we have
\[
\|\nabla \varphi_j\|_{C^{1,\alpha}(\mathbb{R}^{n-1})}= \|\nabla
\varphi\|_{C^{1,\alpha}(\mathbb{R}^{n-1})}\leq C\ \ \forall\ j.
\]
From the above two relations, via Arczela-Ascoli's theorem, we
conclude that, for a subsequence,
\begin{equation}\label{eqphijposter}\textrm{the}\ \varphi_j \ \textrm{converge\
in}\ C^2_{loc}(\mathbb{R}^{n-1})\ \textrm{to\ a\ function}\
\varphi_\infty\in C^2(\mathbb{R}^{n-1}),\end{equation} which is also
convex. We write
\[
\Omega_\infty=\left\{x_n>\varphi_\infty(x')\right\}.
\]
We have that $v_\infty$ is a solution to (\ref{eqNeumanMat}) in
$\Omega_\infty$. Indeed, if $x=(x',x_n)\in \Omega_\infty$, then
$x_n>\varphi_\infty (x')$. It follows that $x_n>\varphi_j (x')$ for
large $j$, and thus
\[
\Delta v_\infty(x)=\lim_{j\to \infty}\Delta v_j(x)=\lim_{j\to
\infty}W'\left(v_j(x)\right)=W'\left(v_\infty(x)\right).
\]
Let $\nu_j(x)$ denote the outer unit normal vector to $\partial
\Omega_j$ at $x\in \partial \Omega_j$, and $\nu_\infty(x)$ the
corresponding vector to $\partial \Omega_\infty$. We have, for
$\nu_\infty$ and $\nu_j$ evaluated at
$\left(x',\varphi_\infty(x')\right)$ and
$\left(x',\varphi_j(x')\right)$ respectively, that
\[
\begin{array}{rcl}
  \left|\nu_\infty \nabla v_\infty\left(x',\varphi_\infty(x') \right)\right| & \leq & \left|\nu_\infty\nabla v_\infty\left(x',\varphi_\infty(x') \right)-\nu_\infty\nabla
v_j\left(x',\varphi_\infty(x')
\right)\right|\\
&&\\ &&+\left|\nu_\infty\nabla
v_j\left(x',\varphi_\infty(x')\right)-\nu_\infty\nabla
v_j\left(x',\varphi_j(x')  \right) \right| \\&&\\
    &   &+\left|\nu_\infty \nabla
v_j\left(x',\varphi_j(x') \right)-\nu_j \nabla
v_j\left(x',\varphi_j(x') \right) \right| +\left|\nu_j \nabla
v_j\left(x',\varphi_j(x') \right) \right| \\
  &  &   \\
    & \stackrel{(\ref{eqvjC2})}{\leq}  &
\sup_{B_1\left(x',\varphi_\infty(x') \right)}|\nabla v_\infty-\nabla
v_j|+C\left|\varphi_\infty(x')-\varphi_j(x')
\right|\\
&&\\
&&+C\left|\nu_\infty\left(x',\varphi_\infty(x')\right)-\nu_j\left(x',\varphi_j(x')\right)
\right|,
\end{array}
\]
where, in turn, the last term can be estimated as
\[\begin{array}{lll}
    \left|\nu_\infty\left(x',\varphi_\infty(x')\right)-\nu_j\left(x',\varphi_j(x')\right)
\right| & \leq &
\left|\nu_\infty\left(x',\varphi_\infty(x')\right)-\nu_\infty\left(x',\varphi_j(x')\right)
\right|\\
&&+\left|\nu_\infty\left(x',\varphi_j(x')\right)-\nu_j\left(x',\varphi_j(x')\right)
\right| \\
     &  &  \\
     & \leq & C\left|\varphi_\infty(x')-\varphi_j(x')
\right|+\sup_{B_1\left(x',\varphi_\infty(x')
\right)}\left|\nu_\infty(y)-\nu_j(y) \right|,
  \end{array}
\]
where we used that $\|\nabla
\varphi_\infty\|_{C^{1,\alpha}(\mathbb{R}^n)}\leq C$, the fact that
$\nu_\infty$ and $\nu_j$ are functions of $\nabla \varphi_\infty$
and $\nabla \varphi_j$ respectively, and (\ref{eqphijposter}).
 Hence, by letting $j\to \infty$, we deduce that $v_\infty$
satisfies Neumann boundary conditions on $\partial \Omega_\infty$.
On the other hand, in view of (\ref{eqvinfty}), recalling that
$W'(\mu_-)\leq 0$, the strong maximum principle and Hopf's boundary
point lemma (applied in the equation for $v_\infty-\mu_-$ in
$\Omega_\infty$) imply that $v_\infty\equiv \mu_-$; a contradiction
to (\ref{eqvinftygeq}). Thus, relation (\ref{eqdoginf}) holds.

Let $u_R$ be as in Theorems \ref{thmmatano}-\ref{thmmatano2}, but
with $B_R\subset \mathbb{R}^n$, $R>0$, and
\[u_R(R)=\mu_-+\delta<\inf_{x\in \overline{\Omega}}u(x),\]
where $\delta>0$ is also chosen so that $W'(\mu_-+\delta)\leq 0$
(the point is that $W(\mu_-+\delta)> 0$). By virtue of
(\ref{equnifdog}), there exists a large $M>\max_{|x'|\leq R}\varphi
(x')+R$ such that
\begin{equation}\label{eqguiR}
u(x)>u_{R}(0)\geq u_{R}(x-Q),\ \ x\in B_{R}(Q),\ \textrm{where}\
Q=(0,\cdots,M).
\end{equation}
 Now, consider the family of functions:
\[\underline{u}_{R,P}(x)=\underline{u}_R(x-P),\ \ P\in \Omega,\ (R>0\ \textrm{fixed\ but\ arbitrary}),\] where $\underline{u}_R$ is defined
by\begin{equation}\label{eqdogbranch}
\underline{u}_R=\left\{\begin{array}{ll}
                         u_R, & x\in B_R,  \\
                          &  \\
                         \mu_-+\delta, & \textrm{otherwise}.
                       \end{array}
 \right.
\end{equation}
Firstly, note that $\underline{u}_{R,P}$ is a weak lower solution to
the equation in (\ref{eqNeumanMat}) (see \cite{Berestyckilion}, and
also recall our first proof of Theorem \ref{thmmine}). Moreover, if
$x\in \partial \Omega$ with $|x-P|<R$, then
\[
\frac{\partial \underline{u}_{R,P}}{\partial \nu}=\nu \cdot \nabla
u_R(x-P)=\nu \cdot \frac{(x-P)}{|x-P|}u_R'\left( |x-P|\right)\leq 0\
\ \textrm{at}\ x,
\]
by (\ref{eqmonotonicity}) and the convexity of $\Omega$. Whereas, if
$x\in \partial \Omega$ with $|x-P|>R$, then
\[
\frac{\partial \underline{u}_{R,P}}{\partial \nu}= 0\ \ \textrm{at}\
x.
\]
In view of (\ref{eqguiR}), and the above observations,  keeping $R$
fixed, starting from $P=Q$, we can slide $B_R(P)$, $P\in \Omega$, to
obtain that
\[
u(x)\geq u_R(0),\ \ x\in \bar{\Omega},\ R>0,
\]
(keep in mind that $u$ and $\underline{u}_{R,P}$ cannot touch on
$\partial B_R(P)$). Then, letting $R\to \infty$ in the above
relation, via the obvious analog of (\ref{eqestimRect}), recalling
that $u\leq \mu$, we conclude that $u\equiv \mu$, as desired.

The proof of the proposition is complete.
\end{proof}


\begin{rem}\label{rempointwise}
The assertion of Proposition \ref{prohalfmatano} remains true if the
uniform convergence in (\ref{equnifdog}) is replaced by pointwise,
provided that we do not allow  $W(\mu_-)$ to be zero. Indeed, the
pointwise convergence, the boundedness of $u$, and Arczela-Ascoli's
theorem, imply that, given $R>0$, we have
\[u\to \mu,\ \textrm{uniformly\ in}\
B_R',\ \textrm{as}\  x_n\to \infty,
\]
where $B_R'$ denoted the ball of radius $R$ in $\mathbb{R}^{n-1}$
with center at the origin (see the last part of the proof of Theorem
1.1 in \cite{gui}). In fact, since $\mu-u\geq 0$, this property can
also be shown by means of Harnack's inequality in the linear
equation for $\mu-u$, which implies that
\[
\sup_{B_R(x)}(\mu-u)\leq C(R)\inf_{B_R(x)}(\mu-u)\ \ \forall x\in
\mathbb{R}^n \ \textrm{and}\ R>0\ \textrm{such\ that}\ B_R(x)\subset
\Omega,
\]
 (see Lemma 2.3 in
\cite{jerison}). Actually, the latter approach only requires that $
u(x',x_n^j)\to \mu$ for some $x'\in \mathbb{R}^{n-1}$ and a sequence
$x_n^j\to \infty$ as $j\to \infty$. In fact, it follows readily from
the proof of Proposition \ref{prohalfmatano} that the latter
property holds if and only if
\[
\sup_{\mathbb{R}^{n}_+}u=\mu.
\]
 Consequently, relation (\ref{eqguiR}) holds true,
with the corresponding minimizer $u_R$ such that $u_R(R)=\mu_-$.
Now, because $W(\mu_-)>0$, the latter minimizer satisfies the
assertions of Lemma \ref{lem1} and Proposition \ref{proUniq}.
\end{rem}

\begin{rem}\label{remrgidityfarmatano}
If $W'(t)\leq 0$, $t\in [\mu_-,\mu]$, then the special case of
Proposition \ref{prohalfmatano}, where $\Omega$ is a half-space
(namely $\varphi$ is a constant), follows easily, via a reflection
argument, from Proposition 2.4 in \cite{farinaRigidity}.
\end{rem}

\subsection{The case of smooth, bounded, star-shaped
domains}\label{secneumannbded} In analogy to Theorem
\ref{thmmatano}, we have
\begin{pro}\label{promatanointerior1}
Let $\Omega$ be a smooth bounded domain of $\mathbb{R}^n$, $n\geq
1$, with outward unit normal $\nu$, which is star-shaped with
respect to some point $x_0\in \Omega$. Let $\mu_-\leq u \leq \mu$ be
a classical solution to (\ref{eqNeumanMat}), where $W\in C^2$
satisfies conditions \textbf{(a'')} with $W(\mu_-)=0$ allowed, and
(\ref{eqW''}). There exist numbers $R_0,\epsilon_1>0$, depending
only on $W$, such that if $\bar{B}_{R_0}(x_0)\subset \Omega$ and
$u(x)>\mu-\epsilon_1$ on $\bar{B}_{R_0}(x_0)$ then $u\equiv \mu$.
\end{pro}
\begin{proof}
The proof of this proposition is in the spirit of that of Theorem
\ref{thmmatano}. By the strong maximum principle and Hopf's boundary
point lemma (see \cite{Gilbarg-Trudinger}), recalling that
$W'(\mu_-)\leq 0$, we deduce that
\[
\min_{x\in \bar{\Omega}}u(x)>\mu_-,
\]
keeping also in mind that $\bar{\Omega}$ is compact (compare with
(\ref{eqdoginf})). Again, we may assume without loss of generality
that $x_0=0$.

Let $u_R$, $R\geq R_0$, be the radial minimizers that we used in
Proposition \ref{prohalfmatano}. As before, the functions in
(\ref{eqdogbranch}) fashion  a smooth family of weak lower solutions
to (\ref{eqNeumanMat}) for $R\geq R_0$. Let
$\epsilon_1=\mu-u_{R_0}(0)$.

Suppose that $u$ is as stated in the proposition with the above
choices of $R_0,\epsilon_1$ (and $x_0=0$).  Clearly, we have
\[
u>\underline{u}_{R_0}\ \ \textrm{on}\ \bar{\Omega}.
\]
Now, similarly to Proposition \ref{propacard}, we do ``ballooning''.
As $R>R_0$ increases, there are three possibilities. The first one
is that there exists some $R_*>R_0$ and an $x_*\in \Omega$ such that
$\underline{u}_R<u$ on $\bar{\Omega}$ for $R\in [R_0,R_*)$ and
$\underline{u}_{R_*}(x_*)=u(x_*)$. The second possibility is the
same as the first but with $x_*\in \partial \Omega$. The third
possibility is that $u$ and $\underline{u}_R$ never touch, namely
$u>\underline{u}_R$ on $\bar{\Omega}$ for every $R\geq R_0$. We make
the following observations. The first scenario cannot occur because
of the strong maximum principle. In the case that the last scenario
holds, letting $R\to \infty$ and recalling Lemma \ref{lem1} (for
these $u_R$'s), we infer that $u\equiv \mu$ as desired. Therefore,
it remains to exclude the second scenario, namely that
$\underline{u}_R$ and $u$ first touch at a point $x_*\in \partial
\Omega$ when $R=R_*>R_0$; note that $0<|x_*|<R_*$. To this end, we
will argue by contradiction and assume that it holds. Note first
that relation (\ref{eqmatanoHopf}) remains unchanged
(notation-wise). Analogously to (\ref{eqstarshapedInner}), we have
$x\nu_x \geq 0$, $x\in
\partial \Omega$. Keeping in mind that, in the case at hand, we have
\[
\underline{u}_{R_*}'(|x_*|)=u_{R_*}'(|x_*|)\stackrel{(\ref{eqmonotonicity})}{<}0,
\]
we get a contradiction.

The proof of the proposition is complete.
\end{proof}
\begin{rem}\label{remmatanoconvex}
If $\Omega$ is bounded, smooth and \emph{convex},  there are no
non-constant stable solutions to (\ref{eqNeumanMat}) for \emph{any}
$W$ (see \cite{casten} and \cite{matano}).
\end{rem}

In analogy to Theorem \ref{thmmatano2}, we can show
\begin{pro}\label{promatanointerior2}
Let $\Omega$ be a smooth bounded domain of $\mathbb{R}^n$, $n\geq
1$, with outward unit normal $\nu$, which is directionally
star-shaped with respect to some direction $x_i$, $i=1,\cdots,n$.
Let $\mu_-\leq u \leq \mu$ be a classical solution of
(\ref{eqNeumanMat}), where $W\in C^2$ satisfies conditions
\textbf{(a'')} with $W(\mu_-)=0$ allowed, and (\ref{eqW''}). There
exist numbers $R_2,\epsilon_2>0$, depending only on $W$, such that
if
$u(x)>\mu-\epsilon_2$ on $\bar{\Omega}\cap \{|x_i|\leq R_2 \}$, then
$u\equiv \mu$.
\end{pro}


\section{Extensions: Multiple ordered solutions}\label{secExtensions}
Suppose that $W:\mathbb{R}\to \mathbb{R}$ is $C^2$ and there are
positive numbers
\[
\mu_1<\cdots<\mu_m,\ \ m\geq 2,
\]
such that
\[
W(\mu_1)>\cdots>W(\mu_m), \ \ W'(0)\leq 0,\ \ W'(\mu_i)=0,\ \
i=1\cdots,m,
\]
and
\[
W(t)>W(\mu_i), \ \ t\in [0, \mu_i),\ i=1,\cdots, m.
\]
Note that at the points $\mu_i$, the potential $W$ has either minima
or saddles. Obviously, we can extend $W$ outside of $[0,\mu_i]$, to
a $C^2$ potential $\tilde{W}_i$, in such a way that condition
\textbf{(a')} is satisfied with $\tilde{W}_i(t)-W(\mu_i)$ in place
of $W$ and $\mu_i$ in place of $\mu$, $i=1,\cdots,m$.
Next, consider  any
\begin{equation}\label{eqe1}
\epsilon\in \left(0,
\min_{i=1,\cdots,m}(\mu_i-\mu_{i-1})\right),\end{equation} with the
convention that $\mu_0=0$, and any
\begin{equation}\label{eqDi}
 D_i>D_i'\ \ \textrm{where}\  \ D_i'\  \ \textrm{solve}\ \
\textbf{U}_i\left(D_i'\right)=\mu_i-\epsilon,\ \ i=1,\cdots,m,
\end{equation}
where
\begin{equation}\label{eqinner}
\textbf{U}_i''(s)=W'\left(\textbf{U}_i(s)\right),\  s>0;\ \
\textbf{U}_i(0)=0,\ \lim_{s\to\infty}\textbf{U}_i(s)=\mu_i.
\end{equation}
By means of Theorem \ref{thmmine}, there exist positive numbers
$R'_i>D_i$, depending only on $\epsilon$, $D_i$, $\tilde{W}_i$,
$i=1,\cdots, m$, and $n$, such that if $\Omega$ has  nonempty
$C^2$-boundary and  contains a closed ball of radius $R'_i$ then
there exists a solution $u_i$ of
\begin{equation}\label{eqEqi}
\Delta u=\tilde{W}_i'(u),\ x\in \Omega,\ \ u(x)=0,\ x\in
\partial \Omega,
\end{equation}
satisfying
\begin{equation} \label{eq12i} 0<u_i(x)<\mu_i,\ \ x\in \Omega,
\end{equation}
and
\begin{equation}\label{eq14+i}
\mu_{i-1}<\mu_i-\epsilon<u_i(x),\ \ x\in
\Omega_{R_i'}+B_{(R'_i-D_i)},\ \ i=1,\cdots,m.
\end{equation}
In view of (\ref{eq12i}), we conclude that $u_i$ solves the original
problem (\ref{eqEq}). Thus, given $\epsilon$ and $D_i$ as in
(\ref{eqe1}) and (\ref{eqDi}) respectively, if $\Omega$ contains a
closed ball of radius $R''$, where $R''=\max_{i=1,\cdots,m}{R_i'}$,
we find that (\ref{eqEq}) has at least $m$ positive solutions which
satisfy (\ref{eq12i})--(\ref{eq14+i}). Moreover, keeping in mind
Remark \ref{remstable}, we know that these solutions are stable.

These solutions may be chosen to be ordered, in the usual sense. In
other words, given $\epsilon$ and $D_i$ as in (\ref{eqe1}) and
(\ref{eqDi}) respectively, there are at least $m$ positive, stable
solutions of (\ref{eqEq}) such that
 \begin{equation}\label{eqorder}u_1(x)< \cdots<u_m(x),\ \ x\in
\Omega,\ \ 1\leq i\leq m,\end{equation} and
(\ref{eq12i})--(\ref{eq14+i}) hold (we have chosen to keep the same
notation).
 Indeed, the solution $u_{i}$
can be captured by using the constant function $\mu_i$ as an upper
solution; and the function $\max\{u_{i-1}(x),\ \underline{u}^i_P\}$
as lower solution, where $\underline{u}^i_P$ is the lower solution
in (\ref{eqlower}) but with $\tilde{W}_i(t)-W(\mu_i)$ in place of
$W(t)$, $i=1,\cdots, m$, and $u_0\equiv 0$. (We use again
Proposition 1 in \cite{Berestyckilion}, see also Proposition 1 in
\cite{kurata}, to make sure that it is a lower solution). As in the
first proof of Theorem \ref{thmmine}, we can sweep with the family
of lower solutions $\underline{u}^i_Q,\ Q\in \Omega_{R_i'}$ to
extend the lower bound on $u_i$ (due to (\ref{eqestimRect})) from
$B_{(R_i'-D_i)}(P)$ to $\Omega_{R'_i}+B_{(R'_i-D_i)}$. Moreover, the
strong inequalities in (\ref{eqorder}) follow from the strong
maximum principle. Naturally, the obtained solutions are stable
(recall Remark \ref{remstable}).

We have just proven the following:

\begin{thm}\label{thmhess1}
Suppose that $\Omega$ and $W$ are as described in this section. Let
$\epsilon$ and $D_i$ be as in (\ref{eqe1}) and (\ref{eqDi})
respectively. There exist positive constants $R_i'>D_i$,
$i=1,\cdots, m$, depending only on $\epsilon$, $D_i$, $W$ and $n$,
such that if $\Omega$ contains a closed ball of radius
$R''=\max_{i=1,\cdots, m}R_i'$, then problem (\ref{eqEq}) has at
least $m$ stable solutions $u_i$, ordered as in (\ref{eqorder}),
such that (\ref{eq12i})--(\ref{eq14+i}) hold true.
\end{thm}

On the other hand, assuming that $\Omega$ is bounded and smooth (a
$C^3$ boundary suffices), the theory of monotone dynamical systems
(see Theorem 4.4 in \cite{matano}) guarantees the existence of at
least $m-1$ unstable solutions $\hat{u}_{i},\ i=1,\cdots, m-1$, of
(\ref{eqEq}) such that
\begin{equation}\label{eqorder2}
u_i(x)<\hat{u}_{i}(x)<u_{i+1}(x),\ \ x\in \Omega,\ \ i=1,\cdots,
m-1.
\end{equation}
This can also be shown by the well known mountain pass theorem, see
\cite{defiguerdo}.


In summary, we have
\begin{thm}\label{thmhess2}
Suppose that, in addition to the hypotheses of Theorem
\ref{thmhess1}, the domain $\Omega$ is assumed to be smooth and
bounded. Then, besides of the $m$ stable solutions $u_i$ that are
provided by Theorem \ref{thmhess1}, there exist at least $m-1$
unstable solutions $\hat{u}_i$ of (\ref{eqEq}), ordered as in
(\ref{eqorder2}) (keep in mind (\ref{eqorder})).
\end{thm}
The above theorem extends an old result of P. Hess \cite{hess}, in
the context of nonlinear eigenvalue problems (which are included in
our setting, see below), where the additional assumption that
$W'(0)<0$ was imposed (see also \cite{brown} for an earlier result
in the case $n=1$). In the same context, the case $W'(0)=0$ was
allowed in \cite{defiguerdo}, at the expense of assuming that
$W'(\mu_i)\neq0,\ i=1,\cdots,m$, and some geometric restrictions on
the domain. All these references considered  nonlinear eigenvalue
problems of the form
\begin{equation}\label{eqEV1}
\Delta u=\lambda^2 W'(u),\ x\in \mathcal{D},\ \ u(x)=0,\ x\in
\partial \mathcal{D},
\end{equation}
where $\mathcal{D}$ is a smooth bounded domain of $\mathbb{R}^n$. By
stretching variables $x\mapsto \lambda^{-1}x$, assuming that $0\in
\mathcal{D}$ (this we can do without loss of generality), keeping
the same notation, we are led to the equivalent problem:
\begin{equation}\label{eqEV2} \Delta u=W'(u),\ x\in \Omega,\
\ u(x)=0,\ x\in \partial \Omega,
\end{equation}
where $\Omega\equiv \lambda \mathcal{D}$, for $\lambda>0$, which is
plainly problem (\ref{eqEq}). If $\lambda$ is sufficiently large,
then certainly the domain $\Omega$ contains the ball $B_{R''}$,
appearing in the assertion of Theorem \ref{thmhess1}, but not the
other way around. In contrast to our approach of using upper and
lower solutions, De Figueiredo in \cite{defiguerdo} obtained the
corresponding stable solutions  as minimizers of the associated
energy functionals (with $W$ suitably modified outside of
$[0,\mu_i],\ i=1,\cdots,m$), and a geometric condition had to be
imposed on the domain in order to ensure that they are distinct for
large $\lambda$. In our case, the fact that they are distinct
follows at once from (\ref{eq12i}) and (\ref{eq14+i}).
 As we have already pointed out,
in \cite{defiguerdo}, the unstable solutions were constructed as
mountain passes (saddle points of the energy).
\begin{rem}\label{remuniqDancer}
It has been proven in \cite{dancerProcLondon} that if $W'(t) <0,\
t\in (0,\mu) $, $W'(0) <0$, or $W'(0) = 0$ but $W''(0)<0$,
$W'(\mu)=0$, and $W''\geq 0$ near $\mu$, then (\ref{eqEV1}), with
$\mathcal{D}$ smooth and bounded, has a unique  solution with values
$(0,\mu)$ when $\lambda$ is large, see also \cite{angenent}.
\end{rem}
\begin{rem}\label{remlayer}
If $\mathcal{D}$ is a bounded domain with  $C^2$-boundary, it
follows from the proof of Theorem \ref{thmmine} that the
corresponding  stable solutions of (\ref{eqEV1}), provided by
Theorem \ref{thmhess1}, develop a boundary layer of size
$\mathcal{O}(\lambda^{-1})$, as $\lambda \to \infty$, along the
boundary of $\mathcal{D}$ (see Proposition \ref{prolayer} below for
more details, and compare with the proof of Theorem 1.1 in
\cite{yanedinburg}, as well as with Theorem 4 in \cite{defiguerdo}
and Lemma 2 in \cite{kurata}). Loosely speaking, this means that the
stable solutions $u_i$ converge uniformly to $\mu_i$ on the domain
$\mathcal{D}$ excluding the strip that is described by
$\textrm{dist}(x,\partial \mathcal{D})\leq
|\ln\lambda|^\alpha|\lambda^{-1},\ \alpha>0$, as $\lambda\to
\infty$. It follows from (\ref{eqorder2}) that the corresponding
unstable solutions of (\ref{eqEV1}), provided by Theorem
\ref{thmhess2}, also develop a (local) boundary layer behavior. In
fact, if $W''(\mu_i)>0$, the fine structure of the boundary layer of
the stable solution  $u_i$  is determined by the unique solution of
the problem (\ref{eqinner}),  see \cite{angenent} and Remark
\ref{remhalfplane} below.  On the other side, under some
restrictions on $\mathcal{D}$ and $W$, unstable solutions possessing
an upward sharp spike layer  on top of $u_i$, located near the most
centered part of the domain, have been constructed in
\cite{dancerSpikes}, \cite{dancerWei2}, and \cite{Jang} (see also
\cite{dancerMorseTrans}). The fine structure of this interior spike
layer is determined by the problem
\[
\Delta V=W'(V+\mu_i) \ \ \textrm{in}\ \mathbb{R}^n;\ \ V(x)\to 0,\
|x|\to \infty.
\]
\end{rem}
\section{On the boundary layer of global minimizers of singularly perturbed elliptic
equations}\label{seclayer} In this section, assuming only
$\textbf{(a')}$, we will prove a general result on the size of the
boundary layer of solutions of (\ref{eqEV1}), which minimize the
associated energy functional, as $\lambda\to \infty$ (recall also
Remark \ref{remlayer}). Setting $\varepsilon=\lambda^{-1}\to 0$,
gives rise to a singular perturbation problem of the form
\begin{equation}\label{eqsingular}
\varepsilon^2\Delta u=W'(u),\ \ x\in \mathcal{D};\ \ u(x)=0,\ x\in
\partial \mathcal{D},
\end{equation}
and in this regard it might be helpful to recall Remark
\ref{rembelowNew}.

 We emphasize that, in contrast to previous results in this
 direction, as Theorem 1.1 in \cite{yanedinburg}, here the size of the boundary
layer is shown to be
\emph{independent of the dimension $n$}. This is due to our previous
improvement over Lemma 2.2 in \cite{yanedinburg} that was made in
Lemma \ref{lem1} herein (recall the discussion preceding it, and
also see Remark \ref{remYan} below). The point is that we have not
assumed any nondegeneracy on $W$ at $\mu$; in the case where
$W''(\mu)>0$ or $n=2$, the structure of the boundary layer is well
understood (recall Remark \ref{remlayer} and see Remark
\ref{remDancer} below). For a different possible approach to this,
see Remark \ref{remhalfplane} below.

The main result in this section is
\begin{pro}\label{prolayer}
Suppose that $\mathcal{D}$ is a bounded domain in $\mathbb{R}^n$,
$n\geq1$, with $C^2$-boundary, and let $W$ satisfy assumption
\textbf{(a')}.
Consider any $\epsilon\in (0,\mu)$ and $D>D'$, where $D'$ as in
(\ref{eqD}). There exists a positive constant $\lambda_*$, depending
only on $\epsilon$, $D$, $\mathcal{D}$, and $W$, such that there
exists a solution $u_\lambda$ of (\ref{eqEV1}), which minimizes the
associated energy functional, satisfies
\begin{equation}\label{eqyu}
0<u_\lambda(x)<\mu,\ \ x\in \mathcal{D},
\end{equation}
 and
\begin{equation}\label{equlambda}
u_\lambda(x)\geq \mu-\epsilon,\ x\in
\bar{\mathcal{D}}_{(D\lambda^{-1})},
\end{equation}
provided that $\lambda \geq \lambda_*$ (recall the definition
(\ref{eqomegaR}), and note that ${\mathcal{D}}_{(D\lambda^{-1})}$ is
a connected domain for large $\lambda$). (See also the comments at
the end of the assertion of Lemma \ref{lem1}).\end{pro}
\begin{proof}
As in the second proof of Theorem \ref{thmmine}, recalling the
discussion leading to (\ref{eqEV2}), there exists a smooth solution
of (\ref{eqEV1}), which minimizes the associated energy and
satisfies (\ref{eqyu}), provided that $\lambda$ is sufficiently
large, say $\lambda\geq \lambda_0$, depending not just on $W$ but
this time also on the domain $\mathcal{D}$.

Since $\partial \Omega\in C^2$, we know that $\Omega$ satisfies the
interior ball condition (see \cite{Gilbarg-Trudinger}). In other
words, there exists a radius $r_0>0$ and a family of balls
$B_{r_0}(q)\subseteq\mathcal{D}$, $q\in
\partial \mathcal{D}_{r_0}$ (i.e. $q\in \mathcal{D}$ with
$\textrm{dist}(q,\partial \mathcal{D})=r_0$) such that, for each
such $q$, the closed ball $\bar{B}_{r_0}(q)$ touches $\partial
\mathcal{D}$ at exactly one point.

Let $\epsilon\in (0,\mu)$ and $D>D'$, where $D'$ as in (\ref{eqD}).
It follows from Lemma \ref{lem1} (after a simple rescaling) that
there exists a $\lambda_*>0$, depending only on $\epsilon$, $D$,
$W$, and $\mathcal{D}$ (in terms of $r_0$), and a global minimizer
$u_{r_0,q}$ of the associated energy to the equation of
(\ref{eqEV1}) in $W^{1,2}_0\left(B_{r_0}(q) \right)$  such that
\[0<u_{r_0,q}(x)<\mu,\ x\in B_{r_0}(q),\ \ \textrm{and}\ \
u_{r_0,q}(x)\geq \mu-\epsilon,\ \ x\in
\bar{B}_{(r_0-D\lambda^{-1})}(q),
\]
provided that $\lambda\geq \lambda_*$. (Without loss of generality,
we may assume that $\lambda_*>\lambda_0$). Thanks to Lemma
\ref{lemdancer} below, we obtain that $u_\lambda(x)\geq
u_{r_0,q}(x), \ x\in B_{r_0}(q)$.
 Since the center $q$ was any point on $\partial \mathcal{D}_{r_0}$, it
follows that assertion (\ref{equlambda}) holds true for $x\in
\mathcal{D}$ such that
\begin{equation}\label{eqlayerlead}{D\lambda^{-1}}\leq
\textrm{dist}(x,\partial \mathcal{D})\leq 2r_0-D\lambda^{-1}.
\end{equation}
 If
$W'(t)<0,\  t\in [\mu-2\epsilon, \mu)$, then the validity of
(\ref{equlambda}), over  the entire  specified domain, follows at
once via the second assertion of Lemma \ref{lemalikakos} (this is
also the case when relation (\ref{eqWmonotone}) holds, recall Remark
\ref{remmonotbelowLem1}). Otherwise, we proceed as follows, see also
Lemma 2 in \cite{kurata}: Firstly, we cover
$\bar{\mathcal{D}}_{r_0}$ by a finite number of balls of radius
$\frac{r_0}{2}$ with centers on $\bar{\mathcal{D}}_{r_0}$. Secondly,
if necessary, we increase the value of $\lambda_*$ such that
$D\lambda_*^{-1}<\frac{r_0}{2}$. Lastly, we apply Lemma
\ref{lemdancer} to show that
\[
u_\lambda(x)\geq u_{r_0,p}(x)\geq \mu-\epsilon,\ \ x\in
\bar{B}_{(r_0-D\lambda^{-1})}(p)\supseteq
\bar{B}_{\frac{r_0}{2}}(p),
\]
for every center $p$ of the finite covering of
$\bar{\mathcal{D}}_{r_0}$, if $\lambda\geq \lambda_*$. We point out
that this last part could  have also been obtained from the weaker
relation (\ref{eqsweers}) (with the obvious modifications).
 The desired
estimate (\ref{equlambda}) now follows from the comments leading to
(\ref{eqlayerlead}) and the above relation.

The proof of the proposition is complete.
\end{proof}
\begin{rem}\label{remlayer2}
A similar result also  holds if the domain $\mathcal{D}$ is
unbounded.
\end{rem}

\begin{rem}\label{remDancer}
The asymptotic behavior, as $\lambda \to \infty$, of uniformly
bounded from above and below (with respect to $\lambda$), stable
solutions of (\ref{eqEV1}), where $\mathcal{D}\subseteq \
\mathbb{R}^n$ is bounded and smooth,  has been studied  in
\cite{dancerMorseTrans} in dimensions $n=2,3$ by techniques related
to the proof of De Giorgi's conjecture in low dimensions. For a
related result in  $\mathbb{R}^4$, see \cite{dupaigneBook}. In fact,
since global minimizers are stable, and since assumption
\textbf{(a')} implies that $W'(0)\leq 0$, the assertions of
Proposition \ref{prolayer} when $n=2$ follow readily from Theorem 6
in \cite{dancerMorseTrans}; this is also the case when $n=3$,
provided that the monotonicity assumption \textbf{(b)} from our
introduction is imposed.
\end{rem}
\begin{rem}\label{remYan}
 Let $\epsilon, \ D,\ R'>0$ be related as
in the assertion of Lemma \ref{lem1}. By means of a simple rescaling
argument (see also the proof of Theorem 1.1 in \cite{yanedinburg}),
Lemmas \ref{lem1} and \ref{lemdancer} yield that the solution of
(\ref{eqEV1}), described in Proposition \ref{prolayer}, satisfies
$\mu-u_\lambda(x)\geq \epsilon$, if $\textrm{dist}(x,\partial
\mathcal{D})>D\lambda^{-1}$, provided that $\lambda$ is sufficiently
large (depending on $\epsilon$, $W$, and $\mathcal{D}$). Note that
relation (\ref{eqsweers}) yields the same estimate but over the
smaller region that is described by $\textrm{dist}(x,\partial
\mathcal{D})>\frac{R'}{2}\lambda^{-1}$,
 which \emph{depends
on $n$}, see \cite{yanedinburg}.
\end{rem}

\begin{rem}\label{remhalfplane}
Let $x_0\in
\partial \mathcal{D}\in C^2$ and $\mathcal{R}$ denote the matrix in
$SO(N,\mathbb{R})$ that rotates the vector $(0,\cdots,0,1)$ onto the
inner normal to $\partial \mathcal{D}$ at $x_0$. We can extract a
sequence of $\lambda\to \infty$ such that any global minimizer
$u_\lambda$, provided by Proposition \ref{prolayer}, satisfies
\[u_\lambda\left(x_0+\lambda^{-1}\mathcal{R} y\right)\to U(y), \]
uniformly on compacts, as $\lambda\to \infty$, where $U$ is some
nonnegative, global minimizer (in the sense of (\ref{eqJer}), this
can be seen as in page 104 of \cite{danceryanCVPDE}) of the
following half-space problem
\[
\Delta u=W'(u),\ y\in \mathbb{R}^n_+;\ \ u(y)=0,\ y\in \partial
\mathbb{R}^n_+,
\]
see \cite{angenent}, \cite{dancerMorseTrans} for more details, where
$\mathbb{R}_+^n=\{(y_1,\cdots,y_n)\ :\ y_n> 0\}$. Furthermore, this
solution is nontrivial by virtue of Remark \ref{remYan}. Hence, by
the strong maximum principle, recall \textbf{(a')}, we deduce that
$U$ is positive in $\mathbb{R}_+^n$. As before, combining Lemmas
\ref{lem1} and \ref{lemdancer},  we obtain that
\[
u(y)\to \mu\ \ \textrm{as}\ \ y_n\to \infty,\ \textrm{uniformly\ in}
\ (y_1,\cdots,y_{n-1})\in \mathbb{R}^{n-1},
 \]
(the weaker assertion (\ref{eqsweers}) is sufficient for this). It
follows from Theorem 1.4 in \cite{berestyckiCaffarelli} that $U$
depends only on the $y_n$ variable and therefore coincides with
$\textbf{U}(y_n)$ that was described in (\ref{eqU}). (If
$W''(\mu)>0$ then this has been shown earlier in \cite{angenent},
see also \cite{berestDuke}, \cite{clemente} for the weaker case
(\ref{eqW''}) and Proposition \ref{prohalfclemMine} below).

\end{rem}
\begin{rem}\label{remshibata}
In \cite{shibata}, the author established  an asymptotic expansion
of $\nu \nabla u_\varepsilon(P)$, $P\in \partial \mathcal{D}$, as
$\varepsilon\to 0$, where $u_\varepsilon$ solves (\ref{eqsingular})
for a class of nonlinearities which in particular satisfy
\textbf{(c)} and (\ref{eqpacardW'}) (see also \cite{clemente} and
\cite{sabina}; see also Remark \ref{remshiSaddle} below). As usual,
the vector $\nu$ denotes the unit outer normal to $\partial
\mathcal{D}$ (having assumed that it is smooth and bounded). This
expansion reveals that if $P_1$
 is the only point which attains the minimum of the mean
curvature of $\partial \mathcal{D}$, then $P_1$ is the steepest
point of the boundary layer.
\end{rem}

\begin{rem}\label{remSingular}
By adapting the proof of Lemma 2.3 in \cite{yanedinburg}, and that
of our Proposition \ref{prolayer}, we can study the boundary layer
of globally minimizing solutions of inhomogeneous singular
perturbation problems of the form
\begin{equation}\label{eqinhomogeneous}
\varepsilon^2\Delta u=W_u(u,x),\ x\in \mathcal{D};\ \ u(x)=0,\ x\in
\partial \mathcal{D},
\end{equation}
as $\varepsilon \to 0$, for appropriate righthand side that is more
general than those that were considered in
\cite{berger1,berger2,devillers,yanedinburg}, see also Lemma 7.13 in
\cite{duBook} and Section 13.3 in \cite{lionsLNMSU} (roughly, we
want \textbf{(a')} to hold with $a(x)$ instead of $\mu$,  for every
fixed $x\in \bar{\mathcal{D}}$, for a smooth positive function $a$).
\end{rem}

\section{The singular perturbation problem with mixed boundary value
conditions}\label{secmixed} Let $\mathcal{D}$ be a bounded domain in
$\mathbb{R}^n$ with $C^2$-boundary. Suppose that $\partial
\mathcal{D}=\Gamma_N\cup \Gamma_D$, where $\Gamma_N$ and $\Gamma_D$
are closed and nonempty. We consider the following mixed boundary
value problem:
\begin{equation}\label{eqmixed}
\left\{\begin{array}{ll}
         \Delta u=\lambda^2W'(u) & \textrm{in}\ \mathcal{D}, \\
          &  \\
         \frac{\partial u}{\partial \nu}=0 &   \textrm{on}\ \Gamma_N,\\
          &  \\
         u=0 & \textrm{on}\ \Gamma_D,
       \end{array}
 \right.
\end{equation}
where $\lambda>0$ is a large parameter, and $\nu$ is the unit
outward normal to $\Gamma_N$ at $x\in \Gamma_N$. Denote
\[
W_{0,\Gamma_D}^{1,2}(\mathcal{D})=\left\{u\in W^{1,2}(\mathcal{D})\
:\ u=0\ \textrm{on}\ \Gamma_D \right\}.
\]
Under assumption \textbf{(a')} on $W\in C^2$, as before,  the energy
functional
\[
I(u)=\int_{\mathcal{D}}^{}\left[\frac{1}{2}|\nabla
u|^2+\lambda^2W(u) \right]dx,\ \ u\in
W_{0,\Gamma_D}^{1,2}(\mathcal{D}),
\]
has a global minimizer $u_\lambda$ such that $0\leq u_\lambda\leq
\mu$ (do not confuse with the usual  radial minimizer $u_R$).
Moreover, as in the second proof of Theorem \ref{thmmine}, we have
that $u_\lambda$ is nontrivial for large $\lambda$. It is more or
less standard that $u_\lambda$ fashions a weak solution to
(\ref{eqmixed}) (see Chapter 5 in \cite{brenerscot}). Then, from the
theory in \cite{stamp}, it follows that $u_\lambda$ is a classical
solution.

Similarly to Theorem \ref{thmmine}, exploiting Proposition
\ref{prohalfmatano}, we have the following result.

\begin{pro}\label{promixed}
Assume  $\mathcal{D}$, $W$, and $u_\lambda$, as above. Given
$\epsilon\in (0,\mu)$, there exist positive constants $\lambda_*,M$
such that
\begin{equation}\label{eqlowermixed}
u_\lambda(x)\geq \mu-\epsilon \ \ \textrm{if}\ \
\textrm{dist}(x,\Gamma_D)\geq M\lambda^{-1}\ \textrm{and}\ \lambda
\geq \lambda_*.
\end{equation}
\end{pro}
\begin{proof}
By using Lemma \ref{lemdancer} below, and sliding around a radial
minimizer of radius $\lambda^{-1} R$ (with $R$ fixed large, as
dictated by Lemma \ref{lem1}), we infer that there exists a constant
$C>0$ such that
\begin{equation}\label{eqang}u_\lambda(x)\geq \mu-\epsilon \ \ \textrm{if}\ \
\textrm{dist}(x,\partial\mathcal{D})\geq
C\lambda^{-1},\end{equation} provided that $\lambda$ is sufficiently
large.

Suppose that the assertion of the proposition  is false. Then, there
exist $\lambda_j\to \infty$ and $x_j\in \bar{\mathcal{D}}$ such that
$u_j=u_{\lambda_j}$ satisfies
\begin{equation}\label{eqtem}
u_j(x_j)<\mu-\epsilon\ \ \textrm{and}\ \
\lambda_j\textrm{dist}(x_j,\Gamma_D)\to \infty.
\end{equation}
By virtue of (\ref{eqang}), we deduce that the numbers
\begin{equation}\label{eqnitakagbded}\lambda_j\textrm{dist}(x_j,\partial\mathcal{D})\  \textrm{remain\
bounded\ as}\ j\to \infty.
\end{equation}
We may assume that $x_j\to x_\infty\in \partial \mathcal{D}$. Take
the diffeomorphism $y=\Psi(x)$ which straightens a boundary portion
near $x_\infty$, as  in relation (2.8) of \cite{niTakagi}. We may
assume that $\Phi=\Psi^{-1}$ is defined in an open set containing
the closed ball $\bar{B}_{2\kappa}$, $\kappa>0$, and that
$y_j=\Psi(x_j)\in B_\kappa^+=B_\kappa \cap \{y_n>0\}$ for all
$\kappa>0$. As in \cite{niTakagi}, let
\[
v_j(y)=u_j\left(\Phi(y) \right)\ \textrm{for}\ y\in
\bar{B}_{2\kappa}^+.
\]
By the properties of this transformation (see \cite{niTakagi}), we
know that
\begin{equation}\label{eqnineuman}
\frac{\partial v_j}{\partial{y_n}}=0\ \textrm{on}\ \Psi(\Gamma_N).
\end{equation}
 Moreover, we define a scaled function
\[
w_j(y)=v_j(y_j+\lambda_j^{-1}y)\ \textrm{for}\ y\in \bar{B}_{\kappa
\lambda_j}.
\]
In view of (\ref{eqnitakagbded}), passing to a subsequence, we may
assume that the $n$-th coordinate of $\lambda_j y_j$ converges to
$\ell\geq 0$, while the remaining coordinates ``get away'' from
$\Psi(\Gamma_D)$ as $j\to \infty$. To be more precise, given $R>0$,
we have that \[\left(y_j+\lambda_j^{-1}\bar{B}_{R}\right)\cap
\Psi(\Gamma_D)={\O}\] for $j$ sufficiently large. By interior
elliptic regularity estimates \cite{Gilbarg-Trudinger} (applied
after we have reflected $v_j$ across $\Psi(\Gamma_N)$, recall
(\ref{eqnineuman}) and the above relation), as in \cite{niTakagi},
passing to a subsequence, we find that
\[
w_j\to w \ \textrm{in}\ C^2_{loc}(\mathbb{R}^n_+),
\]
where $w$ satisfies
\[ \Delta w=W'(w)\ \ \textrm{in}\ \
\mathbb{R}^n\cap\{y_n>-\ell\};\ w_{y_n}=0\ \ \textrm{if}\ y_n=-\ell,
\]
see also \cite{angenent}, \cite{dancerMorseTrans},
\cite{gidasSpruck}, and \cite{yanedinburg}. Furthermore, via
(\ref{eqtem}), we have
\begin{equation}\label{eqcontraMixed}
w(0)\leq \mu-\epsilon.
\end{equation}
 Moreover, thanks
to (\ref{eqang}), we have $w(y)\geq \mu-2\epsilon$ if $y_n\geq C'$
for some $C'>0$ (assuming that $2\epsilon<\mu$). This, in
particular, implies that $v$ is nontrivial. Now, as in Remark
\ref{remhalfplane}, we see that $w\to \mu$, uniformly in
$\mathbb{R}^{n-1}$, as $y_n\to \infty$. On the other hand,
Proposition \ref{prohalfmatano} implies that $w\equiv \mu$ which is
in contradiction to (\ref{eqcontraMixed}).

The proof of the proposition is complete.
\end{proof}
\begin{rem}\label{remallencahn}
Let $\Omega$ be a smooth bounded domain in $\mathbb{R}^n$ which is
symmetric with respect to some hyperplane, say $\{x_1=0 \}$, and $W$
as in Proposition \ref{promixed} and \emph{even}. Let
$\mathcal{D}=\Omega\cap \{x_1>0\}$, $\Gamma_N=\partial \Omega \cap
\{x_1\geq 0\}$, and $\Gamma_D=\bar{\Omega}\cap \{x_1=0\}$. Applying
Proposition \ref{promixed}, yields a positive solution to
(\ref{eqmixed}) which satisfies (\ref{eqlowermixed}). (Some care is
required at the junction points on $\partial \Omega \cap \{x_1=0\}$,
but this regularity issue may be treated by an approximation
argument, as described in Remark \ref{remcone}, see also
\cite{dipierro,grisvard}). Reflecting this solution oddly across the
plane $\{x_1=0\}$, we obtain a solution to the Neumann problem
\begin{equation}\label{eqkow}
\Delta u=\lambda^2 W'(u)\ \textrm{in}\ \Omega;\ \frac{\partial u
}{\partial \nu} \ \textrm{on}\ \partial \Omega,
\end{equation}
which converges, in $L^1(\Omega)$, to the step function
\[
\mu \chi_{\Omega\cap \{x_1>0\}}-\mu \chi_{\Omega\cap \{x_1<0\}},
\]
as $\lambda\to \infty$ ($\chi$ denotes the usual characteristic
function).

In the general case, where $\Omega$ is not symmetric, under some
non-degeneracy assumptions, this type of transition-layered
solutions have been constructed in two and three dimensions, via
perturbation arguments, by \cite{sakamotoIbun}, \cite{kowMorse}, and
\cite{sakamotoTaiwan} (see also the references in \cite{pacard}).

If $\Omega\subset \mathbb{R}^2$ is smooth, bounded, and symmetric
with respect to the coordinate axis, in the same manner, we can
construct solutions to (\ref{eqkow}) that converge, in
$L^1(\Omega)$, to the step function
\[
\mu \chi_{\Omega\cap \{x_1x_2>0\}}-\mu \chi_{\Omega\cap
\{x_1x_2<0\}},
\]
as $\lambda\to \infty$ (see also a related open question in
\cite{floresSaddle}). Analogous constructions hold in higher
dimensions, recall our discussion about ``saddle'' solutions from
the introduction.
\end{rem}

The paper \cite{fuscoTrans} contains an analog of Theorem
\ref{thmfusco} for problem (\ref{eqmixed}), with $\lambda>0$ fixed,
in the case where for each $x\in\Gamma_N$ there is a $\rho>0$ such
that $\mathcal{B}_\rho(x)$ is convex, where $\mathcal{B}_\rho(x)$
denotes the connected component of $B_\rho(x)\cap \mathcal{D}$ such
that $x\in \bar{\mathcal{B}}_\rho(x)$. We believe that there is also
a corresponding analog of Theorem \ref{thmmine}. To support this,
let us sketch the proof of the following proposition.
\begin{pro}\label{promixedsketch}
Assume that $W\in C^2$ satisfies \textbf{(a)}, $\lambda>0$, and
$\mathcal{D}$ as in this section with $\Gamma_N$ convex in the above
sense. Given $\epsilon\in (0,\mu)$, there exist $R_*,C>0$ such that
the existence of $x_*\in \Omega$ such that
$\mathcal{B}_{R_*}(x_*)\cap \Gamma_D=\emptyset$ implies that problem
(\ref{eqmixed}) has a positive solution $u<\mu$ verifying
(\ref{eq12}). Moreover, there exists a $C>0$ such that
\[ u(x)\geq \mu-\epsilon,\ \ \textrm{if}\
\mathcal{G}(x,\Gamma_D)\geq C,
\]
here $\mathcal{G}(x,\Gamma_D)$ denotes the geodesic distance of $x$
from $\Gamma_D$, namely
\[
\mathcal{G}(x,\Gamma_D)=\inf_{\gamma(x,\Gamma_D)}\mathcal{H}^1\left(\gamma(x,\Gamma_D)
\right),
\]
where $\mathcal{H}^1$ is the one-dimensional Hausdorff measure and
the infimum is taken on the set of the absolutely continuous paths
$\gamma(x,\Gamma_D)\subset \bar{\mathcal{D}}$ joining $x$ to
$\Gamma_D$.
\end{pro}
\begin{proof}(Sketch) Plainly note that the function
$\underline{u}=u_{R_*}(x-x_*)$, $x\in \mathcal{B}_{R_*}(x_*)$, zero
otherwise, is a lower solution to (\ref{eqmixed}). The key point is
that the convexity property of $\Gamma_N$ implies that
\[
\frac{\partial\underline{u}}{\partial \nu}\leq 0\ \ \textrm{on}\
\Gamma_N.
\]
Then, we can slide around that lower solution in $\mathcal{D}$, as
long as we stay away from $\Gamma_D$ (in the geodesic sense), to
obtain the desired lower bound.
\end{proof}
\section{Some one-dimensional symmetry properties of certain  solutions to the Allen-Cahn equation}\label{secfarina}
\subsection{Symmetry of entire solutions}
Many authors  have studied the one-dimensional symmetry of certain
entire solutions to problem (\ref{eqentire}) with $W$ as in
(\ref{eqAllen}), namely
\begin{equation}\label{eqfar1}
\Delta u+u(1-u^2)=0\ \ \textrm{in}\ \ \mathbb{R}^n.
\end{equation}
Their study was motivated by \emph{De Giorgi's conjecture} (recall
Remark \ref{remdegiorgidu}) and \emph{Gibbons' conjecture}. The
latter claims that any solution to (\ref{eqfar1}) which tends to
$\pm 1$ as $x_1\to \pm \infty$, \emph{uniformly} in
$\mathbb{R}^{n-1}$, is one-dimensional, i.e.
\begin{equation}\label{eqtanh} u(x)=\tanh \left(\frac{x_1-a}{\sqrt{2}} \right)\ \
\textrm{for\ some}\ a\in \mathbb{R}.
\end{equation}

As we have already pointed out, the former conjecture was motivated
from the theory of minimal surfaces. On the other hand, the latter
conjecture was motivated from a problem in cosmology theory (see
\cite{GibbonsTownsed}).
\begin{rem}
Keep in mind that problem (\ref{eqfar1}) is invariant under
translations and rotations.
\end{rem}
Gibbons' conjecture was proven almost at the same time by three
different approaches: in \cite{barlow} by probabilistic arguments,
in \cite{berestDuke} by the sliding method, and in
\cite{farinaRendi} based on \cite{cafareliPacard}.
 In fact, it was
proven earlier for dimensions up to three in \cite{ghosubGui}; see
also \cite{guiAnnals} for a different proof which holds up to
dimension five. The conjecture of Gibbons' also follows from a
stronger result that can be found in
\cite{barlow,cafaCordobaIntermed} (see also \cite{cabreMoschini}).
The latter says that if $u$ solves (\ref{eqfar1}), such that it
possesses an unstable level set (say $u=0$) which is a globally
Lipschitz graph, then $u$ is as in (\ref{eqtanh}) (possibly after a
rotation and translation).

In this section, we will present some related new one-dimensional
symmetry results, based on Proposition \ref{propacard} as well as on
an old result in \cite{cafamodica} which does not seem to have been
exploited up to this moment. In particular, we are able to provide a
new proof of the  Gibbons' conjecture in two dimensions.

After this section was written, we found that the same result of
Theorem \ref{thmfarmine} below, \emph{under the additional
assumption that $W'$ is odd}, was proven previously in \cite{duMa}
(see also \cite{zhao} for a generalization to the quasi-linear
setting, where the oddness assumption is not stated explicitly in
the statement of Theorem 1.3 therein but used in the proof). The
strategy in the latter references was to take advantage of the
oddness of $W'$, adapting some techniques from \cite{cafareliPacard}
and \cite{berestDuke}, to show that the solution under consideration
is odd (in a certain direction); this property then reduces the
one-dimensional symmetry problem to Gibbons' conjecture which was
already  resolved (recall our previous discussion). Moreover, it was
assumed in \cite{duMa} that (\ref{eqW''}) holds (in this regard, see
Remark \ref{remduma} below). On the other side, the proof in
\cite{duMa} holds for $W'$ Lipschitz (see, however, Remark
\ref{remduma} below).
%
%

Our main result is
\begin{thm}\label{thmfarmine}
Let $u\in C^2(\mathbb{R}^n)$ be a solution to (\ref{eqfar1}) such
that there exists a point $P$ on the hyperplane $\{x_1=0 \}$, say
the origin, such that \begin{equation}\label{eqfarP}u(P)=0\
\textrm{and}\ u>0 \ \textrm{in}\ \mathbb{R}^{n}\cap
\{x_1<0\},\end{equation} then $u$ is one-dimensional of the form
(\ref{eqtanh}).
\end{thm}
\begin{proof}
As we showed in Remark \ref{remdegiorgidu}, we have $|u(x)|<1$,
$x\in \mathbb{R}^n$. Similarly to the proof of Proposition
\ref{propacard} (see also Remark \ref{rempacardBrezOz}), we have
\begin{equation}\label{eqfar2}
u_R(x-Q)<u(x),\ \ x\in B_R(Q),
\end{equation}
provided that $B_R(Q)\subset \mathbb{R}^n\cap \{x_1<0 \}$, where
$u_R$ is a solution to problem (\ref{eqgidasEq}) such that
(\ref{eqmaxball}) holds. Since $u$ is positive in $\mathbb{R}^n\cap
\{x_1<0\}$, we can slide the ball $B_R(Q)$ (keeping $R$ fixed) so
that it is tangent to the hyperplane $\{ x_1=0\}$ at the origin,
while keeping (\ref{eqfar2}). In other words, relation
(\ref{eqfar2}) holds, with $Q=(-R,0,\cdots,0)$, for all $R>0$. In
particular, we have
\begin{equation}\label{eqflat2}
u(x_1,0,\cdots,0)>u_R(R-x_1),\ \ x_1\in (-R,0),\ (\textrm{with\ the\
obvious\ notation}),
\end{equation}
and $u(0,\cdots,0)=u_R(R)=0$, for all $R>0$.
By Hopf's boundary point lemma (in the equation for $u-u_R$), we
deduce that
\[
u_{x_1}(0,\cdots,0)< u_R'(R)\ \ \textrm{for\ all}\ R>0,
\]
(clearly $u$ cannot be identically equal to $u_R(x-Q)$ in $B_R(Q)$).
 So, recalling that $u_R'(R)<0$, we arrive at
\[
\left[u_{x_1}(\textbf{0})\right]^2> \left[u_R'(R)\right]^2\ \
\textrm{for\ all}\ R>0,
\]
where $\textbf{0}$ denotes the origin of $\mathbb{R}^n$. Note that
the left-hand side of the above relation does not depend on $R$.
Now, letting $R\to \infty$ (do ballooning), and recalling Lemma
\ref{lemhelly} (see also (\ref{eqchen})), we infer that
\[
\left[u_{x_1}(\textbf{0})\right]^2\geq 2W(0).
\]

On the other hand, it is known that every bounded solution to
(\ref{eqfar1}) satisfies the gradient bound (\ref{eqmodicafarina}),
and the only solutions for which equality is achieved at some point
are one-dimensional of the form (\ref{eqtanh}) (see Theorem 5.1 in
\cite{cafamodica}). The above relation clearly implies that equality
is achieved at $x=\textbf{0}$ (recall that $u(\textbf{0})=0$).
Consequently, the solution $u$ is one-dimensional.

The proof of the theorem is complete.
\end{proof}

It is well known that there is a deep connection between the
``blown-down'' level sets of solutions to (\ref{eqfar1}) and the
theory of minimal surfaces, see for example \cite{alberti},
\cite{delpinoAnnals}, \cite{pacard} and  \cite{savin}. Let us
suggest a naive argument which connects Theorem \ref{thmfarmine} to
the theory of minimal surfaces.
 If $u$ satisfies the assumptions of
Theorem \ref{thmfarmine}, its zero set near the origin is a graph of
$x_1$ over $\mathbb{R}^{n-1}$ ($u_{x_1}(\textbf{0})<0$, thanks to
Hopf's lemma, so we can apply the implicit function theorem) which
is tangent to the plane $\{x_1=0\}$. In the ``blown-down'' problem
(assuming for the sake of our argument that $u$ is a minimizer in
the sense of \cite{jerison}), near the origin, we get two minimal
graphs (the one being a plane) which are tangent at the origin and
the  one is above the other. The strong maximum principle for
minimal surfaces, see \cite{coldingCour} and Lemma 1 in
\cite{schoenDG}, tells us that both surfaces are planes. This
property can be rephrased as saying that, as we translate a
hyperplane towards a minimal surface, the first point of contact
must be on the boundary.

\begin{rem}\label{remduma}
The assertion of Theorem \ref{thmfarmine} remains true for solutions
$u$ of (\ref{eqentire}), with $W\in C^3$ as in Proposition
\ref{propacard}, provided that we assume in advance that
$\mu_-<u<\mu$ and $W(t)\geq 0$, $t\in [\mu_-,\mu]$ ($\mu_-<0<\mu$).
In fact, we can easily relax the regularity of $W$ to $W'$ being
Lipschitz continuous, assuming that $-W'(t)\geq c t$ for $t\geq 0$
near $0$ (instead of $W''(0)<0$, recall Remark
\ref{remdrawbak2222}). To this end, we have to use Lemmas 3.2-3.3 in
\cite{cafareliPacard} instead of our Proposition \ref{propacard},
and the recent  result in \cite{farinaAdvMath} which says that the
gradient bound (\ref{eqmodicafarina}) is still valid for $W'$ merely
Lipschitz continuous.

The assertion of Theorem \ref{thmfarmine} also remains true if $W'$
is Lipschitz continuous, $W(t)\geq 0$ for $t\in [\mu_-,\mu]$,
$W(t)>0$ for $t\in [0,\mu)$, $W(\mu)=0$, and $W'(\mu)=0$, provided
that we additionally assume that $\mu_-\leq u\leq \mu$ in
$\mathbb{R}^n$ and $u(x_1,x')\to \mu$ as $x_1\to -\infty$ for some
$x'\in \mathbb{R}^{n-1}$. Indeed, thanks to Remark
\ref{rempointwise}, the last assumption implies uniform convergence
over compacts of $\mathbb{R}^{n-1}$. We note that, in the case where
$W'(0)>0$, we have to use $u_R$ as in Lemma \ref{lem1Sign} (we can
perform the sliding method even if $\max\{u_R,0\}$ is not a weak
lower solution).
\end{rem}

\begin{rem}
It has been shown recently in \cite{farinaTrans} that any
\emph{energy minimizing} solution (as described in \cite{jerison})
to (\ref{eqfar1}) is one-dimensional provided that it is positive
for $x_1<0$ (it is not required a-priori that the level set of $u$
touches $x_1=0$ at some point).
\end{rem}

\begin{rem}\label{remshiSaddle}
The ballooning and sliding arguments of Theorem \ref{thmfarmine},
together with the gradient bound (\ref{eqmodicafarina}), can give a
different proof of relation (4.5) in \cite{shiCPAM}, namely that any
 saddle solution of (\ref{eqentire}) satisfies $u_{x_1}(0,x_2)\to
\sqrt{2W(0)}$ as $x_2\to \infty$ (see also \cite{cabreCpde},
\cite{fifesaddle}). In fact, we can show this without assuming
(\ref{eqW''}). For $W$'s enjoying the qualitative properties of
(\ref{eqAllen}), the rate of this convergence is exponentially fast
(see \cite{fifesaddle}). In higher dimensions, it has been remarked
in \cite{saddlecabre3solo} that this convergence is of algebraic
rate. Combining our approach with the fact that, for such $W$'s,
there holds
\[
u_R'(R)=-\sqrt{2W(0)}+\frac{(N-1)\int_{0}^{\mu}\sqrt{2\left(W(0)-W(t)\right)}dt}{\sqrt{2W(0)}}R^{-1}+\mathcal{O}(R^{-2})\
\ \textrm{as}\ R\to \infty,
\]
see \cite{shibata}, we may quantify this rate. For example, for the
three-dimensional saddle solution in \cite{alesioNodea}, we get
\[
u_{x_1}(0,x_2,x_3)=-\sqrt{2W(0)}+\mathcal{O}\left(\frac{1}{\sqrt{x_2^2+x_3^2}}\right),
\]
and
\[
u_{x_2}^2(0,x_2,x_3)+u_{x_3}^2(0,x_2,x_3)\leq
\mathcal{O}\left(\frac{1}{\sqrt{x_2^2+x_3^2}}\right) \ \
\textrm{as}\ x_2^2+x_3^2\to \infty.
\]
\end{rem}

Similarly to Theorem \ref{thmfarmine},  we can show
\begin{pro}\label{progibbonsPeriodic}
Assume that $W\in C^2$ satisfies condition  \textbf{(a')} and is
even. Let $-\mu\leq u \leq \mu$ be a solution to (\ref{eqentire})
such that
\[
\sup_{\mathbb{R}^{n}_-} u=\mu,\ \ \textrm{where}\
\mathbb{R}^n_-=\mathbb{R}^n\cap\{x_1<0\},
\]
and is periodic in the remaining variables $(x_2,\cdots,x_n)$,
namely $u(x_1,x_2,\cdots,x_n)\equiv u(x_1,x_2+T_2,\cdots,x_n+T_n)$
for some $T_i\in \mathbb{R}$. Then, $u$ is one-dimensional in $x_1$
and non-increasing.
\end{pro}
\begin{proof}
Firstly, by the strong maximum principle, unless $u\equiv\mu$, the
periodicity of $u$ implies that
\[
\sup_{|x_1|\leq L}u =\max_{|x_1|\leq L}u =c_L<\mu\ \ \textrm{for\
every}\ L>0.
\]
It follows that there exists a sequence of points $A_j=(a_j,a'_j)\in
\mathbb{R}^n_-$ such that $a_j\to -\infty$, and $u(A_j)\to \mu$. As
we discussed in Remark \ref{rempointwise}, via Harnack's inequality,
we have that
\[
u\to \mu,\ \textrm{uniformly\ on\ compact\ subsets\ of}\
\mathbb{R}^{n-1},\ \textrm{as}\ x_1\to -\infty.
\]
By the periodicity of $u$ in the remaining variables
$(x_2,\cdots,x_n)$, we find that
\begin{equation}\label{eqgibons-}
u\to \mu,\ \textrm{uniformly\ in}\ \mathbb{R}^{n-1},\ \textrm{as}\
x_1\to -\infty.
\end{equation}

Given $R>0$, let $u_R$ be a minimizer as provided by Lemma
\ref{lem1}, extended by zero outside of $B_R$. By virtue of
(\ref{eqgibons-}), we can center the ball $B_R$ at a point $Q\in
\mathbb{R}^{n}_-$ so that \[u(x)>u_R(0)\geq u_R(x-Q)\ \textrm{in}\
B_R(Q),\] say $Q=(-Q_1,0,\cdots,0)$ with $Q_1>R$ sufficiently large.

If $u>0$ in $\mathbb{R}^n$, we can slide the ball $B_R(Q)$ around in
$\mathbb{R}^n$ to get that
\[
u(x)\geq u_R(x-Q)\ \ \forall\ x,Q\in \mathbb{R}^n.
\]
Taking $Q=x$, yields that $u(x)\geq u_R(0)$, $x\in \mathbb{R}^n$.
Since $R>0$ was arbitrary, in view of (\ref{eqestimRect}), and
recalling that $u\leq \mu$, we conclude that $u\equiv\mu$.

Otherwise, $u$ has to vanish somewhere. By virtue of
(\ref{eqgibons-}), and the periodicity of $u$ in the variables
$(x_2,\cdots,x_n)$, we may assume that (\ref{eqfarP}) holds for some
$P$ on the hyperplane $\{x_1=0\}$. Sliding the ball $B_R(Q)$ in
$x_1<0$, until it is tangent at $P$, to obtain a similar relation to
(\ref{eqflat2}), and recalling (\ref{eqchen}), we can conclude as in
Theorem \ref{thmfarmine} that $u$ depends only on $x_1$ and is
monotone.

The proof of the proposition is complete.
\end{proof}

Moreover, we can show:
\begin{pro}
Assume that $W$ is as in Remark \ref{remduma} and $W'(t)>0$, $t\in
(\mu_-,0)$. Then, there does not exist a solution $u\in
C^2(\mathbb{R}^n)$ to (\ref{eqentire}) such that $\mu_-\leq u\leq
\mu$ and the level set $\{x\in \mathbb{R}^n: u(x)=0 \}$ is bounded.
\end{pro}

In the case where the assumption $W\geq 0$ is violated, say when
$W(\mu_-)<0=W(\mu)$, then $u$ has to be radially symmetric with
respect to some point $x_0\in \mathbb{R}^n$ and increasing (see
Theorem 3.3 in \cite{farinaRendi}).

\subsection{One-dimensional symmetry in half-spaces}
Consider the problem
\begin{equation}\label{eqhalfAng}
\left\{ \begin{array}{ll}
          \Delta u=W'(u) & \textrm{in}\ \mathbb{R}^n_-=\{x_1<0\}, \\
           &  \\
          u=0 & \textrm{on}\ \{x_1=0\}, \\
           &  \\
          u>0 & \textrm{in}\ \mathbb{R}^n_-.
        \end{array}
 \right.
\end{equation}
As we have already seen in Remark \ref{remhalfplane}, this type of
problems arise after blowing-up, close to the boundary, singular
perturbation problems of the form (\ref{eqsingular}) (see also
Proposition \ref{promixed}).

The following result was proven by Angenent  in \cite{angenent} by
the method of moving planes:

\begin{pro}\label{proAng}
Assume that $W\in C^2$ satisfies $W'(0)=0$, $W''(0)<0$,
(\ref{eqW''}), $W'(\mu)=0$, $W''(\mu)>0$, and $W'(t)>0$, $t>\mu$.
Then, any bounded solution to (\ref{eqhalfAng}) depends only on the
$x_1$ variable (such solution exists and is strictly decreasing in
$x_1$, recall (\ref{eqU})).
\end{pro}

In \cite{cafareliPacard}, the authors relaxed the condition
$W''(\mu)>0$ to $W'$ being non-decreasing near $\mu$ and allowed for
$W'$ merely Lipschitz (and also included the case $W'(0)<0$). In
fact, the condition $W'$ being non-decreasing near $\mu$ is not
needed, as shown in \cite{farinaRigidity} (see also \cite{duNomax}
for a different approach). The behavior of $W$ near $t=0$ has been
relaxed in \cite{farinaRigidity}. As a matter of fact, there is no
need to assume something for the behavior of $W$ near $t=0$,
provided that $n\leq 5$ (see \cite{farinaFlat}, and the references
therein, where $W'\in C^1$ is also required for $n=4,5$). One of the
main results that was used in the aforementioned references is
Theorem 1 of \cite{berestProcHalf} (see also Theorem 1.4 in
\cite{berestyckiCaffarelli}), which says that if $u$ is a bounded
solution to (\ref{eqhalfAng}), where $W'$ is Lipschitz continuous,
with $M=\sup_{\mathbb{R}_-^n}u$, then $u$ is one-dimensional and
monotone provided that $W'(M)\geq 0$ (furthermore, $W'(M)=0$); see
also Proposition \ref{prohalfclemMine}  as well as Remarks
\ref{remhalfCSLips} and \ref{remJang} herein).

Based on Theorem \ref{thmfarmine}, we can provide a \emph{completely
different} proof of Proposition \ref{proAng}, while also removing
the condition $W''(\mu)>0$. The drawback of our approach is that we
impose a higher degree of regularity on $W$, in order to apply
Proposition \ref{propacard} which is based on bifurcation arguments.

\begin{pro}\label{prohalfMine}
Assume that $W\in C^3$ satisfies the hypotheses in Proposition
\ref{proAng}, except from $W''(\mu)>0$. Then, the same assertion of
the latter proposition holds.
\end{pro}
\begin{proof}
Firstly, note that, as in Proposition \ref{produ}, it follows that
$0<u<\mu$ in $\mathbb{R}^n_-$ (see also Lemma 2.4 in
\cite{farinaFlat}). Then, arguing as in Theorem \ref{thmfarmine}, we
can show that
\begin{equation}\label{eqmodicaline}
u_{x_1}^2\geq 2 W(0)\ \ \textrm{on}\ \{x_1=0\}.
\end{equation}
Now, let
\[
\tilde{W}(t)=\left\{\begin{array}{cc}
                      W(t), & t\geq 0, \\
                       &  \\
                      W(-t), & t<0,
                    \end{array}
 \right.
 \ \ \textrm{and}\ \ \tilde{u}(x)=\left\{\begin{array}{cc}
                      u(x_1,\cdots,x_n), & x_1\leq 0, \\
                       &  \\
                      -u(-x_1,\cdots,x_n), & x_1>0.
                    \end{array}
                    \right.
\]
Since $W'(0)=0$, it follows that $\tilde{W}\in C^2$. Clearly
$\tilde{u}\in C^1(\mathbb{R}^n)\cap C^2\left(\mathbb{R}^n\backslash
\{x_1=0\}\right)$, and satisfies
\[
\Delta \tilde{u}=\tilde{W}'(\tilde{u})\ \ \textrm{in}\
\mathbb{R}^n\backslash \{x_1=0\}.
\]
In particular, since the righthand side belongs in
$C^\alpha(\mathbb{R}^n)$, $\alpha\in (0,1)$, standard interior
Schauder  estimates (see \cite{Gilbarg-Trudinger}) tell us that
$\tilde{u}\in C^{2+\alpha}(\mathbb{R}^n)$.  Hence, we infer that
$\tilde{u}$ is a classical bounded solution to the above equation.
Moreover, by its construction $\tilde{u}$ is odd with respect to
$x_1$. It follows, via (\ref{eqmodicaline}), that
\[
|\nabla\tilde{u}|^2\geq 2 W(\tilde{u})\ \ \textrm{on}\ \{x_1=0\}.
\]
Since $\tilde{W}(t)\geq 0$, $t\in \mathbb{R}$, $\tilde{W}\in C^2$,
and vanishes at $t=\pm \mu$, Theorem 5.1 in \cite{cafamodica} yields
that $\tilde{u}$ is one-dimensional. The assertion of the
proposition follows immediately.

The proof of the proposition is complete.
\end{proof}

Similarly, arguing as in the proof of Proposition
\ref{progibbonsPeriodic}, we can show the following proposition.

\begin{pro}\label{prohalfclemMine}
Suppose that $W\in C^2$ satisfies \textbf{(a')} and (\ref{eqW''}).
If $u$ is a solution to (\ref{eqhalfAng}) such that
\begin{equation}\label{eqsweershalfSup}\sup_{\mathbb{R}_-^n}u= \mu,
\end{equation} then $u$ is one-dimensional.
\end{pro}

\begin{rem}\label{remhalfCSLips}
Noting that most assertions of Lemma \ref{lem1} continue to hold for
$W\in C^{1,1}$, it is not hard to see that Proposition
\ref{prohalfclemMine} holds for $W$ with this regularity. In this
regard, note that the Lipschitz continuity of $W'$ allows for the
strong maximum principle to be applied in the linear equation for
$u-u_R$. A point to be stressed is that the  gradient bound
(\ref{eqmodicafarina}), proven in \cite{cafamodica} under the
assumption $W\in C^2$, was recently generalized, allowing for $W$ to
be $C^{1,1}$, in \cite{farinaAdvMath}. Moreover, in the case where
$W'$ is merely Lipschitz continuous, some care is needed when
reflecting $u$ (see  Corollary 1.3 in \cite{farinaLogTrick}).
\end{rem}

The above proposition was proven originally under the stronger
assumption (\ref{eqgibons-})
 by Cl\'{e}ment and Sweers in
\cite{clemente}, see Proposition 2.5 therein. They also assumed that
 $W\in C^{2+\alpha}$, $\alpha\in (0,1)$.
Their approach was based on comparison arguments with suitable
one-dimensional upper and lower solutions and shooting arguments.
Subsequently, it was extended to more general equations in
\cite{berestProcHalf}, by means of the sliding method (see also
Theorem 1.4 in \cite{berestyckiCaffarelli} and Theorem 4.7 in
\cite{ni}), assuming merely that (\ref{eqsweershalfSup}) holds and
that $W'$ is Lipschitz continuous.

So far, the arguments in this subsection have been based in
reflecting $u$ oddly across $\{x_1=0\}$, which is possible since
$W'(0)=0$. If $W'(0)<0$, this ceases to be an option. Nevertheless,
taking advantage of a recent result of \cite{farinaFlat} which
extends the gradient bound (\ref{eqmodicafarina}) to the case of
half-spaces, under the additional assumptions $W'(0)\leq 0$ and
$u\geq 0$, we can show:

\begin{pro}\label{prohalfFlat}
The assertions of Propositions \ref{prohalfMine},
\ref{prohalfclemMine} remain true if $W'(0)<0$ and
(\ref{eqpacardW'}) hold, together with $W'(t)>0$, $t>\mu$, for
Proposition \ref{prohalfMine}; relation (\ref{eqW''}) for
Proposition \ref{prohalfclemMine}.
\end{pro}
\begin{proof}
As in Proposition \ref{prohalfMine} (recalling that Proposition
\ref{propacard} works for such $W$), we find that
(\ref{eqmodicaline}) holds. On the other side, by Theorem 1.4 in
\cite{farinaFlat}, we have
\begin{equation}\label{eqflatlater}
|\nabla u|^2\leq 2W(u)\ \ \textrm{in} \ {\mathbb{R}_-^{n-1}}.
\end{equation}
In particular, from (\ref{eqmodicaline}) and the above relation with
$x_1=0$ (recalling that $u=0$ there), we obtain that
\begin{equation}\label{eqflat}
u_{x_1}=-\sqrt{2W(0)},\ u_{x_i}=0,\ i=2,\cdots,n,\ \textrm{on}\ \{
x_1=0\}.
\end{equation}
However, it has \emph{not} been shown in \cite{farinaFlat} that, if
equality in (\ref{eqflatlater}) is achieved at some point on
$\{x_1\leq 0\}$, the solution is one-dimensional (actually, this was
shown in the subsequent paper \cite{farinaAdvMath} in the case where
equality is achieved at an interior point which in addition is
noncritical for $u$). Rather than proceeding in this direction, we
will argue as follows. From the relation that corresponds to
(\ref{eqflat2}), recalling (\ref{eqUR}) and (\ref{eqclaim}), we get
\begin{equation}\label{eqslice}
u(x)\geq \textbf{U}(-x_1)\ \ \textrm{in}\  \mathbb{R}^{n-1}_-,
\end{equation}
where $\textbf{U}$ is as in (\ref{eqU}) (in the analog of
(\ref{eqfar2}), when the ball is tangent at $(0,x')$, we consider
any strip $[-L,0]\times \mathbb{R}^{n-1}$, look only at the
$x'$-slice, and let $R\to \infty$). On the other hand, since both
$u(x)$ and $\textbf{U}(-x_1)$ solve (\ref{eqhalfAng}), by the strong
maximum principle and Hopf's boundary point lemma (applied to the
linear equation for $u-\textbf{U}$), we deduce that either
$u_{x_1}<-\textbf{U}'(0)=-\sqrt{2W(0)}$ on $\{x_1=0\}$ or
$u(x)\equiv \textbf{U}(-x_1)$ in $\mathbb{R}^{n-1}_-$. In view of
(\ref{eqflat}), we conclude that the latter scenario holds.

The proof of the proposition is complete.
\end{proof}

As noted in \cite{berestyckiCaffarelli}, one-dimensional symmetry
results for (\ref{eqhalfAng}) can be thought of as extensions of the
Gidas, Ni and Nirenberg \cite{gidas} symmetry result for spheres,
when the radius of the sphere increases to infinity while a point on
the boundary is being kept fixed. This is essentially what we do in
Theorem \ref{thmfarmine}. This procedure that we apply can be
appropriately named ``method of expanding spheres''.

\begin{rem}\label{remJang}
After this subsection was written, we found the paper
\cite{jangNonli}, where it was shown that solutions of
(\ref{eqhalfAng}), having values in $(0,\mu)$, with $W'$ Lipschitz
and $W(t)>W(\mu)=0$ for $t\in [0,\mu)$, satisfy the gradient bound
(\ref{eqflatlater}) on the hyperplane $\{x_1=0\}$, which is enough
for our purposes. Actually, the approach of \cite{jangNonli}, for
establishing the one-dimensional symmetry of solutions to
(\ref{eqhalfAng}), is similar in spirit to ours.

Armed with this information, it is not hard to see that the
assertion of Proposition \ref{prohalfclemMine} continues to hold
solely under the aforementioned assumptions (if $W'(0)>0$, we have
to use $u_R$ as in Lemma \ref{lem1Sign}). Hence, we can essentially
recover the general result of \cite{berestProcHalf}.
\end{rem}

\subsection{A rigidity result}
The following rigidity result was proven in Theorem 4.2 of
\cite{farinaAdvMath}, assuming additionally that $W'(t)\leq 0$,
$W'(t)+W'(-t)\leq 0$, $t\in (0,\mu)$, and that $W'$ is
non-decreasing near $\mu$. The second condition is clearly satisfied
if $W'$ is odd. In this regard, it might be useful to recall our
discussion preceding Theorem \ref{thmfarmine}, related to
\cite{duMa} where the authors also assumed additionally that $W'$ is
odd. This is not a coincidence, since both \cite{duMa} and
\cite{farinaAdvMath} employ modifications of the method of moving
planes. In contrast, making use of the arguments from Theorem
\ref{thmfarmine}, we will see that the aforementioned assumptions on
$W$ are not needed.
\begin{thm}\label{thmrigidity}
Suppose that $W\in C^{1,1}(\mathbb{R})$, $W(t)\geq 0$ for
$t\in\mathbb{R}$, $W(\mu)=0$, and $W(t)>0$ for $t\in (\mu_-,\mu)$.
If $u\in C^2(\mathbb{R}^n)$ satisfies (\ref{eqentire}), $\mu_-\leq u
\leq \mu$ in $\mathbb{R}^n$, and
\begin{equation}\label{eqrigunif}
u\to \mu, \ \textrm{uniformly\ in}\ \mathbb{R}^{n-1},\ \textrm{as}\
x_1\to \pm \infty,
\end{equation}
then
\[
u\equiv\mu.
\]
\end{thm}
\begin{proof}
Since
\[
W(t)=\int_{\mu}^{t}W'(s)ds,
\]
and $W(t)>0$, $t\in (\mu_-,\mu)$, there exists a sequence
\[
\epsilon_j\to 0^+ \ \textrm{such\ that}\ W'(\mu-\epsilon_j)\leq 0.
\]
For every $j\gg 1$,
 we intend to
show that
\[
u\geq \mu-\epsilon_j \ \ \textrm{in}\ \mathbb{R}^n,
\]
which clearly implies the assertion of the proposition. To this end,
we argue by contradiction, namely assume that
\begin{equation}\label{eqrig3}
u(x_0)<\mu-\epsilon_j\ \textrm{for\ some}\ x_0\in \mathbb{R}^n.
\end{equation}
In the sequel, for notational convenience, we will drop the
subscript $j$ from $\epsilon$. By virtue of (\ref{eqrigunif}), we
can define
\[
\ell_-=\sup \{s\in \mathbb{R}: \ u(x_1,x')>\mu-\epsilon\
\textrm{if}\ x_1<s\ \textrm{and}\ x'\in \mathbb{R}^{n-1}\},
\]
and
\[
\ell_+=\inf \{s\in \mathbb{R}: \ u(x_1,x')>\mu-\epsilon\
\textrm{if}\ x_1>s\ \textrm{and}\ x'\in \mathbb{R}^{n-1}\}.
\]
In view of  (\ref{eqrig3}), we get that $\ell_\pm \in \mathbb{R}$
and $\ell_-<\ell_+$.

Let $u_R$ be an energy minimizing solution to the following problem:
\[
\Delta u_R=W'(u_R),\ \mu-\epsilon<u_R<\mu\ \textrm{in}\ B_R;\
u_R=\mu-\epsilon\ \textrm{on}\ \partial B_R,
\]
as provided by Lemma \ref{lem1}. In fact, it is easy to see that
most assertions of Lemma \ref{lem1} continue to hold if $W\in
C^{1,1}$ instead of $C^2$, and $u_R=m$ on $\partial B_R$ with $m\in
(0,\mu)$ such that $W'(m)\leq 0$ (with the obvious modifications).
By the uniform asymptotic behavior of $u$ as $x_1\to -\infty$, we
deduce that, given $R>0$, there exists $Q_1>R-\ell_-$ such that
\[
u(x)>u_R(0)\geq u_R(x-Q),\ x\in B_R(Q),\ \textrm{where}\
Q=(-Q_1,0,\cdots,0).
\]
Since $u>\mu-\epsilon$ if $x_1<\ell_-$, and $W'(\mu-\epsilon)\leq
0$, we can slide the ball $B_R(Q)$ around in $\{x_1<\ell_-\}\times
\mathbb{R}^{n-1}$ (as usual), to arrive at
\[
u(x_1,x')>u_R(x_1+R-\ell_-),\ \ x_1\in (\ell_--R, \ell_-),\ x'\in
\mathbb{R}^{n-1},
\]
using the notation $u_R(|x|)=u_R(x)$, $x\in B_R$ (with the obvious
meaning). Making use of the obvious analog of (\ref{eqclaim})
(loosely speaking, letting $R\to \infty$ in the above relation), we
obtain that
\[
u(x_1,x')\geq U(\ell_--x_1),\ \ x_1\leq \ell_-,\ x'\in
\mathbb{R}^{n-1},
\]
where here $U\in C^2[0,\infty)$ denotes the unique classical
solution to
\[
U''=W'(U),\ s>0;\ U(0)=\mu-\epsilon,\ \lim_{s\to \infty}U(s)=\mu,
\]
we note that $U'>0$. Similarly, we have
\[
u(x_1,x')\geq U(x_1-\ell_+),\ \ x_1\geq \ell_+,\ x'\in
\mathbb{R}^{n-1}.
\]

On the other side, from the definition of $\ell_-$, there exist
sequences $(x_1)_j\geq\ell_-$ and $(x')_j\in \mathbb{R}^{n-1}$ such
that
\begin{equation}\label{eqgibbonsxj}
(x_1)_j\to \ell_-\ \textrm{and}\ u\left((x_1)_j, (x')_j \right)\leq
\mu-\epsilon.
\end{equation}
Let
\[
v_j(x_1,x')=u\left(x_1,x'+(x')_j \right), \ \ x_1\in \mathbb{R},\
x'\in \mathbb{R}^{n-1}.
\]
Each $v_j$ satisfies (\ref{eqentire}), $\mu_-\leq v_j\leq \mu$,
\[
v_j\left((x_1)_j,0 \right)\leq \mu-\epsilon,\ \textrm{where}\
(x_1)_j\to \ell_-,
\]
\[
v_j(x_1,x')\geq U(\ell_--x_1)\ \textrm{if}\ x_1\leq \ell_-,\ x'\in
\mathbb{R}^{n-1};\ \ v_j(x_1,x')\geq U(x_1-\ell_+)\ \textrm{if}\
x_1\geq \ell_+,\ x'\in \mathbb{R}^{n-1}.
\]
Making use of standard elliptic regularity  estimates
\cite{Gilbarg-Trudinger}, and the usual diagonal-compactness
argument, passing to a subsequence, we find that $v_j\to v_\infty$
in $C^2_{loc}(\mathbb{R}^n)$, where $v_\infty$ satisfies
(\ref{eqentire}), $\mu_-\leq v_\infty \leq \mu$,
$v_\infty(\ell_-,0)\leq \mu-\epsilon$, and
\begin{equation}\label{eqgibbonsgeq}
v_\infty(x_1,x')\geq U(\ell_--x_1)\ \textrm{if}\ x_1\leq \ell_-,\
x'\in \mathbb{R}^{n-1};\ \ v_\infty(x_1,x')\geq U(x_1-\ell_+)\
\textrm{if}\ x_1\geq \ell_+,\ x'\in \mathbb{R}^{n-1}.
\end{equation}
It follows that $v_\infty(\ell_-,0)= \mu-\epsilon$ and
\[
\partial_{x_1}v_\infty(\ell_-,0)\leq -U'(0)=-\sqrt{2W\left(\mu-\epsilon
\right)}=-\sqrt{2W\left(v_\infty(\ell_-,0) \right)}.
\]
Similarly to the proof of Proposition \ref{prohalfFlat}, by the
strong maximum principle, and the unique continuation principle
\cite{hormander}, we infer that $v_\infty\equiv U(\ell_--x_1)$ (if
$W\in C^2$, we can apply Theorem 5.1 in \cite{cafamodica}, as we did
in Theorem \ref{thmfarmine}). However, this contradicts the second
relation in (\ref{eqgibbonsgeq}).

The proof of the proposition is complete.
\end{proof}
\begin{rem}
In light of the multiple-end solutions to the equation $\Delta
u+u-u^3=0$ in the plane, constructed recently in \cite{delpinoJFA10}
(see also \cite{guiEven,kowalczykliupacard}), we infer that the
uniform assumption in (\ref{eqrigunif}) is \emph{necessary} in
Theorem \ref{thmrigidity}.
\end{rem}

\subsection{A new proof of Gibbons' conjecture in two dimensions}
In this subsection, exploiting further the approach that we have
developed in the previous ones, as well as the Hamiltonian structure
of the equation, we will give a totally  new proof of the well known
Gibbons' conjecture,  mentioned in the beginning of this section, in
two dimensions.

\begin{thm}
Suppose that $W\in C^2(\mathbb{R})$, $W(t)\geq 0$ for
$t\in\mathbb{R}$,  $W(t)>0$ for $t\in (\mu_-,\mu_+)$,
$W(\mu_\pm)=0$, and $W''(\mu_\pm)>0$. If $u\in C^2(\mathbb{R}^2)$
satisfies (\ref{eqentire}), $\mu_-\leq u \leq \mu_+$ in
$\mathbb{R}^2$, and
\begin{equation}\label{eqrigunif}
u\to \mu_\pm, \ \textrm{uniformly\ in}\ \mathbb{R},\ \textrm{as}\
x_1\to \pm \infty,
\end{equation}
then $u(x)=U(x_1)$, where $U$ satisfies
\begin{equation}\label{eqgibbonsU}U''=W'(U), \ x_1\in \mathbb{R},\ \
U(x_1)\to  \mu_\pm \ \textrm{as}\ x_1\to \pm \infty.
\end{equation}
\begin{proof}
By our assumptions on $W$, there exists a $\mu_0\in (\mu_-,\mu_+)$
such that $W'(\mu_0)=0$ and $W''(\mu_0)\leq 0$.

By our assumptions on $u$, there exist real numbers
$\ell_-\leq\ell_+$ such that
\[
\ell_-=\sup \{s\in \mathbb{R}: \ u(x_1,x')<\mu_0\ \textrm{if}\
x_1<s\ \textrm{and}\ x_2\in \mathbb{R}\},
\]
and
\[
\ell_+=\inf \{s\in \mathbb{R}: \ u(x_1,x_2)>\mu_0\ \textrm{if}\
x_1>s\ \textrm{and}\ x_2\in \mathbb{R}\}.
\]
By the strong maximum principle, we deduce that $u<\mu_0$ if
$x_1<\ell_-$ and $u>\mu_0$ if $x_1>\ell_+$. By virtue of Theorem
\ref{prohalfclemMine}, in order to conclude, it suffices to show
that $\ell_-=\ell_+$. We observe that, thanks to Theorem
\ref{thmfarmine}, the assertion of the current theorem follows if
there exists $x_2 \in \mathbb{R}$ such that $u(\ell_-,x_2)=\mu_0$ or
$u(\ell_+,x_2)=\mu_0$. So, let us assume that there exist sequences
$\pm(x_\pm)_j>\pm\ell_\pm$ and $(x_{2\pm})_j\in \mathbb{R}$ such
that
\[
(x_\pm)_j\to \ell_\pm,\ \ u\left((x_\pm)_j, (x_{2\pm})_j \right)\to
\mu_0, \ \ \textrm{and}\ \ |(x_{2\pm})_j|\to \infty\ \ \textrm{as}\
j\to \infty.
\]

Similarly to Theorem \ref{thmrigidity}, passing to a subsequence, we
find that
\begin{equation}\label{eqgibbonslimes}
u\left(x_1,(x_{2\pm} )_j\right)\to U(x_1-\ell_\pm)\ \ \textrm{in}\
C_{loc}^2(\mathbb{R}),\ \ \textrm{as}\ j\to \infty,
\end{equation}
for some $U$  that satisfies (\ref{eqgibbonsU}) and $U(0)=0$.

Since $W''(\mu_\pm)>0$, it is standard to show that there exist
constants $c,C>0$ such that
\begin{equation}\label{eqgibbonsExp}
|u(x)-\mu_\pm|+\left|\partial_{x_1}^k \partial_{x_2}^\alpha u
\right|\leq Ce^{-c|x_1|},\ \ \forall\ x=(x_1,x_2)\in \mathbb{R}^2,
\end{equation}
for $k=0,1,2,3$ and multi-indexes $\alpha$ with $|\alpha|\leq 3$
such  that $k+|\alpha|>0$. So, similarly to
\cite{buscaFelmer,guiMalchiodi} (see also \cite{weiAxisymmetry}), we
can justify integration by parts and establish the following
identity of ``Hamiltonian'' type:
\[
\partial_{x_2}g(u;x_2)=-h(u;x_2)\ \ \textrm{for\ every}\ x_2\in
\mathbb{R},
\]
 where
\[
g(u;x_2)=\int_{\mathbb{R}}^{}x_1\left\{\frac{1}{2}u_{x_1}^2-\frac{1}{2}u_{x_2}^2+W(u)
\right\} dx_1 \ \ \textrm{and}\ \
h(u;x_2)=\int_{\mathbb{R}}^{}u_{x_1}u_{x_2}dx_1.
\]
Indeed, with the obvious notation, we have
\[
\partial_{x_2}g=\int_{\mathbb{R}}^{}x_1\left\{u_{x_1}u_{x_1x_2}-u_{x_2x_2}u_{x_2}+W'(u)u_{x_2}
\right\}dx_1.
\]
We find that
\[
\begin{array}{rcl}
  \int_{\mathbb{R}}^{}x_1u_{x_1}u_{x_1x_2}dx_1 & = & \int_{\mathbb{R}}^{}\left\{\partial_{x_1} \left(x_1
u_{x_1}u_{x_2}
\right)-x_1u_{x_1x_1}u_{x_2}-u_{x_1}u_{x_2}\right\}dx_1 \\
    &   &   \\
    & =  &
    \int_{\mathbb{R}}^{}\left\{-x_1u_{x_1x_1}u_{x_2}-u_{x_1}u_{x_2}\right\}dx_1.
\end{array}
\]
It follows that
\[\begin{array}{rcl}
    \partial_{x_2}g & = & \int_{\mathbb{R}}^{}x_1\left\{-u_{x_1x_1}-u_{x_2x_2}+W'(u)
\right\}u_{x_2}dx_1-\int_{\mathbb{R}}^{}u_{x_1}u_{x_2}dx_1
 \\
     &   &   \\
     & = & -\int_{\mathbb{R}}^{}u_{x_1}u_{x_2}dx_1,
  \end{array}
\]
as claimed. Next, motivated from \cite{buscaFelmer},
\cite{guiMalchiodi}, we will show that $g$ is a constant. To this
end, we claim that
\[
\partial_{x_2}h(u;x_2)=0\ \ \forall \ x_2\in \mathbb{R}.
\]
Indeed, we have that
\[
\partial_{x_2}h=\int_{\mathbb{R}}^{}\left(u_{x_2x_1}u_{x_2}+u_{x_1}u_{x_2x_2}
\right)dx_1=\int_{\mathbb{R}}^{}u_{x_1}u_{x_2x_2}dx_1.
\]
So, we obtain that
\[
\partial_{x_2}h=\int_{\mathbb{R}}^{}u_{x_1}u_{x_2x_2}dx_1=
\int_{\mathbb{R}}^{}u_{x_1}\left\{W'(u)-u_{x_1x_1} \right\}dx_1=0,
\]
as desired. It follows that
\[
\partial^2_{x_2} g=0\ \ \textrm{in}\  \ \mathbb{R}.
\]
Since $g$ is bounded (from (\ref{eqgibbonsExp})),
we infer that $g$
is a constant. In particular, we have that
\[
g\left(u;(x_{2-})_j\right)=g\left(u;(x_{2+})_j\right),\ \ j\geq 1.
\]
Using (\ref{eqgibbonslimes}), (\ref{eqgibbonsExp}), and letting
$j\to \infty$, with the help of Lebesgue's dominated convergence
theorem, we arrive at
\[
\int_{\mathbb{R}}^{}x_1\left\{\frac{1}{2}\left[U'(x_1-\ell_-)\right]^2+W\left(U(x_1-\ell_-)\right)
\right\} dx_1 =
\int_{\mathbb{R}}^{}x_1\left\{\frac{1}{2}\left[U'(x_1-\ell_+)\right]^2+W\left(U(x_1-\ell_+)\right)
\right\} dx_1.
\]
We conclude that $\ell_-=\ell_+$ as desired.

The proof of the theorem is complete.
\end{proof}
\end{thm}

\begin{rem}
Related Hamiltonian identities can be found in \cite{gui} (see also
\cite{alikakosDamialis,alikakosThreeOrmore} for a slightly different
viewpoint which employs a stress energy tensor, and
\cite{delpinoModuli}).
\end{rem}

\begin{rem}
The original proofs in \cite{barlow,berestDuke,farinaRendi}, which
are valid in all dimensions, work with the weaker conditions that
$W'$ is increasing near $\mu_\pm$ and $W\in C^{1,1}$.
\end{rem}

\begin{rem}For a proof of the conjecture, when
the nonlinearity $W'$ is discontinuous, we refer to
\cite{farinaDiscontinuous}.
\end{rem}

\begin{rem}
A parabolic version of Gibbons' conjecture can be found in
\cite{bereProcGibbons}.
\end{rem}

\begin{rem}
Proofs of analogs of the Gibbons' conjecture, involving more general
operators than the usual Laplacian (such as fully nonlinear elliptic
differential operators or fractional Laplacians), can be found in
\cite{birindeliGibbons} and \cite{farinaIndianaGibbons}.
\end{rem}

\begin{rem}\label{remFarSoave}
A Gibbons' type of conjecture for entire solutions with algebraic
growth of a  semilinear elliptic system, arising in Bose-Einstein
condensation, was proven recently  in \cite{farinaSoave}.
\end{rem}

\section{One-dimensional symmetry in convex cylindrical
domains}\label{seccylindric} In \cite{carbou,farCVPDE}, the authors
considered energy minimizing solutions to
\begin{equation}\label{eqcarbouEq}
    \Delta u +u-u^3=0\  \textrm{in}\ \mathbb{R}\times \omega;\
\frac{\partial u}{\partial \nu}=0\ \textrm{on}\ \mathbb{R}\times
\partial\omega,
\end{equation}
such that
   \[ u \to  \pm 1,\ \textrm{uniformly\ for}\ x'\in \bar{\omega},\ \textrm{as}\
x_1\to \mp \infty,
  \]where $\omega$ is a smooth  bounded  domain of $\mathbb{R}^{n-1}$ and $\nu$
denotes $(\mathbb{R}\times\partial\omega)$'s outer unit normal
vector (in fact, they studied minimizers of the energy with $\omega$
merely bounded, without looking at the Euler-Lagrange equation).
Using a rearrangement argument, they showed that $u$ is
one-dimensional (see also \cite{brockElem} and \cite{kawohl}).
Related results can be found in \cite{berestCylind}.

Surprisingly enough, if $n=2$, our Proposition
\ref{progibbonsPeriodic} implies that the limit in \emph{just one
direction} is needed  to reach the same conclusion \emph{without}
even assuming that $u$ is an energy minimizing solution. In this
section, following the strategy of the previous section, we will
show that the same property holds true in any dimension, provided
that $\omega$ is convex. We emphasize that our approach applies to
equations with more general nonlinearities and does not make use of
the oddness of the nonlinearity in hand. It suffices that $W\in
C^2$, $W>0$ in $t\in (\mu_-,\mu)$ and $W(\mu_-)=W(\mu)=0$.

\subsection{A gradient bound in convex cylindrical domains}
In order to apply the strategy of Section \ref{secfarina}, we will
first prove that the gradient bound (\ref{eqmodicafarina}) continues
to hold  in this setting. For the corresponding problem with
Dirichlet boundary conditions, this was shown recently in
\cite{farinaAdvMath}. As in the latter reference, we will follow the
lines that were set  in \cite{cafamodica} for the whole space
problem,  with the necessary modifications in order to deal with the
presence of the boundary. To this end, the authors of
\cite{farinaAdvMath} introduced (among other things) the idea to
translate the domain. Our proof is essentially the same, however we
keep the domain fixed and, in contrast to the Dirichlet boundary
condition case, we have to appeal to a result in
\cite{sternbergZumb} (originally due to \cite{casten,matano}).

\begin{pro}\label{procarbouModica}
Let $\Omega=\Omega_0\times \mathbb{R}^{n-n_0}$, where
$\Omega_0\subset \mathbb{R}^{n_0}$ is a bounded, smooth ($\partial
\Omega$ at least $C^2$) and convex domain, and  $1\leq n_0< n$. Let
$u\in C^2(\bar{\Omega})\cap L^\infty(\Omega)$ be a solution to
\begin{equation}\label{eqcarbuEq}
\Delta u-W'(u)=0,\ x\in \Omega;\ \ \frac{\partial u}{\partial
\nu}=0,\ x\in \partial \Omega,
\end{equation}
where $\nu$ denotes the outer unit normal vector to $\partial
\Omega$, and $W\in C^2(\mathbb{R})$. If $W(t)\geq 0$, $t\in
\mathbb{R}$, then
\[
|\nabla u|^2-2W(u)\leq 0,\ \ x\in \Omega.
\]
\end{pro}
\begin{proof}
Let $u$ be as in the assertion of the proposition. We set
\[
\mathcal{F}=\left\{v\in C^2(\bar{\Omega})\ \textrm{solutions of}\
(\ref{eqcarbuEq})\ \textrm{with}\ |v|\leq \|u\|_{L^\infty(\Omega)}\
\textrm{on}\ \bar{\Omega} \right\}.
\]
Clearly $u\in \mathcal{F}$. Next, let
\[
P(v,x)=\left|\nabla v(x) \right|^2-2W\left(v(x) \right),\ \ v\in
\mathcal{F},\ x\in \bar{\Omega}.
\]
These type of $P$-functions have been extensively investigated in
the PDE literature (see Chapter 5 in \cite{sperb}).

By formula (2.7) in \cite{cafamodica}, for $v\in \mathcal{F}$ we
have
\begin{equation}\label{eqcarboucafa}
\left|\nabla v(x) \right|^2\Delta P(v,x)-2W'\left(v(x) \right)\nabla
v(x)\cdot\nabla P(v,x)\geq \frac{\left| \nabla P(v,x)\right|^2}{2}\
\ \textrm{if}\ x\in \Omega\ \textrm{and}\ \nabla v(x)\neq 0.
\end{equation}
Moreover, we find
\[
\frac{\partial}{\partial \nu}P(v,x)=\frac{\partial}{\partial
\nu}\left(|\nabla v|^2 \right)-2W'(v)\frac{\partial v}{\partial
\nu}=\frac{\partial}{\partial \nu}\left(|\nabla v|^2 \right)\ \
\textrm{on}\ \partial \Omega.
\]
Since $\Omega$ is smooth and convex, and $v\in C^2(\bar{\Omega})$
satisfies $\frac{\partial v}{\partial \nu}=0$ on $\partial \Omega$,
it follows from Lemma 2.2 in \cite{sternbergZumb} (see also Lemma
2.1 in \cite{alikakosConvex},  \cite{casten}, Theorem 1.1 in
\cite{LiYau}, Lemma 5.3 in \cite{matano}, and  page 79 in
\cite{sperb}) that
\[
\frac{\partial}{\partial \nu}\left(|\nabla v|^2 \right)\leq 0\ \
\textrm{on}\ \partial \Omega.
\]
In turn, this implies that
\begin{equation}\label{eqcarboumatano}
\frac{\partial}{\partial \nu}P(v,x)\leq 0 \ \ \textrm{on}\ \partial
\Omega\ \textrm{for\ every}\ v\in \mathcal{F}.
\end{equation}
Now, we consider
\[
P_0\equiv \sup_{\stackrel{v\in \mathcal{F}}{x\in \bar{\Omega}}}
P(v,x).
\]
By elliptic regularity theory (see page 24 in
\cite{nirenbergTopics}), for $\alpha\in (0,1)$, there exists a
constant $C>0$ such that
\begin{equation}\label{eqcarbouC3bound}
\|v\|_{C^{2,\alpha}(\bar{\Omega})}\leq C\ \ \textrm{for\ all}\ v\in
\mathcal{F}.
\end{equation}
Hence, it follows that $P_0$ is finite, i.e. $P_0\in \mathbb{R}$.
The proposition will be proved if we show that
\[
P_0\leq 0.
\]
To this end, we will argue by contradiction, namely we assume that
\[
P_0> 0.
\]
We then take $v_k\in \mathcal{F}$ and $x_k\in \bar{\Omega}$ such
that
\begin{equation}\label{eqcarbou2side}
P_0-\frac{1}{k}\leq P(v_k,x_k)\leq P_0,\ \ k\geq 1.
\end{equation}
We write
\[
x_k=(y_k,z_k)\in \bar{\Omega},\ \ \textrm{where}\ y_k\in
\bar{\Omega}_0,\ z_k\in \mathbb{R}^{n-n_0},
\]
and set
\[
u_k(x)=v_k\left(x+(0,z_k) \right),\ \ x\in \bar{\Omega}.
\]
Making use of (\ref{eqcarbouC3bound}), passing to a subsequence, we
may assume that
\[
u_k\to u_\infty\ \ \textrm{in}\ C^2_{loc}(\bar{\Omega}),
\]
for some $u_\infty\in C^2(\bar{\Omega})$, satisfying
(\ref{eqcarbuEq}), with $|u_\infty|\leq \|u\|_{L^\infty(\Omega)}$ on
$\bar{\Omega}$.
In particular, we have that \[u_\infty \in \mathcal{F}.\] We may
further assume that
\[
y_k\to y_\infty \in \bar{\Omega}_0.
\]
From (\ref{eqcarbou2side}), we obtain that
\[
P(u_\infty,x_\infty)=P_0,\ \ \textrm{where}\
x_\infty=(y_\infty,0)\in \bar{\Omega}.
\]
Consider the set
\[
\mathcal{U}=\left\{x\in \bar{\Omega}\ \textrm{such\ that}\
P(u_\infty,x)=P_0 \right\}.
\]
We already know that $\mathcal{U}$ is nonempty (because $x_\infty
\in \mathcal{U}$). Moreover, since $u_\infty\in C^2(\bar{\Omega})$,
it follows that
\begin{equation}\label{eqcarbouclosed}
\mathcal{U}\ \textrm{is\ relatively\ closed\ in}\ \bar{\Omega}.
\end{equation}

We plan to prove that
\begin{equation}\label{eqcarbouopen}
\mathcal{U}\ \textrm{is\ relatively\ open\ in}\ \bar{\Omega}.
\end{equation}
Let $x_0\in \mathcal{U}$. Firstly, since $W\geq 0$, observe that
\[
\left|\nabla u_\infty(x_0) \right|^2=P_0+2W\left(u_\infty(x_0)
\right)\geq P_0>0.
\]
So, there exists an $r>0$ such that
\[
\left|\nabla u_\infty(x) \right|>0,\ \ x\in B_r(x_0)\cap
\bar{\Omega}.
\]
It then follows from (\ref{eqcarboucafa}) that
\[
\Delta P(u_\infty,x)-2\frac{W'\left(u_\infty(x)
\right)}{\left|\nabla u_\infty(x) \right|^2}\nabla u_\infty(x)\cdot
\nabla P(u_\infty,x)\geq 0,\ \ x\in B_r(x_0)\cap \bar{\Omega}.
\]
Keep in mind that
\[
P(u_\infty,x)\leq P_0,\ x\in \bar{\Omega}\ \ \textrm{and}\ \
P(u_\infty,x_0)=P_0.
\]
Two cases can occur:
\begin{itemize}
  \item If $x_0\in \Omega$, it follows at once from the
strong maximum principle that
\[
P(u_\infty,x)=P_0,\ \ x\in B_r(x_0)\cap \bar{\Omega},
\]
namely $B_r(x_0)\cap \bar{\Omega}\subset \mathcal{U}$;

  \item If $x_0\in \partial \Omega$, by (\ref{eqcarboumatano}) and Hopf's
boundary point lemma, we are led again to the same conclusion.
\end{itemize}
Thus, we have shown that relation (\ref{eqcarbouopen}) holds.

 By
(\ref{eqcarbouclosed}), (\ref{eqcarbouopen}), and the connectedness
of $\bar{\Omega}$, we conclude that
\[
\mathcal{U}=\bar{\Omega}.
\]
In other words, we have arrived at
\begin{equation}\label{eqcarbouinf}
\left|\nabla u_\infty(x) \right|^2=P_0+2W\left(u_\infty(x)
\right)\geq P_0>0,\ \ x\in \bar{\Omega}.
\end{equation}
We will show that this comes in contradiction with the fact that
$u_\infty$ is bounded. We fix a $Q\in \Omega$ and consider the
gradient flow
\[
\left\{\begin{array}{l}
         \gamma'(t)=\nabla u_\infty\left(\gamma(t)\right), \\
          \\
         \gamma(0)=Q.
       \end{array}
 \right.
\]
We note that $\gamma$ is globally defined since $\nabla u_\infty\in
L^\infty(\Omega)$ and $\gamma$ cannot hit $\partial \Omega$ due to
$\frac{\partial u_\infty}{\partial \nu}=0$ on $\partial \Omega$. We
have
\[
\frac{d}{dt}\left[u_\infty\left(\gamma(t) \right) \right]=\nabla
u_\infty\left(\gamma(t) \right) \cdot \gamma'(t)=\left|\nabla
u_\infty\left(\gamma(t) \right)
\right|^2\stackrel{(\ref{eqcarbouinf})}{\geq} P_0>0
\]
Thus, we get
\[
u_\infty\left(\gamma(t) \right)\geq u_\infty(Q)+P_0 t,\ \ t\geq 0,
\]
which implies that $u_\infty$ is unbounded; a contradiction.

The proof of the proposition is complete.
\end{proof}
\subsection{The symmetry result}
Our main result in this section is the following:
\begin{pro}\label{procarbouMine}
Let $u$ be a nonconstant bounded  solution to (\ref{eqcarbouEq})
such that $u\to 1$ as $x_1\to -\infty$ uniformly for $x'\in
\bar{\omega}$. If $\omega$ is smooth and convex, then $u$ is one
dimensional.
\end{pro}
\begin{proof}
Given $R>0$, let $u_R$ be as in Lemma \ref{lem1}, with $n=1$ and
$W'(t)=t^3-t$ (i.e. $\mu=1$). By the uniform in $\bar{\omega}$
asymptotic behavior of $u$ as $x_1\to -\infty$, we infer that there
exists a large $M>R$ such that $u(x)>u_R(0)$ if $x\in
\bar{\Omega}\cap \{x_1\leq -M+R\}$, i.e.
\begin{equation}\label{eqyukoM}
u>u_R(x_1+M)\ \ \textrm{on}\ \bar{\Omega}\cap \left\{|x_1+M|\leq R
\right\}.
\end{equation}
Let
\begin{equation}\label{eqchencomp}
\underline{u}_{R,Q_1}(x_1,x')=\left\{\begin{array}{ll}
                                       u_R(x_1+Q_1), & |x_1+Q_1|<R,\ x'\in \omega, \\
                                         &   \\
                                        0, & \textrm{otherwise}.
                                     \end{array}
 \right.
\end{equation}
Note that
\[
\frac{\partial \underline{u}_{R,Q_1}}{\partial \nu}=0\ \
\textrm{on}\ \partial \Omega \ \textrm{if}\ x_1+Q_1\neq \pm R.
\]

We claim that $u$ vanishes at some point on $\mathbb{R}\times
\bar{\omega}$. Indeed, if not, we can move $Q_1$ in $\mathbb{R}$ to
find that $u>u_R(0)$ in $\bar{\Omega}$ for all $R>0$, by the sliding
method (note also that $\underline{u}_{R,Q_1}$ and $u$ cannot touch
on $x_1+Q_1=\pm R$). In view of the obvious analog of
(\ref{eqestimRect}), this implies that $u\equiv 1$ which cannot
happen since $u$ is assumed to be nonconstant.

Now, by virtue of the uniform asymptotic behavior of $u$ as $x_1\to
-\infty$, we may assume without loss of generality that
\begin{equation}\label{eqyukoP}
u>0\ \textrm{if}\ x_1<0\ \textrm{and}\ u(P)=0\ \textrm{at\ some}\
P=(0,P')\ \textrm{with}\ P'\in \bar{\omega},
\end{equation}
(because (\ref{eqcarbouEq}) is invariant with respect to
translations in the $x_1$-direction). So, in view of
(\ref{eqyukoM}), we can slide $\underline{u}_{R,Q_1}$ along the
$x_1$-axis (decreasing $Q_1$), staying below the graph of $u$, until
we reach
\[
u(x)>\underline{u}_{R,R}(x),\ \ x\in \{x_1<0\}\times \bar{\omega}.
\]
In particular, recalling (\ref{eqyukoP}), we obtain that
\begin{equation}\label{eqyukoR}
u(x_1,x')>u_R(x_1+R),\ \ -R<x_1<0,\ x'\in \bar{\omega}, \
\textrm{and} \ u(0,P')=0=u_R(R),\ \textrm{where}\  P'\in
\bar{\omega}.
\end{equation}
 As in Proposition \ref{prohalfFlat} (see in particular (\ref{eqslice})), via the obvious analog of (\ref{eqestimRect}) (loosely speaking, letting $R\to
\infty$ in (\ref{eqyukoR})), we obtain that
\begin{equation}\label{eqyukogeq}
u(x)\geq \textbf{U}(-x_1)\ \textrm{on}\ \{x_1\leq 0\}\times
\bar{\omega},\ \textrm{and}\ u=\textbf{U}=0\ \textrm{at}\ P=(0,P'),
\end{equation}
where $\textbf{U}$ as in (\ref{eqU}).

 Two cases can occur:
\begin{itemize}
  \item If $P'\in \omega$, by Hopf's boundary point lemma (applied to the equation for $u-\textbf{U}(-x_1)$), we deduce
  that either
  \[
u_{x_1}<-\textbf{U}'(0)=-\sqrt{2W(0)}\ \ \textrm{at}\ P=(0,P')\in
\Omega,
  \]
or $u\equiv \textbf{U}(-x_1)$ on $\{x_1\leq 0\}\times \bar{\omega}$.
Since $\omega$ is convex and $u$ is bounded, as in Proposition
\ref{prohalfFlat}, the former scenario cannot happen by virtue of
the gradient bound in Proposition \ref{procarbouModica} (this is the
first time in the proof that we used the convexity of $\omega$). We
therefore must have that $u(x)=\textbf{U}(-x_1)$ on $\{x_1\leq 0
\}\times \omega$ and, by the unique continuation principle
\cite{hormander} (applied to the equation for $u-\textbf{U}(-x_1)$),
we conclude that $u\equiv \textbf{U}(-x_1)$ in $\Omega$, as desired.

  \item If $P'\in \partial \omega$, from (\ref{eqyukogeq})
  and the gradient bound of Proposition \ref{procarbouModica},  we obtain that
  \begin{equation}\label{eqyukocontra}
u_{x_1}(0,P')=-\sqrt{2W(0)}.
  \end{equation}
  Actually,  by the strong maximum principle and
Hopf's boundary point lemma, unless $u\equiv \textbf{U}(-x_1)$,
there is strict inequality in (\ref{eqyukogeq}) at points in
$\{x_1<0\}\times\bar{\omega}$. In the latter case, we would like to
employ Hopf's boundary point lemma to get $
u_{x_1}(0,P')<-\sqrt{2W(0)},
 $
  which contradicts (\ref{eqyukocontra}).
  However, this time we cannot fit a ball in $\{x_1<0 \}\times
  \omega$ which is tangent to $P$. Nevertheless, with a little care,
  we can adapt the standard proof of Hopf's boundary point lemma to
  cover the situation at hand, where the point is on a corner of the
 boundary of the domain $\{x_1<0 \}\times\omega$. Indeed, let
 \begin{equation}\label{eqyukophi}
\varphi=\textbf{U}(-x_1)-u.
 \end{equation} We have
 \[
\Delta \varphi-c(x)\varphi=0,\ \varphi<0\   \textrm{in}\
\{x_1<0\}\times \omega,
 \]
 for some bounded function $c$, say $|c(x)|<d$, and $\varphi(0,P')=0$.
 For $a>0$ to be determined, let
 \[
v(x)=v(x_1,x')= e^{-a(x_1+1)}-e^{-a}>0,\ \ x_1\in (-1,0),\ x'\in
\omega.
 \]
 We can choose $a>0$ sufficiently large ($a>\sqrt{d}$) so that
 \begin{equation}\label{eqyukoserrin}
\Delta v-dv>0\  \ \textrm{on}\ [-1,0]\times \bar{\omega}.
 \end{equation}
 Now, let
 \[
\tilde{v}=\frac{\ell}{2v(-1)}v<0,\ \ \textrm{where} \
\ell=\max_{x_1=-1}\varphi<0.
 \]
 It follows that $\tilde{v}$ satisfies \[\Delta \tilde{v}-d\tilde{v}<0\  \ \textrm{on}\ [-1,0]\times \bar{\omega},
 \ \frac{\partial \tilde{v}}{\partial \nu}=0\ \textrm{on}\  [-1,0]\times \partial \omega, \]
    \begin{equation}\label{eqyukov'}\tilde{v}=\frac{\ell}{2}<0\
 \textrm{on}\ x_1=-1;\ \tilde{v}=0\ \textrm{and}\ \tilde{v}_{x_1}>0\ \textrm{on}\
 x_1=0.\end{equation}
 We claim that
 \begin{equation}\label{eqyukoclaim}
\varphi-\tilde{v}\leq 0\ \ \textrm{on}\ [-1,0]\times \bar{\omega}.
 \end{equation}
 We have that
 \[
\varphi-\tilde{v}\leq 0\ \ \textrm{on}\ x_1=-1\ \textrm{and}\ x_1=0,
\ \textrm{and}\ \frac{\partial (\varphi-\tilde{v})}{\partial \nu}=0\
\textrm{on}\ [-1,0]\times \partial \omega.
 \]
 Suppose that (\ref{eqyukoclaim}) does not hold, namely that the maximum of $\varphi-\tilde{v}$ over $[-1,0]\times \bar{\omega}$ is positive and is achieved at some
 $x_0\in (-1,0)\times \bar{\omega}$.
 Note that
 \[
\Delta
(\varphi-\tilde{v})>c(x)\varphi-d\tilde{v}>\left(c(x)-d\right)\varphi>0\
\ \textrm{at}\ x=x_0,
 \]
 (we have silently used that $\varphi, \tilde{v}\in C^2(\bar{\Omega})$ in case $x_0\in \partial
 \Omega$).
 By the maximum principle, we deduce that $x_0$ cannot be in $(-1,0)\times \omega$. Moreover, by the usual Hopf's boundary point
 lemma, the point $x_0$ can neither be in $(-1,0)\times\partial
 \omega$. We have thus been led to a contradiction, which means that relation
 (\ref{eqyukoclaim}) holds true. It follows in particular that the restriction of
 $\varphi-\tilde{v}$ on the line $\{-1\leq x_1\leq 0\}\times \{P'\}$ attains its
 maximum value at $x_1=0$, which implies, via (\ref{eqyukophi}), that
 \[
u_{x_1}(0,P')\leq-\textbf{U}'(0)-\tilde{v}_{x_1}(0,P')\stackrel{(\ref{eqyukov'})}{<}-\textbf{U}'(0)=-\sqrt{2W(0)}.
 \]
 Recalling that $u(0,P')=0$, the above relation contradicts (\ref{eqyukocontra}).
 We conclude again that $u\equiv \textbf{U}(-x_1)$, as desired.
\end{itemize}

 The proof of the proposition  is complete.
\end{proof}

\begin{rem}
In analogy to Theorem \ref{thmrigidity}, we have the following:

Let $u$ be a bounded solution to (\ref{eqcarbouEq}) such that $u\to
1$ as $x_1\to \pm \infty$ uniformly for $x'\in \bar{\omega}$. If
$\omega$ is smooth and convex, then $u\equiv 1$.

\end{rem}

\appendix
\section{Some useful ``comparison'' lemmas of the calculus of variations}\label{secappenda}
The following is essentially Lemma 2.1 in \cite{fuscoTrans}.
\begin{lem}\label{lemfusco}
Let $\mathcal{O}\subset \mathbb{R}^n$ be an open set and let $v\in
W^{1,2}(\mathcal{O})$. Define $\tilde{v}:\mathcal{O}\to \mathbb{R}$
as
\[
\tilde{v}(x)=\left\{\begin{array}{ll}
                      v(x) & \textrm{if}\ v(x)\in [0,\mu], \\
                       &  \\
                      \mu & \textrm{if}\ v(x)\in (-\infty,-\mu)\cup (\mu,\infty),  \\
                       &  \\
                      -v(x) & \textrm{if}\ v(x)\in (-\mu,0).
                    \end{array}
 \right.
\]
Then $\tilde{v}\in W^{1,2}(\mathcal{O})$ and, if $W$ is $C^2$ and
satisfies \textbf{(a')}, we have
\[
\int_{\mathcal{O}}^{}\left\{\frac{1}{2}|\nabla
\tilde{v}|^2+W(\tilde{v}) \right\}dx \leq
\int_{\mathcal{O}}^{}\left\{\frac{1}{2}|\nabla v|^2+W(v) \right\}dx.
\]
\end{lem}
\begin{proof}(Sketch)
Firstly, note that $\tilde{v}=G(v),\ x\in \mathcal{O}$, for some
Lipschitz (piecewise linear) function $G:\mathbb{R}\to \mathbb{R}$.
Thus, $\tilde{v}\in W^{1,2}(\mathcal{O})$, see for instance
\cite{evansGariepy}. Then, to finish, note that
\begin{equation}\label{eqkinder}
|\nabla \tilde{v}|\leq |\nabla v|\ \ \textrm{and, \ thanks to
\textbf{(a')}},\ \ W(\tilde{v})\leq W(v)\ \ \textrm{a.e.\ in}\
\mathcal{O},
\end{equation}
(the former inequality may be proven as in page 93 in
\cite{kinderler}).
\end{proof}

The following is an extension of Lemma \ref{lemfusco}, and is
motivated from  \cite{alikakosReplace,alikakosPeriodic} (see also
\cite{katzourakis} for an extension). Our proof follows
\cite{sourdisDancer}.
\begin{lem}\label{lemalikakos}
Let $\Omega \subset \mathbb{R}^n$, $n\geq 1$, be a bounded domain
with Lipschitz boundary, and $W:\mathbb{R}\to \mathbb{R}$  be a
$C^2$  potential  such that conditions \textbf{(a')} and
(\ref{eqWmonotone}) hold. Further, let $\mathcal{A}\subset\Omega$ be
a  bounded domain with  Lipschitz boundary such that
$\bar{\mathcal{A}}\subset\Omega$. Moreover, assume that
\begin{itemize}
  \item $u\in W^{1,2}(\Omega)$, $0\leq u\leq \mu$ a.e. in $\Omega$
\\
  \item $\mu-u\leq \eta$ a.e. on $\partial \mathcal{A}$, in the sense of Sobolev traces (see \cite{evansGariepy}), for
  some $\eta\in \left(0,\frac{d}{2} \right)$.
\end{itemize}
Then, there exists $\tilde{u}\in W^{1,2}(\Omega)$ such that
\begin{equation}\label{eqrelation}\left\{\begin{array}{l}
    \tilde{u}(x)=u(x),\ \ x\in \Omega\backslash \mathcal{A}, \\
      \\
    \mu-\eta\leq \tilde{u}(x)\leq \mu,\ \ x\in \mathcal{A}, \\
      \\
    \int_{\Omega}^{}\left\{\frac{1}{2}|\nabla \tilde{u}|^2+W(\tilde{u})
\right\}dx\leq \int_{\Omega}^{}\left\{\frac{1}{2}|\nabla u|^2+W(u)
\right\}dx.
  \end{array}\right.
\end{equation}

If condition (\ref{eqWmonotone}) holds with strict inequality, and
there exists a set $\mathcal{B}\subset \mathcal{A}$ of nonzero
measure such that
\[
u<\mu-\eta\ \ \textrm{a.e.\ on}\ \ \mathcal{B},
\]
then the last relation in (\ref{eqrelation}) holds with a strict
inequality.
\end{lem}
\begin{proof} (Sketch)
The first assertion of the lemma can be deduced similarly to Lemma
\ref{lemfusco}. Indeed, the desired function is
\begin{equation}\label{equtilda}
\tilde{u}(x)=\left\{\begin{array}{ll}
                                 \min\left\{\mu, \max\left\{u(x),\ 2\mu-2\eta-u(x) \right\}\right\},  & x\in \mathcal{A}, \\
                                   &   \\
                                  u(x), & x\in
                                  \Omega\backslash
                                  \mathcal{A}.
                                \end{array}
 \right.
\end{equation}
We point out that $\tilde{u}\in W^{1,2}(\mathcal{A})$ similarly to
Lemma \ref{lemfusco}, and $\tilde{u}\in W^{1,2}_0(\Omega)$ because
$\mathcal{A}$ has Lipschitz boundary and $\tilde{u}=u$ on $\partial
\mathcal{A}$ in the sense of Sobolev traces (see again
\cite{evansGariepy}). Note that if $\mu-2\eta \leq u(x)\leq \mu$
then $\mu-d <u(x)\leq \tilde{u}(x)\leq \mu$, so relation
(\ref{eqWmonotone}) implies that $W\left(\tilde{u}(x) \right)\leq
W\left(u(x) \right)$. Furthermore, if $0 \leq u(x)\leq \mu-2\eta$
then $\tilde{u}(x)=\mu$ and $W\left(\tilde{u}(x) \right)=0\leq
W\left(u(x)\right)$. Also keep in mind the first relation in
(\ref{eqkinder}).

The second assertion can be shown with a little more care. Replacing
$u$ by the minimizer of the corresponding energy functional
$J(\cdot; \mathcal{A})$ (recall (\ref{eqenergy})) among functions
$v\in W^{1,2}(\mathcal{A})$ such that $v-u\in
W^{1,2}_0(\mathcal{A})$, we may assume that $u$ is a smooth solution
of (\ref{eqEqnobdry}) in $\mathcal{A}$.
 Firstly,
we consider the case where
\[
\mu-2\eta \leq u(x)\leq \mu\ \  \textrm{on}\ \bar{\mathcal{A}}.
\]
In that case, we have that $\mu-d <u< \tilde{u}\leq \mu$ on
$\mathcal{B}$. In turn, from the assumption that the inequality in
(\ref{eqWmonotone}) is strict, we obtain that $W\left(\tilde{u}
\right)<W\left(u \right)$ on $\mathcal{B}$. Since the set
$\mathcal{B}$ has positive measure, taking into account our previous
discussion for the first assertion, we arrive at
\begin{equation}\label{eqreplace}
\int_{\Omega}^{}W(\tilde{u})dx<\int_{\Omega}^{}W(u)dx.
\end{equation}
Hence, the second assertion holds in this case. On the other side,
if
\[
0 \leq u(x_0)< \mu-2\eta\ \  \textrm{for \ some}\ x_0 \in
{\mathcal{A}},
\]
then $0\leq u\leq \mu-2\eta \leq \tilde{u}=\mu$ in an open
neighborhood of $x_0$. In this neighborhood, via \textbf{(a')}, we
have that $W(u)\geq \min_{t\in[0,\mu-2\eta]}W(t)>0$ while
$W(\tilde{u})=0$. It follows that relation (\ref{eqreplace}) holds
in this case as well. Keeping in mind the first relation in
(\ref{eqkinder}), we conclude that the second assertion of the lemma
holds.

The sketch of proof of the lemma is complete.
\end{proof}

The following is Lemma 2.3 in \cite{danceryanCVPDE}, which is
reproduced in Lemma 1 in \cite{kurata} and Lemma 2.1 in
\cite{yanedinburg}, see also Theorem 1.4 in \cite{friedman} and
Lemma 3.1 in \cite{jerison}.
\begin{lem}\label{lemdancer} Let $\mathcal{D}$ be a bounded domain in $\mathbb{R}^n$ with
smooth boundary. Let $g_1(x, t), g_2(x, t)$ be locally Lipschitz
functions with respect to $t$, measurable functions with respect to
$x$, and for any bounded interval $I$ there exists a constant $C$
such that $\sup_{x\in \mathcal{D},t\in I}  |g_i(x, t)|\leq C$, $i=1,
2$, holds. Let
\[
G_i(x,t)=\int_{0}^{t}g_i(x,s)ds,\ i=1,2.
\]
 For $\eta_i \in W^{1,2}(\mathcal{D}), \ i=1, 2,$ consider
the minimization problem:
\[
\inf\left\{J_i(u;\mathcal{D})\ |\ u-\eta_i\in W^{1,2}_0(\mathcal{D})
\right\},\ \ \textrm{where}\ \
J_i(u;\mathcal{D})=\int_{\mathcal{D}}^{}\left\{\frac{1}{2}|\nabla
u|^2-G_i(x,u)\right\}dx.
\]
Let $u_i\in W^{1,2}(\mathcal{D}),\ i=1,2,$ be minimizers to the
minimization problems above. Assume that there exist constants $m<M$
such that
\begin{itemize}
  \item $m\leq u_i(x)\leq M$ a.e.  for    $i=1,2,\ x\in \mathcal{D}$,
  \item $g_1(x,t)\geq g_2(x,t)$ a.e. for  $x\in \mathcal{D},\ t\in
  [m,M]$,
  \item $M\geq\eta_1(x)\geq \eta_2(x)\geq m$ a.e. for  $x\in \mathcal{D}$.
\end{itemize}
Suppose further that $\eta_i\in W^{2,p}(\mathcal{D})$ for any $p>1$,
and that they are \emph{not identically equal} on $\partial
\mathcal{D}$. Then, we have
\[
u_1(x)\geq u_2(x),\ \ \ x\in\mathcal{D}.
\]
\end{lem}
\section{A Liouville-type theorem}\label{appenLiouville}
Assume that $f:\mathbb{R}\to \mathbb{R}$ is continuous, and
\[
\left\{
\begin{array}{l}
  f(0)=0, \\
    \\
  f(t)>0,\ t>0, \\
   \\
  f \ \textrm{is\ non--decreasing\ and\ convex\ on}\ [0,\infty), \\
   \\
\large{\int}_{t_0}^{\infty}\left[\int_{t_0}^{t}f(s)ds
\right]^{-\frac{1}{2}}dt<\infty
 \ \ \forall\ t_0>0.
\end{array}
\right.
\]
In the mathematical literature, the above integral condition is
known as Keller–-Osserman condition, see \cite{farinaHB},
\cite{keller} and \cite{osserman}. These conditions are clearly
satisfied  for  \begin{equation}\label{eqfbrezis}f(t)=t|t|^{p-1}\ \
\textrm{with}\  p>1.
\end{equation}
 The following is Theorem 4.7 in the review
article \cite{farinaHB}. As we have already discussed at the end of
Remark \ref{remDoublingOptimal}, it was originally proven in
\cite{brezisLiouville} for the special case of the power
nonlinearity (\ref{eqfbrezis}).

\begin{thm}\label{thmFarina}
Let $f$ satisfy the above properties.
\begin{description}
  \item[(i)] Suppose $u\in L^1_{loc}(\mathbb{R}^n)$ is such that $f(u)\in
  L^1_{loc}(\mathbb{R}^n)$ and
  \[
-\Delta u+f(u)\leq 0\ \ \textrm{in}\ \ \mathcal{D}'(\mathbb{R}^n)\
\textrm{(distributionally)}.
  \]
  Then $u\leq 0$ a.e. on $\mathbb{R}^n$.
  \item[(ii)]Assume also that $f$ is an odd function. Suppose  $u\in L^1_{loc}(\mathbb{R}^n)$ is such that $f(u)\in
  L^1_{loc}(\mathbb{R}^n)$ and
  \[
-\Delta u+f(u)= 0\ \ \textrm{in}\ \ \mathcal{D}'(\mathbb{R}^n).
  \]
  Then $u= 0$ a.e. on $\mathbb{R}^n$.
\end{description}
\end{thm}

\section{A doubling lemma}\label{secAppenDoubling}
The following is a very useful result from \cite{pqs}.
\begin{lem}\label{lemDoubling}
Let $(X,d)$ be a complete metric space, $\Gamma \subset X$,
$\Gamma\neq X$, and $\gamma:X\backslash \Gamma\to (0,\infty)$.
Assume that $\gamma$ is bounded on all compact subsets of
$X\backslash\Gamma$. Given $k>0$, let $y\in X\backslash \Gamma$ be
such that
\[
\gamma(y)\textrm{dist}(y,\Gamma)>2k.
\]
Then, there exists $x\in X\backslash \Gamma$ such that
\begin{itemize}
  \item $\gamma(x)\textrm{dist}(x,\Gamma)>2k$,
\\
  \item $\gamma(x)\geq \gamma(y)$,
\\
  \item $2\gamma(x)\geq \gamma(z)\ \ \ \forall \ z\in
  B_{\frac{k}{\gamma(x)}}$.
\end{itemize}
\end{lem}
We remark that this doubling lemma is proven similarly as Baire's
category theorem.

\section{Some remarks on  equivariant entire solutions to a class of  elliptic systems of the form $\Delta u=W_u(u)$, $u:\mathbb{R}^n\to \mathbb{R}^n$}\label{appenAlik}
 In this appendix, motivated from Remark \ref{remalikakossimple}, we
indicate how to  simplify some arguments of the recent paper
\cite{alikakosNewproof}, in the case of the equations that are
considered there as representative examples.

 We will use
exactly the same notation of \cite{alikakosNewproof}. This appendix
should be read with a copy of \cite{alikakosNewproof} at hand.

 In \cite{alikakosNewproof}, the author provides a simpler
proof of the recent result in \cite{alikakosARMA}, concerning the
existence of  equivariant entire solutions to a class of semilinear
elliptic systems of the form $\Delta u= W_u(u)$, $u:\mathbb{R}^n\to
\mathbb{R}^n$ (the same approach applies for the case
$u:\mathbb{R}^n\to \mathbb{R}^m$). Besides of assuming that $W_u$ is
also equivariant, the main assumption in the latter papers, which
was subsequently  removed completely in \cite{fuscoPreprint} (recall
also our Theorem \ref{thmmine} for more general results in the
scalar case), is that of ``$Q$-monotonicity'' (this essentially
corresponds to assumption \textbf{(b)} from our introduction). In
all the examples of $W$'s found in
\cite{alikakosARMA,alikakosNewproof}, for which this assumption
could be verified, the function $Q$ was plainly
\[
Q(u)=|u-a_1|,
\]
(this was also the case for the example in
\cite{alikakosSmyrnelis}). In the sequel, except from Remark
\ref{remslide}, we will assume this choice of $Q$.

We may assume that
\begin{equation}\label{eqapal}
W(u)\geq c^2 |u-a_1|^2=c^2Q^2(u),\ \ u\in D\cap B_M,
\end{equation}
because $a_1\in \mathbb{R}^n$ is the only minimum of $W\in C^2$ in
$D$, $W>0$ in $D\backslash \{a_1\}$, and $a_1$ is non-degenerate.
Let $x_R\in D\cap B_{4R}$ be any point as in the beginning of
Section 6 in \cite{alikakosNewproof} (namely with
$B_{3R}(x_R)\subset D$). From Lemma 4.1 in \cite{alikakosNewproof},
via the above relation (we have $|u_R|\leq M$), we obtain that
\[
\Delta Q(u_R)\geq 0\ \textrm{in}\ D\  (\textrm{weakly})\
\textrm{and}\ \int_{B_{2R}(x_R)}^{}Q^2\left(u_R(x)\right)dx\leq
CR^{n-1},
\]
for some constant $C>0$ that is independent of $R$. Now, as in
Remark \ref{remalikakossimple} herein, we infer that
\[
\sup_{B_{R}(x_R)}Q(u_R)\leq C
R^{-n}\int_{B_{2R}(x_R)}^{}Q\left(u_R\right)dx\leq C
R^{-n}R^{\frac{n}{2}}R^{\frac{n-1}{2}}=CR^{-\frac{1}{2}}\to 0\ \
\textrm{as} \ R\to \infty,
\]
($C$ is again independent of $R$). Let us mention that Section 6 in
\cite{alikakosNewproof} was devoted to proving  a similar relation
(in fact, weaker but without making use of (\ref{eqapal}))
 using De Giorgi's oscillation
lemma and an iteration scheme.

\begin{rem}\label{remslide}
It seems likely that the use of the function $\vartheta$ in
\cite{alikakosNewproof} can be avoided (as well as the covering
argument of \cite{fuscoTrans}) in order to extend the domain of
validity of the above bound. To support this, we note that the
function $Q(u_R)$ is a weak lower solution to a problem of the form
\begin{equation}\label{eqappenAlikaf}
\Delta u=f(u)=\left\{\begin{array}{ll}
                      c^2(a-u), & 0\leq u \leq a, \\
                        &  \\
                       0, &  u\geq a,
                     \end{array}
 \right.
\end{equation}
in $D$ (see also the relation between (45) and (46) in
\cite{alikakosBasicFacts}).
 As we observed in our introduction,  the
continuous patching of the radial comparison functions, analogously
to (\ref{eqfuscoTechnics}) (after reflecting them), together with
zero can form a weak
 upper solution to (\ref{eqappenAlikaf}) which may also be chosen to lie above $Q(u_R)$ on $B_R(x_R)$.
Then, one can extend the domain of validity of  the above estimate
by plainly  sliding around $x_R$ in $D$ (a fixed distance away from
the boundary), using a weak version of the sliding method (the point
is that the function $f$ is Lipschitz).
\end{rem}

 \textbf{Acknowledgment.} I would like to thank Prof. Farina for offering valuable comments on a previous version of this article,
and especially for bringing to my attention his paper
\cite{farinaRigidity} which led to Theorem \ref{thmrigidity}.
  I also would like to thank D. Antonopoulou for her insightful comments. The research leading to these results has
received funding from the European Union's Seventh Framework
Programme (FP7-REGPOT-2009-1) under grant agreement
$\textrm{n}^\textrm{o}$ 245749.

\end{document}